\documentclass[10pt]{amsart}

\usepackage{enumerate}
\usepackage{amssymb}
\usepackage{bm}
\usepackage{graphicx}
\usepackage[centertags]{amsmath}
\usepackage{amsfonts}
\usepackage{amsthm}
\usepackage{a4}
\usepackage{latexsym}
\usepackage{amsmath}
\usepackage{amsfonts}
\usepackage{amssymb}
\usepackage{amscd}
\usepackage{amsthm}
\usepackage{layout}
\usepackage{color}
\usepackage{mathrsfs}

\usepackage[normalem]{ulem}
\usepackage{comment}

\theoremstyle{definition}
\newtheorem{pr}{Problem}
\newtheorem{thm}{Theorem}[section]
\newtheorem{cor}[thm]{Corollary}
\newtheorem{prop}[thm]{Proposition}
\newtheorem{lem}[thm]{Lemma}

\newtheorem*{thm*}{Theorem}
\newtheorem*{notn*}{Notation}

\theoremstyle{definition}
\newtheorem{dfn}[thm]{Definition}

\newtheorem{exmp}[thm]{Example}

\newtheorem{notn}[thm]{Notation}

\theoremstyle{definition}
\newtheorem{rem}[thm]{Remark}
\newtheorem{rems}[thm]{Remarks}


\newtheorem{thmA}{Theorem}

\newcommand{\R}{\mathbb{R}}
\newcommand{\N}{\mathbb{N}}

\newcommand{\dt}{\mathscr{D}}
\newcommand{\supp}{\textnormal{supp}}

\newcommand{\spn}{\textnormal{span}}
\newcommand{\cspn}{\overline{\textnormal{span}}}
\renewcommand{\-}{\textnormal{-}}
\DeclareMathOperator{\sgn}{sgn}
\DeclareMathOperator{\co}{co}

\allowdisplaybreaks 

\begin{document}

\title[Joint spreading models and uniform approximation]{Joint spreading models and uniform approximation of bounded operators}
\author[S. A. Argyros]{S. A. Argyros}
\address{National Technical University of Athens, Faculty of Applied Sciences,
	Department of Mathematics, Zografou Campus, 157 80, Athens, Greece}
\email{sargyros@math.ntua.gr}\email{ale.grgu@gmail.com}
\author[A. Georgiou]{A. Georgiou}
\author[A.-R. Lagos]{A.-R. Lagos}
\address{National Technical University of Athens, School of Electrical and Computer Engineering, Zografou Campus, 157 80 Athens, Greece}
\email{lagosth993@gmail.com}
\author[P. Motakis]{Pavlos Motakis}
\address{Department of Mathematics, University of Illinois at Urbana-Champaign, Urbana, IL 61801, U.S.A.}
\email{pmotakis@illinois.edu}

\thanks{The fourth named author
  was  supported by the National Science Foundation under Grant Numbers
  DMS-1600600 and DMS-1912897.}

\maketitle

	\begin{abstract}
		We investigate the following property for Banach spaces. A Banach space $X$ satisfies the Uniform Approximation on Large Subspaces (UALS) if there exists $C>0$ with the following property: for any $A\in\mathcal{L}(X)$ and convex compact subset  $W$ of $\mathcal{L}(X)$ for which there exists $\varepsilon>0$ such that for every $x\in X$ there exists $B\in W$ with $\|A(x)-B(x)\|\le\varepsilon\|x\|$, there exists a  subspace $Y$ of $X$ of finite codimension and a $B\in W$ with $\|(A-B)|_Y\|_{\mathcal{L}(Y,X)}\leq C\varepsilon$. We prove that a class of separable Banach spaces including $\ell_p$, for $1\le p< \infty$,  and $C(K)$, for $K$ countable and compact, satisfy the UALS. On the other hand every $L_p[0,1]$, for $1\le p\le \infty$ and $p\neq2$, fails the property and the same holds for $C(K)$, where $K$ is an uncountable metrizable compact space. Our sufficient conditions for UALS are based on joint spreading models, a multidimensional extension of the classical concept of spreading model, introduced and studied in the present paper.
	\end{abstract}

	\section*{introduction} This paper is devoted to the study of the Uniform Approximation on Large Subspaces (UALS) for an infinite dimensional Banach space $X$. This concept concerns a special case of the following general question. Find conditions such that the $\varepsilon$-pointwise approximation of a function $f$ by the elements of a family of functions $W$ yields that there exists a $g\in W$ which uniformly $\varepsilon'$-approximates the function $f$. One of the best results in this frame is the well known consequence of Hahn-Banach theorem. If $X$ is a Banach space, it may be viewed as a subspace of $X^{**}$ through the natural embedding. If $x_0\in X$, $W$ is a closed convex subset of $X$ and $\varepsilon>0$ are such that for every $x^*\in X^*$ there exists $x\in W$ with $|x^*(x_0)-x^*(x)|\le\varepsilon\|x^*\|$, then we have that for every $\varepsilon'>\varepsilon$ there exists $y_0\in W$ so that $\|x_0-y_0\|\le\varepsilon'$. It is natural to ask how the previous is extended  to the space of bounded linear operators $\mathcal{L}(X)$. The UALS property is an attempt to provide an answer. Notice that in the definition of UALS there are two differences from the above result. The first one is that the set $W$ is norm compact and this is necessary since the  normalized operators of rank one $\varepsilon$-pointwise approximate the identity for every $\varepsilon>0$.  The second one is that we expect the uniform approximation to happen on a finite codimensional subspace. This is also necessary, since as Bill Johnson pointed out, for every $X$ with $\dim X\ge2$ there exist $C>0$, $A\in\mathcal{L}(X)$ and a convex compact $W\subset\mathcal{L}(X)$ such that, for every $x$ in the unit ball of $X$, there is a $B$ in $W$ with $\|A(x)-B(x)\|=0$, whereas the norm distance of $A$ from $W$ is greater than $C$. There are two classes of Banach spaces satisfying the UALS. The first are spaces with the scalar-plus-compact property \cite{AH}, \cite{A} and the second one includes spaces with strong asymptotic homogeneity. The latter concerns the uniform uniqueness of $l$-joint spreading models, an extension of the classical spreading models \cite{BS}.

The paper is organized in five sections. In the first section we introduce the notion of plegma spreading sequences. These are finite collections of Schauder basic sequences in a Banach space that interact with one another in a spreading way when indexed by plegma families, a notion which first appeared in \cite{AKT}.

The second section is motivated by the definition of plegma spreading sequences and concerns the problems of whether or not finite collections of Schauder (unconditional) basic sequences contain subsequences that form, under a suitable order, a common Schauder (unconditional) basic sequence. For Schauder basic sequences we provide a complete characterization given in the following.

    \begin{thmA}
        Let $(x^1_n)_n,\ldots,(x^l_n)_n$ be seminormalized sequences in a Banach space $X$ such that each one is either weakly null, equivalent to the basis of $\ell_1$, or non-trivial weak-Cauchy. Let $I\subset\{1,\ldots,l\}$ be such that $(x^i_n)_n$ is a non-trivial weak-Cauchy sequence with $w^*\-\lim x^i_n=x^{**}_i$ for every $i\in I$ and set $F=\spn\{ x^{**}_i\}_{i\in I}$. Then there exist $M_1,\ldots,M_l$ infinite subsets of $\N$ such that $\cup_{i=1}^l\{x^i_{n} \}_{n\in M_i}$ is a Schauder basic sequence, under a suitable enumeration, if and only if $X\cap F=\{0\}$.
    \end{thmA}

For unconditional sequences the following holds.

    \begin{thmA}
        Let $(e^i_{n})_{i=1,n\in\N}^l$ be a plegma spreading sequence such that each $(e^i_n)_n$ is unconditional. Then $(e^i_{n})_{i=1,n\in\N}^l$ is also an unconditional sequence.
    \end{thmA}	

    We also provide a variant of the classical B. Maurey - H. P. Rosenthal example \cite{MR} of two unconditional sequences $(e^1_n)_n,(e^2_n)_n$ in a space $X$ such that, for every $M,L$ infinite subsets of $\N$, the sequence $(e_n^1)_{n\in M}\cup(e_n^2)_{n\in L}$ is not unconditional. This shows that the assumption of a plegma spreading sequence in the above theorem is necessary. Further, it is well known that the space generated by two unconditional sequences is not necessarily unconditionally saturated.

    In section three we define  joint spreading models, which as we have mentioned already, are a multidimensional extension of the classical Brunel-Sucheston spreading models. We also present some of their basic properties.

    The fourth section concerns spaces that admit uniformly unique joint spreading models with respect to certain families of Schauder basic sequences. Examples of such spaces are $\ell_p(\Gamma)$, for $1\le p<\infty$, $c_0(\Gamma)$ and, as we show, all Asymptotic $\ell_p$ spaces in the sense of \cite{MMT}. We also prove that the James Tree space admits a uniformly unique $l$-joint spreading model with respect to the family of all normalized weakly null Schauder basic sequences in $JT$. Each $l$-joint spreading model generated by a sequence from this family is $\sqrt{2}$-equivalent to the unit vector basis of $\ell_2$ and this is the best constant \cite{HB}, \cite{Be}. Our proof is a variant of the well known result due to I. Ameniya and  T. Ito \cite{AI} that every normalized weakly null sequence in $JT$ has a subsequence equivalent to the basis of $\ell_2$.

    The fifth section is devoted to the study of spaces satisfying the UALS and to classical spaces where this property fails. In the first part we study the property for spaces with very few operators, namely spaces with the scalar-plus-compact property \cite{AH}, \cite{A}. We prove the following.
     \begin{thmA}
	 	Every Banach space with the scalar-plus-compact property satisfies the  UALS.
	 \end{thmA}

     The basic result for UALS concerns spaces which admit uniformly unique joint spreading models with respect to families of Schauder basic sequences that have certain stability properties. For this we first introduce the class of difference-including families (see Definition \ref{difference including}) and we prove the following.
     	 \begin{thmA}
	 \label{main theorem intro}
	 	Let $X$ be a Banach space and assume that for every separable subspace $Z$ of $X$ we have a difference-including collection $\mathscr{F}_Z$ of normalized Schauder basic sequences  in $Z$.  If there exists a uniform $K\geq 1$ such that each such $Z$ admits a $K$-uniformly unique $l$-joint spreading model with respect to $\mathscr{F}_Z$, then $X$ satisfies the UALS property.
	 \end{thmA}
    A key ingedient of the proof is Kakutani's Fixed Point theorem for multivalued mappings \cite{BK}, \cite{Kakutani}. This argument has appeared in a work of W. T. Gowers and B. Maurey, which is related to the theorem (see Lemma 9 of \cite{GM}), and was the motivation for defining the UALS property. As a consequence of the above theorem, the following spaces and all of their subspaces satisfy the UALS. The space $\ell_p(\Gamma)$, for $1\le p<\infty$, $c_0(\Gamma)$, the James Tree space and all Asymptotic $\ell_p$ spaces for $1\le p\le\infty$.  The UALS property behaves quite well in duality. In particular the following hold.

     	  \begin{thmA}
	  Let $X$ be a reflexive Banach space with an FDD. Then $X$ satisfies the UALS if and only if $X^*$ does.
	  \end{thmA}

	 \begin{thmA}
	 \label{dual theorem intro}
	 Let $X$ be a Banach space with an FDD. Assume that there exist a uniform constant $C>0$ and, for every separable subspace $Z$ of $X^*$,  a difference-including family $\mathscr{F}_Z$ of normalized sequences in $X^*$ such that $Z$ admits a $C$-uniformly unique $l$-joint spreading model with respect to $\mathscr{F}_Z$. Then $X$ satisfies the UALS property.
	 \end{thmA}
     As a consequence of the above, $\mathscr{L}_\infty$-spaces with separable dual and their quotients with an FDD satisfy the UALS. Thus the spaces $C(K)$, for $K$ countable compact, have the property. We also provide an example of a reflexive Banach space that admits a uniformly unique spreading model and fails the UALS property. This example shows that if a space admits a uniformly unique spreading model, this does not necessarily imply that it admits a uniformly unique $l$-joint spreading model for every $l\in\N$, and that the assumption in Theorems \ref{main theorem intro} and \ref{dual theorem intro} of a uniformly unique $l$-joint spreading model cannot be weakened by assuming a uniformly unique spreading model. Finally, we prove that the spaces $L_p[0,1]$, for $1\le p\le \infty$ and $p\neq2$, and $C(K)$, for any uncountable and metrizable compact space $K$, fail the UALS.

     \bigskip

     \textbf{\em Acknowledgement.} We express our thanks to I. Gasparis, W. B. Johnson and B. Sari for their comments and remarks that allowed us to improve the content of the paper.

	\section{Plegma Spreading Sequences}
	
	    We recall the notion of plegma families which first appeared in \cite{AKT} and were used to define higher order spreading models. Interestingly, they were used in a rather different way there and we slightly modify their definition. We shall refer to the notion from \cite{AKT} as strict plegma families. We use them to introduce the notion of plegma spreading sequences. These are finite collections of sequences that interact with one another in a spreading way when indexed by plegma families. We start with some notation we will use throughout the paper.

    \begin{notn*}
        By $\N=\{1,2,\ldots\}$ we denote the set of all positive integers. We will use capital letters as $L,M,N,\ldots$ (resp. lower case letters as $s,t,u,\ldots$) to denote infinite subsets (resp. finite subsets) of $\N$. For every infinite subset $L$ of $\N$, the notation $[L]^\infty$ (resp. $[L]^{<\infty}$) stands for the set of all infinite (resp. finite) subsets of $L$. For every $s\in[\N]^{<\infty}$, by $|s|$ we denote the cardinality of $s$. For $L\in[\N]^\infty$ and $k\in\N$, $[L]^k$ (resp. $[L]^{\le k}$) is the set of all $s\in[L]^{<\infty}$ with $|s|=k$ (resp. $|s|\le k$). For every $s,t\in[\N]^{<\infty}$, we write $s<t$ if either at least one of them is the empty set, or $\max s<\min t$. Also for  $\emptyset\neq s\in[\N]^\infty$ and $n\in\N$ we write $n<s$ if $n<\min s$.

        We shall identify strictly increasing sequences in $\N$ with their corresponding range, i.e. we view every strictly increasing sequence in $\N$ as a subset of $\N$ and conversely every subset of $\N$ as the sequence resulting from the increasing order of its elements. Thus, for an infinite subset $L=\{l_1<l_2<\ldots\}$ of $\N$ and $i\in\N$, we set $L_i=l_i$ and similarly, for a finite subset $s=\{n_1,\ldots,n_k\}$ of $\N$ and for $1\le i\le k$, we set $s(i)=n_i$.

        Given a Banach space $X$ with a Schauder basis $(e_n)_n$, then for every $x\in X$ with $x=\sum_na_ne_n$ we write $\supp(x)$ to denote the support of $x$, i.e.  $\supp(x)=\{n\in\N:a_n\neq0\}$. Generally, we follow \cite{LT} for standard notation and terminology concerning Banach space theory.
    \end{notn*}
	
	\begin{dfn}
		Let $M\in[\N]^\infty$ and  $\mathcal{F}$ be either $[M]^k$ for some $k\in\N$ or $[M]^\infty$.\linebreak A {\em plegma} (resp. {\em strict plegma}) family in $\mathcal{F}$ is a finite sequence $(s_i)_{i=1}^l$ in $\mathcal{F}$ satisfying the following properties.
		\begin{enumerate}
			\item[(i)] $s_{i_1}(j_1)<s_{i_2}(j_2)$ for every $1\le j_1<j_2\le k$ or $j_1<j_2\in\N$ and $1\le i_1,i_2\le l$.
			\item[(ii)] $s_{i_1}(j)\le s_{i_2}(j)$ $\big($resp. $s_{i_1}(j)< s_{i_2}(j)\big)$ for every $1\le i_1<i_2\le l$ and every $1\le j\le k$ or $j\in\N$.
		\end{enumerate}
	For each $l\in \N$, the set of all sequences $(s_i)^l_{i=1}$ which are plegma families in $\mathcal{F}$ will be denoted by $Plm_l(\mathcal{F})$ and that of the strict plegma ones by $S$-$Plm_l(\mathcal{F})$.
	\end{dfn}

	The following is a consequence of Ramsey's Theorem \cite{Ramsey}.

	\begin{thm}[\cite{AKT}]\label{plmramsey}
		Let $M$ be an infinite subset of $\N$ and $k,l\in\N$. Then for every finite partition $S$-$Plm_l([M]^k)=\cup_{i=1}^nP_i$, there exist $L\in[M]^\infty$ and $1\le i_0\le n$ such that $S$-$Plm_l([L]^k)\subset P_{i_0}$.
	\end{thm}

	\begin{dfn}\label{def plegma shifts}
	Let $\pi=\{1,\ldots,l\}\times\{1\ldots,k\}$, $s=(s_i)_{i=1}^l$ be a plegma family in $[\N]^k$ and $(e^i_{n})_{i=1,n\in\N}^l$ be a sequence in a linear space $E$.
		\begin{enumerate}
			\item[(i)] The {\em plegma shift} of $\pi$ with respect to the plegma family $s$ is the set\linebreak $s(\pi)=\{(i,s_i(j)):(i,j)\in\pi \}$ and for a subset $A$ of $\pi$, the {\em plegma shift} of $A$ with respect to $s$ is the set $s(A)=\{(i,s_i(j)):(i,j)\in A \}$.
			\item[(ii)]	Let $x\in E$ with $x=\sum_{(i,j)\in F}a_{ij}e^i_j$ and $F\subset\pi$. The {\em plegma shift} of $x$ with respect to $s$ is the vector $s(x)=\sum_{(i,j)\in F}a_{ij}e^i_{s_i(j)}$.			
		\end{enumerate}
	\end{dfn}

	Recall that a sequence $(e_n)_n$ in a seminormed space $E$ is called spreading if $\|\sum_{i=1}^na_ie_i\|=\|\sum_{i=1}^na_ie_{k_i}\|$ for every $n\in\N$, $k_1<\ldots<k_n$ and $a_1,\ldots,a_n\in\R$. Then, under Definition \ref{def plegma shifts}, we have the following reformulation: $(e_n)_n$ is spreading if $\|\sum_{i=1}^na_ie_i\|=\|s(\sum_{i=1}^na_ie_i)\|$ for every $n\in\N$, $a_1,\ldots,a_n\in\R$ and every plegma family $s\in Plm_1([\N]^n)$. Next we introduce the notion of plegma spreading sequences, which are an extension of the above.
	
	\begin{dfn}
		A sequence $(e^i_{n})_{i=1,n\in\N}^l$ in a Banach space $E$ will be called \textit{plegma spreading} if each $(e^i_n)_n$ is a normalized Schauder basic sequence and, for every $x\in\spn\{e^i_{n}\}_{i=1,n\in\N}^l$, we have that $\|x\|=\|s(x)\|$ for all plegma shifts $s(x)$ of $x$.
	\end{dfn}

	\begin{rem} Let $(e^i_{n})_{i=1,n\in\N}^l$ be a plegma spreading sequence.
		\begin{enumerate}
			\item[(i)] For every $I\subset \{1,\ldots,l\}$ the sequence $(e^i_n)_{i\in I,n\in\N}$ is also plegma spreading and in particular the sequence $(e_n^i)_n$ is spreading for every $1\le i\le l$.
			\item[(ii)] The set $\{e^i_n\}_{i=1,n\in\N}^l$ is linearly independent.
            \item[(iii)] For every $(s_i)_{i=1}^l\in Plm_l([\N]^\infty)$, the sequence $(e^i_{s_i(n)})_{i=1,n\in \N}^l$ is isometric to $(e^i_{n})_{i=1,n\in\N}^l$ under the natural mapping $T(e^i_{n})=e^i_{s_i(n)}$.
            \item[(iv)] For every $k\in\N$, $x\in\spn\{e^i_{n}\}_{i=1,n=1}^{l,k}$ and $s_n=(s^n_i)_{i=1}^l\in Plm_l([\N]^k)$, such that $s_l^m(k)<s^n_1(1)$ for every $m<n$, we have that the sequence $(x_n)_n$, with $x_n=s_n(x)$, is spreading.
		\end{enumerate}		
	\end{rem}

    \section{Finite Families of Sequences in Banach Spaces}

In this section we study in which cases $l$-tuples of Schauder (unconditional) basic sequences in a given Banach space have subsequences indexed by plegma families that form a common Schauder (unconditional) basic sequence with a natural order. As it turns out this is related to the $w^*$ limits of these sequences in the second dual. The case in which some of the sequences are equivalent to the unit vector basis of $\ell_1$ is the interesting one and we use ultrafilters to deduce the desired conclusion.

    \subsection{Finite Families of Schauder Basic Sequences}	
We first treat non-trivial weak-Cauchy sequences and sequences equivalent to the unit vector basis of $\ell_1$ and eventually we consider weakly null sequences as well. We include a proof of the following well known lemma for completeness.

    \begin{lem}\label{bslem}
        Let $(x_n)_n$ be a normalized sequence in a Banach space $X$ and $x^*_1,\ldots,x^*_k\in {X^*}$ with $\lim x^*_i(x_n)=0$ for all $1\le i \le k$. For every $\delta>0$, there exists $n_0\in \N$ such that $d(x_n,\cap_{i=1}^k\ker x^*_i)<\delta$ for every $n\ge n_0$.
    \end{lem}
    \begin{proof}
        Let $Y=\spn\{x^*_1,\ldots,x^*_k\}$ and $F$ be a finite ${\delta}/{4}$-net of $S_Y$.  Then there exists $n_0\in\N$ such that $f(x_n)<{\delta}/{4}$ for every $f\in F$ and $n\ge n_0$. Pick any $n\ge n_0$. Then, if $d(x_n,\cap_{i=1}^k\ker x^*_i)\ge \delta$,  we may find $x^*\in {X^*}$ with $\|x^*\|=1$ such that $x^*(x_n)\ge \delta$ and $\cap_{i=1}^k\ker x^*_i\subset\ker x^* $. Hence $x^*\in Y$ and there exists $f\in F$ with $\|x^*-f\|<{\delta}/{4}$, which is a contradiction to $x^*(x_n)\ge\delta$ since $f(x_n)<{\delta}/{4}$.
    \end{proof}

The following is a variation of Mazur's method \cite[Theorem 1.a.5]{LT} for finding Schauder basic sequences in infinite dimensional Banach spaces.

    \begin{prop}\label{basicmain}
        Let $(e^1_n)_n,\ldots,(e^l_n)_n$ be seminormalized sequences in a Banach space $X$ and let $E$ denote the closed linear span of $\{e^i_n\}_{i=1,n\in\N}^l$ and $S_E$ the unit sphere of $E$. Assume that there exist $\varepsilon>0$ and a  collection $\{K_F: F\subset S_E$ finite$\}$ of finite subsets of $X^*$ such that
        \begin{enumerate}
            \item[(i)] for every finite $F\subset S_E$ and $x\in F$ there exists $x^*\in K_F$ with $\|x^*\|=1$ and $x^*(x)\ge\varepsilon$,
            \item[(ii)] for every finite $F\subset S_E$ and $1\le i\le l$, there exists $L\in[\N]^\infty$ such that $\lim_{n\in L} x^*(e^i_{n})=0$ for all $x^*\in K_F$, and
            
            \item[(iii)] for every finite $F'\subset F\subset S_E$, we have $K_{F'}\subset K_F$.
        \end{enumerate}
        Then there exist $M_1,\ldots,M_l\in[\N]^\infty$ and a suitable enumeration under which $\cup_{i=1}^l\{e^i_n\}_{n\in M_i}$ is a Schauder basic sequence.
    \end{prop}
    \begin{proof}
We may assume that the sequences $(e^{i}_n)_n$, $1\le i\le l$, are normalized. Indeed, if we normalize the given sequences then conditions (i), (ii), and (iii) will not be affected. If we obtain the result for the normalized versions of the given sequences, then we can revert to subsequences of the initial ones. Let $(\varepsilon_n)_n$ be a sequence in $(0,1/2)$ such that $\sum_{n=1}^\infty \varepsilon_n<{\varepsilon}/{5}$.  We will construct, by induction on $\N$, a Schauder basic sequence $(x_k)_k$ with $x_k=e^{i_k}_{n_k}$, where $i_k= (k-1\mod l)+1$ and $n_{k+1}>n_{k}$. Hence the sets $M_i=\{ n_k : i_k=i\}$, for $1\le i\le l$, and the lexicographic order on $\N\times\{1,\ldots,l\}$ yield the desired result.

        We set $x_1=y_1=e^1_1$ and $F_1 = \{x_1/\|x_1\|,-x_1/\|x_1\|\}$. Assume that $x_1,\ldots,x_k$, $y_1,\ldots,y_k$, and $F_1\subset F_2\subset\cdots \subset F_k$ have been chosen, for some $k\in\N$, such that the following are satisfied: for $1\leq m\leq k$ each $x_m$ is of the form $e^{i_m}_{n_m}$, each $F_m$ is an $\varepsilon_m/2$-net of the unit sphere of $X_m = \spn\{x_1,y_1,\ldots,x_m,y_m\}$, and for $m>1$ each $y_m$ is in $Y_m = \cap\{E\cap\mathrm{ker}x^*:x^*\in K_{F_{m-1}}\}$ with $\|x_m - y_m\| < \varepsilon_m$. We describe the next inductive step. By property (ii) there is $L\in[\mathbb{N}]^\infty$ such that for all $x^*\in K_{F_k}$ we have $\lim_{n\in L}x^*(e^{i_{k+1}}_n)= 0$. Apply Lemma \ref{bslem} to the sequence $(e^{i_{k+1}}_n)_n$ and the subset $\{x^*|_E:x^*\in F_k\}$ of $E^*$ to find $x_{k+1} = e^{i_{k+1}}_{n_{k+1}}$ such that $d(x_{k+1},\cap\{Y\cap \ker x^*:x^*\in K_{F_k}\})<{\varepsilon_{k+1}}/{2}$. Then we may choose a $y_{k+1}\in\cap\{Y\cap \ker x^*:x^*\in K_{F_k}\}$ with $\|x_{k+1}-y_{k+1}\|<\varepsilon_{k+1}$. Finally, pick an $\varepsilon_{k+1}$-net $F_{k+1}$ of the unit sphere of $X_{k+1} = \spn\{x_1,y_1,\ldots,x_{k+1},y_{m+1}\}$.
        
Note that for each $k\in\mathbb{N}$ and $x\in X_m$ there exists $x^*\in K_{F_{m}}$ with $\|x^*\| = 1$ and $x^*(x) \geq (\varepsilon-\varepsilon_m/2)\|x\| \geq (9\varepsilon/10)\|x\|$. Also, because for $m\leq n$ we have $K_{F_m}\subset K_{F_n}$ we have $y_{n+1}\in Y_{n+1}\subset Y_m$ and hence for $x^*\in K_{F_m}$ we deduce $x^*(y_{n+1}) = 0$. We use these facts to first observe that $(y_k)_k$ is $K$-Schauder basic, for $K=(10/(9\varepsilon))$. Indeed, if $m\leq n$ and $a_1,\ldots,a_n\in\mathbb{R}$ then $x = \sum_{i=1}^ma_iy_i\in X_m$ and for some $x^*\in F_m$ with $\|x^*\| = 1$
\begin{equation*}
\Big\|\sum_{i=1}^ma_iy_i\Big\| \leq \frac{6}{5\varepsilon}x^*\Big(\sum_{i=1}^ma_iy_i\Big)
 = \frac{10}{9\varepsilon}x^*\Big(\sum_{i=1}^na_iy_i\Big) \leq \frac{10}{9\varepsilon}\Big\|\sum_{i=1}^na_iy_i\Big\|.
\end{equation*}       
By the principle of small perturbations, for $(x_k)_k$ to be Schauder basic it suffices to show $2K\sum_k\|x_k - y_k\|/\|y_k\| < 1$. This follows from   $\sum_k\|x_k - y_k\|/\|y_k\| < 2\sum_n\varepsilon_k <2\varepsilon/5$.
    \end{proof}

    \begin{rem}\label{remplm}
        Let us observe that $(M_i)_{i=1}^l$, as constructed in the previous proof, is a plegma family in $[\N]^\infty$. In general, for every $M_1,\ldots,M_l\in [\N]^\infty$, there exists a plegma family $(s_i)_{i=1}^l$ in $[\N]^\infty$ with $s_i\subset M_i$.

        Moreover, for any plegma family $(s_i)_{i=1}^l\in Plm_l([\N]^\infty)$, we associate a natural order on $\{s_i(n)\}_{i=1,n\in\N}^l$ and that is the lexicographic order on $[\N]\times\{1,\ldots,l\}$.
    \end{rem}

 	The following lemma is an immediate consequence of the principle of local reflexivity \cite{LR}, however we also give the following easy proof.
 	
    \begin{lem}\label{lemlr}
        Let $X$ be a Banach space and $F$ be a linear subspace of $X^{**}$ with finite dimension and $X\cap F=\{0\}$. Then, for every $\delta>0$ and $x\in X$ with $\|x\|=1$, there exists $x^*\in X^*$ with $\|x^*\|\le 1+\delta$ such that $x^*(x)\ge \varepsilon=d(S_X,F)$ and $x^{**}(x^*)=0$ for every $x^{**}\in F$.
    \end{lem}
    \begin{proof}
     Let $x\in X$ with $\|x\|=1$ and $Y=\spn\{F\cup\{x\}\}$. Then there exists $x^{***}\in S_{Y^*}$ such that $x^{***}(x)\ge\varepsilon$ and $x^{***}(x^*)=0$ for every $x^{**}\in F$. Then we consider the identity map $I:Y\to X^{**}$ and recall that its conjugate $I^*:X^{***}\to Y^*$ is $w^*\-w^*$-continuous. Since $B_{X^*}$ is a $w^*$-dense subset of $B_{X^{***}}$ and $Y$ is of finite dimension, it follows that $I^*(B_{X^*})$ is norm dense in $B_{Y^*}$ and hence, for every $\delta>0$, we have that $B_{Y^*}\subset I^*((1+\delta)B_{X^*})$, which yields the desired result.
    \end{proof}

	Recall that a sequence $(x_n)_n$ in a Banach space $X$ is called non-trivial weak-Cauchy if there exists $x^{**}\in X^{**}\setminus X$ such that $w^*\-\lim x_n=x^{**}$. The next proposition provides a complete characterization to the aforementioned problem for finite collections of such sequences.

    \begin{prop}\label{propntwc}
         Let $(e^1_n)_n,\ldots,(e^l_n)_n$ be seminormalized non-trivial weak-Cauchy sequences in a Banach space $X$ and $F=\spn\{e^{**}_i\}_{i=1}^l$, where $w^*\-\lim e^i_n=e^{**}_i$. Then there exists an $(s_i)_{i=1}^l\in Plm_l([N]^\infty)$ such that $\{e^i_{s_i(n)} \}_{i=1,n\in \N}^l$ is a Schauder basic sequence, enumerated by the natural plegma order, if and only if $X\cap F=\{0\}$.
    \end{prop}
    \begin{proof}
        Let $X\cap F\neq\{0\}$, hence there exists $x=\sum_{i=1}^la_ie^{**}_i\in X$ with $x\neq 0$.\linebreak If there exists a plegma family $(s_i)_{i=1}^l\in Plm_l([\N]^\infty)$ such that $\cup_{i=1}^l\{e^i_n\}_{n\in M_i}$ is a Schauder basic sequence, we consider the sequence $(x_n)_n$ with $x_n=\sum_{i=1}^la_ie^i_{s_i(n)}$. Notice that $(x_n)_n$ is a Schauder basic sequence with $w\-\lim x_k=x$, which is a contradiction since $x\neq0$.

        Suppose now that $X\cap F=\{0\}$ and let $\varepsilon=d(S_E,F)$. Then for every $x\in S_E$, by Lemma \ref{lemlr}, there exists an $f_x\in E^*$ with $\|f_x\| = 1$ such that $f_x(x)\ge\varepsilon/2$ and $x^{**}(f_x)=0$ for every $x^{**}\in F$ and hence $\lim x^*(e^i_n)=0$ for every $1\le i\le l$. Finally, applying Proposition \ref{basicmain} (setting $K_F = \{f_x:x\in F\}$) and Remark \ref{remplm} the proof is complete.
    \end{proof}

	Next we give an example of a plegma spreading sequence, formed by two  non-trivial Weak-Cauchy sequences, that is not Schauder basic.
	
	\begin{dfn}[\cite{J1}]\label{james space}
		On the space $c_{00}(\N)$ we define the following norm
		\[
		\|x\|_J=\sup\Bigg(\sum_{i=1}^n\bigg(\sum_{k\in I_i}x(k)\bigg)^2\Bigg)^{\frac{1}{2}},
		\]
		where the supremum is taken over all finite collections $I_1,\ldots,I_n$ of disjoint intervals of natural numbers. The \textit{James' space}, denoted by $J$, is the completion of $c_{00}(\N)$ with respect to $\|\cdot\|_J$.
	\end{dfn}

   \begin{exmp}
    \label{james mixed funnily}
    	Let $(e_n)_n$ denote the standard basis of James' space and recall that it is a non-trivial weak-Cauchy sequence. We consider the sequences $(e^1_n)_n$ and $(e^2_n)_n$ in {\em J}, with $e^1_n=e_{2n}+e_1$ and $e^2_n=e_{2n+1}-e_1$, which are also non-trivial weak-Cauchy and denote by $e^{**}_1,e^{**}_2$ their $w^*\-$limits. Notice that the sequence $(e^i_n)_{i=1,n\in\N}^2$ is a plegma spreading sequence in $J$. Moreover, since $e_1\in J\cap\spn\{e^{**}_1,e^{**}_2 \}$ and $T(e^i_n)=e^i_{s_i(n)}$ is an isometry for every $(s_i)_{i=1}^2\in Plm_l([\N]^\infty)$, then the same arguments as in Proposition \ref{propntwc} yield that $(e^i_n)_{i=1,n\in\N}^2$ is not Schauder basic.
    \end{exmp}

    We pass now to study the case of finite families of $\ell_1$ sequences in a Banach space.\linebreak As is well known, $\beta\N$ denotes the Stone-{\v C}ech compactification of $\N$ and therefore $\ell_\infty(\N)$ is isometric to $C(\beta\N)$. It is also known that the elements of $\beta\N$ are the ultrafilters on $\N$. The identification of $\ell_\infty(\N)$ with $C(\beta\N)$ yields that the conjugate space of $\ell_\infty(\N)$ is isometric to $\mathcal{M}(\beta\N)$, the set of all regular measures on $\beta\N$.

     For $f\in\ell_\infty(\N)$ and $p$ an ultrafilter on $\N$, the evaluation of the Dirac measure $\delta_p$ on the function $f$ is given as $\delta_p(f)=\lim_p f(n)$, where $\lim_p f(n)$ is the unique limit of $(f(n))_{n}$ with respect to the ultrafilter $p$. Let us also observe that if $T:\ell_1\to X$ is an isomorphic embedding, then $T^{**}:\mathcal{M}(\beta\N)\to X^{**}$ and for every $p\in\beta\N$ and $x^*\in X^*$, we have that $T^{**}\delta_p(x^*)=\lim_px^*(Te_n)$. For further information on ultrafilters we refer to \cite{CN}.

    \begin{lem}\label{leml1}
        Let $X$ be a Banach space and $T:\ell_1\to X$ be an isomorphic embedding. Let $\alpha\in\R$, $x^*_1,\ldots,x^*_k\in X^*$ and $p$ be a non-principal ultrafilter on $\N$ such that $T^{**}\delta_p(x^*_i)=\alpha$ for all $1\le i\le k$. Then there exists $M\in[\N]^\infty$ such that $\lim_{n\in M} x^*_i(Te_n)=\alpha$ for every $1\le i\le k$.
    \end{lem}
	\begin{proof}
		Notice that $T^{**}\delta_p(x^*_i)=\lim_px^*_i(Te_n)$ and also that, for any $n\in\N$, the set $M_n=\{m\in \N:|x^*_i(Te_m)-\alpha|<{1}/{n},\;\text{for all } 1\le i\le l \}$ is in $p$ and is not finite, since $p$ is a non-principal ultrafilter. Let $M$ be a diagonalization of $(M_n)_n$, i.e. $M(n)\in M_n$ for all $n\in\N$, then $(Te_n)_{n\in M}$ is the desired subsequence.
	\end{proof}

	\begin{lem}\label{ufilter}
		Let $X$ be a separable Banach space and $F$ be a finite dimensional subspace of $X^{**}$ with $X\cap F=\{0\}$. Let also $(e^1_n)_n,\ldots,(e^l_n)_n$ be sequences in $X$ such that each one is equivalent to the basis of $\ell_1$ and denote by $T_i$ the corresponding embedding. Then there exist non-principal ultrafilters $p_1,\ldots,p_l$ on $\N$ such that
    \begin{enumerate}
        \item[(i)] The set $F\cup\{T_i^{**}\delta_{p_{i}}\}_{i=1}^{l}$ is linearly independent.
        \item[(ii)] $X\cap\spn\{F\cup\{T_i^{**}\delta_{p_{i}}\}_{i=1}^{l}\}=\{0\}$.
    \end{enumerate}
	\end{lem}
	\begin{proof}
		Let us observe that the cardinality of $\beta\N$ is $2^\mathfrak{c}$ whereas that of any separable Banach space is less or equal to $\mathfrak{c}$. We also remind that the family $\{\delta_p:p\in\beta\N \}$ is equivalent to the basis of $\ell_1(2^\mathfrak{c})$ and hence linearly independent. The same remains valid for a fixed $1\le i\le l$ and the family $\{T^{**}_i\delta_p:p\in\beta\N \}$, since $T^{**}_i$ is an isomorphism. We consider the linear space $X^{**}/X$ and we denote by $Q$ the natural quotient map $Q:X^{**}\to X^{**}/X$.
		\\\par
		{\em\underline{Claim}} : For every $1\le i \le l$, there exists an uncountable subset $A_i$ of $\beta\N$ such that the family $\{QT^{**}_i\delta_p:p\in A_i\}$ is linearly independent.
		\begin{proof}[Proof of Claim]\renewcommand{\qedsymbol}{}
			If not, there would exist a countable subset $A_i$ of $\beta\N$ such that $\{QT^{**}_i\delta_p:p\in A_i \}$ is a maximal independent subfamily of $\{QT^{**}_i\delta_p:p\in \beta\N\}$, for some $1\le i\le l$. Then $\{T^{**}_i\delta_p:p\in\beta\N\}\subset\spn\{X\cup\{T_i^{**}\delta_p:p\in A_i \}\}$, which yields a contradiction, since the algebraic dimension of $X$ is less or equal to $\mathfrak{c}$.
		\end{proof}
		Since $F$ satisfies $X\cap F=\{0\}$, it follows that $Q|_F$ is an isomorphism and by induction we will choose, for every $1\le i\le l$, an ultrafilter $p_{i}$ on $\N$ such that $p_{i}\in A_i$ and $QT^{**}_i\delta_{p_{i}}\notin\spn\{ Q[F]\cup\{ QT^{**}_{j}\delta_{p_{j}} \}_{j<i}\}$. For $i=1$, there exists $p_{1}\in A_1$ with $QT^{**}_{1}\delta_{p_{1}}\notin Q[F]$, since $F$ has finite dimension and $A_1$ is uncountable. Suppose that $p_1,\ldots,p_i$ have been chosen for some $i<l$. Then there exists $p_{i+1}\in A_{i+1}$ with $QT^{**}_{i+1}\delta_{p_{i+1}}\notin\spn\{ Q[F]\cup\{ QT^{**}_{j}\delta_{p_{j}} \}_{j\le i}\}$, for the same reason as above and this completes the inductive construction. Notice that the ultrafilters $p_1,\ldots,p_l$ are non-principal since $T^{**}\delta_{p_i}\notin X$ for every $1\le i\le l$.
	\end{proof}

    \begin{cor}\label{coruf}
        Let $(e^1_n)_n,\ldots,(e^l_n)_n$ be sequences in a separable Banach space $X$ such that each one is equivalent to the basis of $\ell_1$ and denote by $T_i$ the corresponding embedding. Then, for every $k\in\N$ and every $1\le i\le l$ and $1\le j\le k$, there exists a non-principal ultrafilter $p_{ij}$ on $\N$ such that
    \begin{enumerate}
        \item[(i)] The set $\{T_i^{**}\delta_{p_{ij}}\}_{i=1,j=1}^{l,k}$ is linearly independent.
        \item[(ii)] $X\cap\spn\{T_i^{**}\delta_{p_{ij}}\}_{i=1,j=1}^{l,k}=\{0\}$.
    \end{enumerate}
    \end{cor}

    The following lemma is an immediate consequence of the above and it will used in the next subsection.

    \begin{lem}\label{uflim}
        Let $(e^1_n)_n,\ldots,(e^l_n)_n$ be sequences in a separable Banach space $X$ such that each one is equivalent the basis of $\ell_1$ and denote by $T_i$ the corresponding embedding. Then there exist $x^*_1\ldots,x^*_l\in X^*$ such that the following hold.
        \begin{enumerate}
            \item[(i)] For every $1\le i\le l$, there exists $M_i\in[\N]^\infty$ such that $\lim_{n\in M_i} x^*_i(e^i_n)=1$.
            \item[(ii)] For every $1\le i,j\le l$, there exists $M^i_{j}\in[\N]^\infty$ such that $\lim_{n\in M^i_{j}} x^*_i(e^j_n)=0$.
        \end{enumerate}
    \end{lem}
    \begin{proof}
    	From Corollary \ref{coruf}, there exist $p_1,\ldots,p_l,q_1,\ldots,q_l$ non-principal ultrafilters on $\N$ such that the set $\{T_i^{**}\delta_{p_{i}}\}_{i=1}^{l}\cup\{T_i^{**}\delta_{q_{i}}\}_{i=1}^{l}$ is linearly independent. Then, for each $1\le i\le l$, we choose $x^{***}_i\in X^{***}$ such that $x^{***}_i(T^{**}_j\delta_{p_{j}})=\delta_{ij}$ and $x^{***}_i(T^{**}_{j}\delta_{q_{j}})=0$ for every $1\le j\le l$. The principle of local reflexivity then yields an $x^*_i\in X^*$ such that $T^{**}_j\delta_{p_j}(x^*_i)=\delta_{ij}$ and $T^{**}_{j}\delta_{q_{j}}(x^*_i)=0$ for every $1\le j\le l$. Finally, applying Lemma \ref{leml1} we obtain the desired subsequences.
    \end{proof}

	Next we give a characterization to the problem in the general case. Recall that as follows from Rosenthal's $\ell_1$ theorem \cite{Rosenthal} and the theory of Schauder bases, if $(x_n)_n$ is a Schauder basic sequence in a Banach space $X$, then it contains a subsequence which is either weakly null or equivalent to a basis of $\ell_1$ or non-trivial weak-Cauchy.

    \begin{thm}\label{thmbsc}
        Let $(e^1_n)_n,\ldots,(e^l_n)_n$ be seminormalized sequences in a Banach space $X$ such that each one is either weakly null, equivalent to the basis of $\ell_1$, or non-trivial weak-Cauchy. Let $I\subset\{1,\ldots,l\}$ be such that $(e^i_n)_n$ is a non-trivial weak-Cauchy sequence with $w^*\-\lim e^i_n=e^{**}_i$ for every $i\in I$ and set $F=\spn\{ e^{**}_i\}_{i\in I}$. Then there exists $(s_i)_{i=1}^l\in Plm_l([N]^\infty)$ such that $\{e^i_{s_i(n)} \}_{i=1,n\in \N}^l$ is a Schauder basic sequence, enumerated by the natural plegma order, if and only if $X\cap F=\{0\}$.
    \end{thm}
	\begin{proof}
		Let $J\subset\{1,\ldots,l \}$ be such that $(e^i_n)_n$ is equivalent to the basis of $\ell_1$ for each $i\in J$ and denote by $T_i$ the corresponding embedding. Then Lemma \ref{ufilter} yields for every $i\in J$ a non-principal ultrafilter $p_i$ on $\N$ such that $X\cap Y=\{0\}$, where $Y=\spn\{F\cup\{T^{**}_i\delta_{p_i} \}_{i\in J}\}$. For $\varepsilon=d(S_X,Y)$ it follows from Lemma \ref{lemlr} that, for every $x\in X$ with $\|x\|=1$, there exists $f_x\in X^*$ with $\|f_x\| = 1$ such that $f_x(x)\ge \varepsilon/2$ and $x^{**}(f_x)=0$ for every $x^{**}\in Y$. For every finite $F\subset S_X$, we set $K_F = \{f_x:x\in F\}$. Note that $\lim x^*(e^i_n)=0$ for every $i\in I$, every finite $F\subset S_X$ and every $x^*\in K_F$. Also, from Lemma \ref{leml1}, for each $i\in J$, there exists $M_i\in[\N]^\infty$ with $\lim_{n\in M_i} x^*(e^i_{n})=0$ for all $x^*\in K_f$. Applying Proposition \ref{basicmain}, we derive the desired result.
    \end{proof}

	More specifically, for a plegma spreading sequence $(e^i_n)_{i=1}^l$, Rosenthal's theorem yields that each $(e^i_n)$ is either weakly null, equivalent to the unit vector basis of $\ell_1$, or non-trivial Weak-Cauchy, since it is spreading. Taking also into account the behavior of plegma spreading sequences we give a corollary of the above theorem.

	\begin{cor}	
		Let $(e^i_n)_{i=1,n\in\N}^l$ be a plegma spreading sequence in a Banach space $X$ and $I\subset\{1,\ldots,l\}$ be such that $(e^i_n)_n$ is a non-trivial weak-Cauchy sequence with $w^*\-\lim e^i_n=e^{**}_i$ for every $i\in I$ and set $F=\spn\{ e^{**}_i\}_{i\in I}$. Then $(e^i_n)_{i=1,n\in\N}^l$ is a Schauder basic sequence, enumerated by the lexicographic order on $[\N]\times\{1,\ldots,l\}$, if and only if $X\cap F=\{0\}$.
	\end{cor}
	\begin{proof}
		 Theorem \ref{thmbsc} yields a plegma family $(s_i)_{i=1}^l$ in $[\N]^\infty$ such that $(e^i_{s_i(n)})_{i=1,n\in\N}^l$ is a Schauder basic sequence if and only if $X\cap F= \{0\}$, and since $T(e_n^i)=e^i_{s_i(n)}$ is an isometry, then the same holds for $(e^i_n)_{i=1,n\in\N}^l$.
	\end{proof}

    \subsection{Finite Families of Unconditional Sequences}
    We now study the case of unconditional sequences. We start with plegma spreading sequences. Recall that every weakly null spreading sequence in a Banach space is unconditional. The following proposition extends this result to plegma spreading sequences, using similar arguments as in the classical case.

	\begin{prop}\label{w0unc}
		Let $(e^i_{n})_{i=1,n\in\N}^l$ be a plegma spreading sequence such that each $(e^i_n)_n$ is weakly null. Then $(e^i_{n})_{i=1,n\in\N}^l$ is an unconditional sequence.
	\end{prop}
    \begin{proof}
    	Let $\pi=\{1,\ldots,l\}\times\{1,\ldots,k \}$ and $x=\sum_{(i,j)\in \pi}a_{ij}e^i_j$. Since each $(e^i_n)_n$ is weakly null, then for every $\varepsilon>0$ and $(i_0,j_0)\in\pi$, there exist $s_1(x),\ldots,s_m(x)$ plegma shifts of $x$ and a convex combination $\sum_{t=1}^m\lambda_ts_t(x)$ such that
    	\begin{enumerate}
    		\item[(i)] $s_{t_1}(e^i_j)=s_{t_2}(e^i_j)$ for every $(i,j)\in\pi'=\pi\setminus\{(i_0,j_0)\}$ and $1\le t_1,t_2,\le m$.
    		\item[(ii)] $\|\sum_{t=1}^n\lambda_ts_t(e^{i_0}_{j_0})\|<\varepsilon\|x\|/a_{i_0j_0}$.
    		\item[(iii)] $\sum_{t=1}^m\lambda_ts_t(x)=\sum_{ (i,j)\in\pi'}a_{ij}s_1(e^i_{_j})+\sum_{t=1}^n\lambda_ta_{i_0j_0}s_t(e^{i_0}_{j_0})$.
    	\end{enumerate}
    Then since $(e^i_{n})_{i=1,n\in\N}^l$ is plegma spreading we have that $\|s_t(x)\|=\|x\|$ for all $1\le t\le m$ and also $\|\sum_{ (i,j)\in\pi'}a_{ij}s_1(e^i_{j})\|=\|\sum_{ (i,j)\in\pi' }a_{ij}e^i_{j}\|$ and hence
    \[ \Big\|\sum_{ (i,j)\in\pi'}a_{ij}e^i_{j}\Big\|\le (1+\varepsilon)\Big\|\sum_{(i,j)\in \pi}a_{ij}e^i_j\Big\|.\]
    Finally, applying iteration we show that for every $\varepsilon>0$ and every $F\subset \pi$,
    \[
    \Big\|\sum_{ (i,j)\in F}a_{ij}e^i_{j}\Big\|\le (1+\varepsilon)\Big\|\sum_{(i,j)\in \pi}a_{ij}e^i_j\Big\|.
    \]
    \end{proof}

    \begin{prop}
        Let $(e^i_{n})_{i=1,n\in\N}^l$ be a plegma spreading sequence. If each $(e^i_n)_n$ is equivalent to the basis of $\ell_1$, then the same holds for $(e^i_{n})_{i=1,n\in\N}^l$.
    \end{prop}
    \begin{proof}
    	Let $0<\varepsilon<1$, $0<\delta<(1-\varepsilon)/2l$, and $x=\sum_{j=1}^k\sum_{i=1}^la_{ij}e^i_{j}$ with $\sum_{j=1}^k\sum_{i=1}^l|a_{ij}|=1$. Then either for $x$ or $-x$ there exists $1\le i_0\le l$ such that $\sum_{j\in J_{i_0}^+}a_{i_0j}\ge {1}/{2l}$, where $J_{i_0}^+=\{j:a_{i_0j}>0 \}$. Moreover, from Lemma \ref{uflim}, for each $1\le i,j\le l$, there exist $x_{i}\in X^*$ and $M_{i},M^i_{j}$ in $[\N]^\infty$ such that $\lim_{n\in M_i}x_{i}^*(e^{i}_n)=1$ and $\lim_{n\in M^i_{j}}x_{i}^*(e^j_n)=0$. We set $M=\max\{\|x^*_i\|:{i=1,\ldots,l}\}$ and choose a plegma family $(s_i)_{i=	1}^l$ in $[\N]^k$ such that
    	\begin{enumerate}
    		\item[(i)] $s_{i_0}(j)\in M_{i_0}$ and $x^*(e^i_{s_{i_0}(j)})>1-\varepsilon$ for every $j\in J^+_{i_0}$.
    		\item[(ii)] $s_{i_0}(j)\in M^{i_0}_{i_0}$ for every $j\in J_{i_0}^-=\{j:a_{i_0j}<0 \}$.
    		\item[(ii)] $s_j\subset M^{i_0}_j$ for every $1\le j\le l$ with $j\neq i_0$.
    		\item[(iii)] $x_{i_0}^*(\sum_{j=1}^k\sum_{\substack{ i=1 \\ i\neq i_0}}^la_{ij}e^i_{s_i(j)}+\sum_{j\in J^-_{i_0}}a_{i_0j}e^{i_0}_{s_{i_0}(j)})<\delta$.
    	\end{enumerate}
    Hence $x_{i_0}^*(x)\ge \frac{1-\varepsilon}{2l}-\delta$ and therefore $\|x\|\ge (\frac{1-\varepsilon}{2l}-\delta)/M$ which yields the desired result.
    \end{proof}

	Combining the two previous propositions we have the following final result.

    \begin{thm}
    \label{unconditional plegma spreading jointly unconditional}
        Let $(e^i_{n})_{i=1,n\in\N}^l$ be a plegma spreading sequence such that each $(e^i_n)_n$ is unconditional. Then $(e^i_{n})_{i=1,n\in\N}^l$ is also an unconditional sequence.
    \end{thm}
    \begin{proof}
    	Let $I\subset\{1,\ldots,l\}$ be such that the sequence $(e^i_n)_n$ is weakly null for every $i\in I$ and denote its complement by $J$. We denote by $E_0$ the closed linear span of $\{e^i_n\}_{i\in I,n\in \N}$ and by $E_1$ that of $\{e^i_n\}_{i\in J,n\in \N}$. Then, for any $x\in E_0+E_1$ with $x=\sum_{j=1}^k(\sum_{i\in I}a_{ij}e^i_j+\sum_{i\in J}b_{ij}e^i_j)$, using similar arguments as in the proof of Proposition \ref{w0unc}, we have that for every $\varepsilon>0$
    \[\Big\|\sum_{j=1}^k\sum_{i\in J}b_{ij}e^i_j\Big\|\le(1+\varepsilon)\Big\|\sum_{j=1}^k\Big(\sum_{i\in I}a_{ij}e^i_j+\sum_{i\in J}b_{ij}e^i_j\Big)\Big\|.\]
    Hence $E_0+E_1=E_0\oplus E_1$ and since both $(e^i_n)_{i\in I,n\in\N}$ and $(e^i_n)_{i\in J,n\in\N}$ are unconditional sequences, as follows from the two previous propositions, then the same holds for their union.
    \end{proof}

\subsection{Unconditional Sequences in Singular Position}

   The following is a variant of the Maurey-Rosenthal classical example \cite{MR}. As Theorem \ref{unconditional plegma spreading jointly unconditional} asserts, the strong assumption of being plegma spreading yields that $l$-tuples of unconditional sequences are jointly unconditional. The purpose of this example is to demonstrate that this strong condition is in fact necessary.

    \begin{prop}
	Let $N_1$, $N_2$ be a partition of $\mathbb{N}$ into two infinite sets. There exists a Banach space $X$ with a Schauder basis $(e_n)_n$ such that the following hold.
        \begin{enumerate}
            \item[(i)] The sequences $(e_n)_{n\in N_1}$ and $(e_n)_{n\in N_2}$ are unconditional.
            \item[(ii)] For every $M\subset \N$ such that $M\cap N_1$ and $M\cap N_2$ are both infinite sets, the sequence $(e_n)_{n\in M}$ is not unconditional.
        \end{enumerate}
    \end{prop}

    We fix a strictly increasing sequence of natural numbers $(\mu_i)_i$ such that
    \begin{equation*}\label{muj}
    \sum_{i=1}^\infty\sum_{j>i}\frac{\sqrt\mu_i}{\sqrt\mu_j}\le\frac{1}{2}
    \end{equation*}

    \noindent Denote by $\mathcal{P}$ the collection of all finite sequences $(E_k)_{k=1}^n$ of successive non-empty finite subsets of $\N$. Take an injection $\sigma:\mathcal{P}\to \N$ such that $\sigma((E_k)_{k=1}^{n})>\max\{\# E_k \}_{k=1}^{n}$, for every $(E_k)_{k=1}^{n}\in\mathcal{P}$, and finally fix a partition of $\N$ into two infinite subsets $N_1$ and $N_2$.

    \begin{dfn}
        A sequence $(E^i_k)_{i=1,k=1}^{2,n}$ of non-empty finite subsets of $\N$ is called a {\em special sequence} if the following hold.
        \begin{enumerate}
            \item[(i)] $E^1_k\subset N_1$ and $E^2_k\subset N_2$ for every $1\le k\le n$.
            \item[(ii)] The sets $E^1_k$ and $E^2_k$ are successive for every $1\le k\le n$.
            \item[(iii)] The sets $E^2_k$ and $E^1_{k+1}$ are successive for every $1\le k <n$.
            \item[(iv)] $\#E^1_1 = \#E^2_1= \mu_{j_1}$, for some $j_1\in\mathbb{N}$.
            \item[(v)] $\#E^1_{k_0}=\#E^2_{k_0}=\mu_{j_{k_0}}$, where $j_{k_0}=\sigma((E^i_k)_{i=1,k=1}^{2,k_0-1})$ for every $1<{k_0}\le n$.
        \end{enumerate}
    \end{dfn}

    \begin{rem}\label{remtreespecial}
		Let $(E^i_k)_{i=1,k=1}^{2,n}$ and $(F^i_k)_{i=1,k=1}^{2,m}$ be special sequences and set $k_0=\min\{k:\# E^1_k\neq\#F^1_k\}$. Then since $\sigma$ is an injection, notice that if $k_0>1$ the following hold.
		\begin{enumerate}
			\item[(i)] If $k_0>2$ then $E^1_k=F^1_k$ and $E^2_k=F^2_k$ for every $1\le k<k_0-1$.
			\item[(ii)] $\#E^1_{k_0-1}=\# F^1_{k_0-1}$ and $E^1_{k_0-1}\neq F^1_{k_0-1}$ or $E^2_{k_0-1}\neq F^2_{k_0-1}$.
			\item[(iii)] $\#E^1_{k}\neq\# F^1_{k}$ for every $k_0\le k\le \min\{n,m\}$.
		\end{enumerate}
	\end{rem}

    Let $(e^*_i)_i$ as well as $(e_i)_i$ denote the unit vector basis of $c_{00}(\N)$ and for every $f,x\in c_{00}(\N)$ with $f=\sum_{i=1}^na_ie^*_i$ and $x=\sum_{i=1}^mb_ie_i$ set $f(x)=\sum_{i=1}^{\min\{n,m\}}a_ib_i$. Finally, for $f,g\in c_{00}(\N)$ with $f=\sum_{k=1}^{n}a_{i_k}e^*_{i_k}$ and $g=\sum_{k=1}^nb_{j_k}e^*_{j_k}$, where $a_{i_k},b_{j_k}\neq0$, we will say that $f$ and $g$ are {\em consistent} if $\sgn(a_{i_k})=\sgn(b_{j_k})$ for every $1\le k\le n$.

    \begin{dfn}
        Consider the following subsets of $c_{00}(\N)$.
        \[
        W_0=\{0\}\cup\{\pm e^*_i:i\in \N\},
        \]
        \[
        W_1=\Bigg\{\frac{1}{\sqrt{\mu_j}}\sum_{i\in E}\varepsilon_ie^*_i:E\subset\N_1\;\textnormal{or}\;E\subset N_2,\#E=\mu_j,\varepsilon_i\in\{-1,1\}\;\textnormal{for}\;i\in E\Bigg\},
        \]
        \begin{align*}
        W_2=\Bigg\{&\sum_{i=1}^2\sum_{k=1}^{n}f^i_k:f^i_k\in W_0\cup W_1,(\supp(f^i_k))_{i=1,k=1}^{2,n}\;\textnormal{is a special sequence},\\
        & f^1_k\;\textnormal{and}\;f^2_k\;\textnormal{are consistent for every }1\le k\le n\Bigg\}
        \end{align*}
        and set $W=\{P_E(f):f\in W_0\cup W_1\cup W_2,E\;\textnormal{interval of }\N \}$. Define a norm on $c_{00}(\N)$ by setting $\|x\|=\sup\{f(x):f\in W\}$ and let $X^{(2)}_{MR}$ denote its completion with respect to this norm.
        \end{dfn}
    	
    	\begin{rem}
    		For any $f=\sum_{i=1}^na_ie^*_i$ in $W$ and $1\le k <l \le n$, it follows that $\sum_{i=k}^la_ie^*_i$ is in $W$ as well. That is, the sequence $(e_i)_i$ forms a normalized and bimonotone Schauder basis for $X^{(2)}_{MR}$.
    	\end{rem}

    	\begin{rem}
			For any $f=\sum_{i=1}^na_ie^*_i$ in $W$ and any choice of signs $(\varepsilon_i)_{i\in N_1}$ (or $(\varepsilon_i)_{i\in N_2}$), it follows that there exist $(b_i)_{i=1}^n$ such that $b_i=\varepsilon_ia_i$ for all $i\in N_1\cap\{1,\ldots,n\}$ (or $i\in N_2\cap\{1,\ldots,n\}$) and  $g=\sum_{i=1}^nb_ie^*_i$ is in $W$. Hence, the sequences $(e_i)_{i\in N_1}$ and $(e_i)_{i\in N_2}$ are $1$-unconditional.
    	\end{rem}

        \begin{dfn}
            We will call an $x$ in $X^{(2)}_{MR}$ a {\em weighted vector} if $x=\frac{1}{\sqrt{\mu_\ell}}\sum_{i\in E}\varepsilon_ie_i$ with $\#E=\mu_\ell$ and $\varepsilon_i\in\{-1,1\}$. We also define the {\em weight} of $x$ as $w(x)=\mu_\ell$.

            Moreover, any $f=\frac{1}{\sqrt{\mu_\ell}}\sum_{i\in E}\varepsilon_ie^*_i$ in $W$ with $\#E=\mu_\ell$ and $\varepsilon_i\in\{-1,1\}$, will be called a {\em weighted functional} and we define the {\em weight} of $f$ as $w(f)=\mu_\ell$.
        \end{dfn}

        \begin{lem}\label{incweig}
            Let $x_1,\ldots,x_n$ be successive weighted vectors with increasing weights and $f_1,\ldots,f_m$ be successive functionals with increasing weights. If $w(x_i)\neq w(f_j)$ for every $1\le i\le n$ and $1\le j\le m$, then $\sum_{i=1}^m\sum_{i=1}^n|f_j(x_i)|\le\frac{1}{2}$.
        \end{lem}
        \begin{proof}
            Let us observe that for $x,f$ weighted with $w(x)=\mu_\ell$ and $w(f)=\mu_k$ such that $\mu_\ell\neq \mu_k$, it holds that $|f(x)|\le\frac{\min\{ \sqrt{\mu_\ell}, \sqrt{\mu_k}  \}}{\max\{ \sqrt{\mu_\ell}, \sqrt{\mu_k} \} }$ and hence $|f(x)|\le\frac{\sqrt{\mu_\ell}}{\sqrt{\mu_k}}$ if $\mu_\ell<\mu_k$ and $|f(x)|\le\frac{\sqrt{\mu_k}}{\sqrt{\mu_\ell}}$ if $\mu_k<\mu_\ell$. Let $w(x_i)=\mu_{\ell_i}$ and $w(f_j)=\mu_{k_j}$. Then, for each pair $(i,j)$, we have that $|f_j(x_i)|\le\min\Big\{ \frac{ \sqrt{ \mu_{\ell_i}} }{ \sqrt{\mu_{k_j}} },\frac{ \sqrt{ \mu_{k_j}} }{ \sqrt{\mu_{\ell_i}} }   \Big\}$ and since each pair $(\mu_{\ell_i},\mu_{k_j})$ appears only once, we have that
            \[
            \sum_{j=1}^m\sum_{i=1}^n|f_j(x_i)|\le\sum_{j=1}^\infty\sum_{j>i}\frac{\sqrt{\mu_i}}{\sqrt{\mu_j}}\le\frac{1}{2}.
            \]
        \end{proof}

		\begin{prop}\label{propunc}
			Let $(E^i_k)_{i=1,k=1}^{2,n}$ be a special sequence and define the vector $x^i_k=(1/\sqrt{\# E^i_k})\sum_{i\in E^i_k}e_i$ for $1\le k\le n$ and $i=1,2$. Then $\|\sum_{i=1}^2\sum_{k=1}^nx^i_k\|\ge 2n$ whereas $\|\sum_{i=1}^2\sum_{k=1}^n(-1)^{i}x^i_k\|\le 5$.
		\end{prop}
		\begin{proof}
			Set $\mu_{j_k}=\#E^1_k$ for $1\le k\le n$. The first part follows easily from the fact that $f=\sum_{i=1}^2\sum_{k=1}^n(1/\sqrt{\# E^i_k})\sum_{j\in E^i_k}e^*_j$ is in $W$. For the second part set $y=\sum_{i=1}^2\sum_{k=1}^n(-1)^{i}x^i_k$ and let $g=\sum_{i=1}^2\sum_{k=1}^m\frac{1}{\sqrt{\mu_{j_k}}}\sum_{j\in F^i_k}e^*_j$ in $W_2$ and also $k_0=\min \{k:\#E^1_k\neq \#F^1_k \}$, making the convention $\min\emptyset = m+1$. If $k_0=1$, then the previous lemma yields that $|g(y)|\le\frac{1}{2}$. Otherwise, by Remark \ref{remtreespecial} and Lemma \ref{incweig} the following hold.
			\begin{enumerate}
				\item[(i)] If $k_0>2$, then $g^1_{k}(y)=-g^2_k(y)$ for every $1\le k<k_0-1$.
				\item[(ii)] $|g^i_{k_0-1}(y)|\le 1$ and $\sum_{k=k_0}^m|g^i_k(y)|\le\frac{1}{2}$ for $i=1,2$.
			\end{enumerate}
			 Hence $|g(y)|\le 3$. Finally, in the general case that $g\in W$, using similar arguments we conclude that $|g(y)|\le 5$.
		\end{proof}
	
		\begin{prop}
			Let $M\subset \N$ be such that $M\cap N_1$ and $M\cap N_2$ are both infinite sets. Then the sequence $(e_i)_{i\in M}$ is not unconditional.
		\end{prop}
		\begin{proof}
			We may, choose for each $n\in \N$, a special sequence $(E^i_k)_{i=1,k=1}^{2,n}$ such that $E^1_k\subset M\cap N_1$ and $E^2_k\subset M\cap N_2$ for every $1\le k\le n$, and apply Proposition	\ref{propunc} to conclude that
			\[
			\sup\Bigg\{ \Big\|\sum_{i\in M}\varepsilon_{i}a_ie_i\Big\|:\varepsilon_i\in\{-1,1\},\Big\|\sum_{i\in M}\varepsilon_{i}a_ie_i\Big\|\le 1\Bigg\}\ge\frac{2n}{5}.
			\]	
			Since $n$ is arbitrary, it follows that $(e_i)_{i\in M}$ is not unconditional.	
		\end{proof}

       The following problem is however open.

        \begin{pr}
            Let $(e^1_n)_n$ and $(e^2_n)_n$ be subsymmetric sequences in a Banach space, i.e. spreading and unconditional. Do there exist $M,L$ infinite subsets of $\N$ such that the sequence $\{e^1_n\}_{n\in M}\cup\{e^2_n\}_{n\in L}$ is unconditional?
        \end{pr}

		Despite the fact that for $(e_n)_{n\in N_1}$ and $(e_n)_{n\in N_2}$  any subsequences fail to form a common unconditional sequence, the following more general result shows that we may find further block subsequences which satisfy this property.

        \begin{prop}
            Let $(x_n)_n$ and $(y_n)_n$ be unconditional sequences in a Banach space $X$. There exist  block sequences $(z_n)_n$ and $(w_n)_n$, of $(x_n)_n$ and  $(y_n)_n$ respectively, such that  $\{z_n\}_{n\in\N}\cup\{w_n\}_{n\in\N}$ is an unconditional sequence.
        \end{prop}
        \begin{proof}
            Assume that there exist two subsequences $(x_n)_{n\in M_1}$ and $(y_n)_{n\in M_2}$ such that $d(S_Z,S_Y)>0$, where $Z=\spn\{x_n\}_{n\in M_1}$ and $Y=\spn\{y_n\}_{n\in M_2}$. Then $Y+Z$ is closed. Hence by the Closed Graph Theorem we have that $Y+Z=Y\oplus Z$ and this yields that $\{x_n \}_{n\in M_1}\cup\{y_n \}_{n\in M_2}$ is unconditional.

            Otherwise, we choose by induction $(z_n)_n$ and $(w_n)_n$ which are normalized blocks of $(x_n)_n$ and $(y_n)_n$ respectively with $\sum_{n=1}^\infty\|z_n-w_n\|<1/(2C)$, where $C$ is the basis constant of $(x_n)_n$, and hence also of $(z_n)_n$. Then, by the principle of small perturbations, it follows that the sequence $\{z_{2n}\}_{n\in\N}\cup\{w_{2n-1}\}_{n\in\N}$ is equivalent to $(z_n)_n$, which is unconditional.
        \end{proof}

        \begin{rem}
        A natural question arising from the previous proposition is whether every space generated by two unconditional sequences is unconditionally saturated. The answer is negative and this follows from a well known more general result. Let $X$ be a Banach space with a Schauder basis $(x_n)_n$, $Y$ be a separable Banach space and $(d_n)_n$ be a dense subset of the unit ball of $Y$. Then the sequences $(x_n)_n$ and $(y_n)_n$ with $y_n=x_n+d_n/2^n$ are equivalent and generate the space $X\oplus Y$. Hence if $(x_n)_n$ is unconditional and $Y$ contains no unconditional sequence we obtain the desired result. We thank Bill Johnson for bringing to our attention this classical argument.
        \end{rem}

\section{Joint Spreading Models}

We introduce the notion of $l$-joint spreading models which is the central concept of this paper. It describes the joint asymptotic behavior of a finite collection of sequences. As it is demonstrated in \cite{AM1}, in certain spaces this behavior may be radically more rich than the one of usual spreading models. It is worth pointing out that spreading models have been tied to the study of bounded linear operators \cite{AM2} and the present paper clarifies that joint spreading models are no exception.

    \begin{dfn}\label{ljsm}

    	Let $l\in\N$, $(x^1_n)_n,\ldots,(x^l_n)_n$ be Schauder basic sequences in a Banach space $(X,\|\cdot\|)$ and $(e^i_n)_{i=1,n\in\N}^l$ be a sequence in a Banach space $(E,\|\cdot\|_\ast)$.
    	
    	Let $M\in[\N]^\infty$. We will say that the $l$-tuple $((x_n^i)_{n\in M})_{i=1}^l$ \textit{generates} $(e^i_n)_{i=1,n\in\N}^l$ as an \textit{l-joint spreading model} if the following is satisfied. There exists a null sequence $(\delta_n)_n$ of positive reals such that for every $k\in\N$, $(a_{ij})_{i=1, j=1}^{l, k}\subset[-1,1]$ and every strict-plegma family $(s_i)_{i=1}^l\in S$-$Plm_l([M]^k)$ with  $M(k)\le s_1(1)$, we have that
    	\[
    	\Bigg| \Big\|\sum_{j=1}^k\sum_{i=1}^l a_{ij}x^i_{s_i(j)}\Big\| - \Big\|\sum_{j=1}^k\sum_{i=1}^l a_{ij}e^i_{j}\Big\|_\ast \Bigg| < \delta_k.
    	\]
    	
    	We will also say that $((x_n^i)_{n})_{i=1}^l$ \textit{admits} $(e^i_n)_{i=1,n\in\N}^l$ as an \textit{$l$-joint spreading model} if there exists $M\in[\N]^\infty$ such that $((x_n^i)_{n\in M})_{i=1}^l$ generates $(e^i_n)_{i=1,n\in\N}^l$.
    	
    	Finally, for a subset $A$ of $X$, we will say that $A$ {\em admits} $(e^i_n)_{i=1,n\in\N}^l$ as an \textit{$l$-joint spreading model} if there exists an $l$-tuple $((x_n^i)_n)_{i=1}^l$ of sequences in $A$ which admits $(e^i_n)_{i=1,n\in\N}^l$ as an \textit{$l$-joint spreading model}.
    \end{dfn}
	
	Notice that for $l=1$, the previous definition recovers the classical Brunel-Sucheston spreading models.

	\begin{rem}
		Let $(x^1_n)_n,\ldots,(x^l_n)_n$ be Schauder basic sequences in a Banach space $(X,\|\cdot\|)$. Let also $M\in[\N]^\infty$ be such that the $l$-tuple $((x_n^i)_{n\in M})_{i=1}^l$ generates the sequence $(e^i_n)_{i=1,n\in\N}^l$ as an $l$-joint spreading model. Then the following hold.
		\begin{enumerate}
			\item[(i)] For every $1\le i\le l$, the sequence $(e_n^i)_n$ is the spreading model admitted by  $(x^i_n)_n$.
			\item[(ii)] The sequence $(e^i_n)_{i=1,n\in\N}^l$ is plegma spreading.  Although $l$-joint spreading models are defined using strict plegma families, these sequences behave in a spreading way that involves plegma families.
			\item[(iii)] For every $M'\in[M]^\infty$ we have that $((x_n^i)_{n\in {M^\prime}})_{i=1}^l$ generates $(e^i_n)_{i=1,n\in\N}^l$ as an $l$-joint spreading model.
			\item[(iv)] For every $(\delta_n)_n$ null sequence of positive reals there exists $M'\in[M]^\infty$ such that $((x_n^i)_{n\in M'})_{i=1}^l$ generates $(e^i_n)_{i=1,n\in\N}^l$ as an $l$-joint spreading model with respect to $(\delta_n)_n$.
			\item[(v)] If $|\|\cdot\||$ is an equivalent norm on $X$, then every $l$-joint spreading model admitted by $(X,\|\cdot\|)$ is equivalent to an $l$-joint spreading model admitted by $(X,|\|\cdot\||)$.
		\end{enumerate}
	\end{rem}

	Next we prove a Brunel-Sucheston type result for $l$-joint spreading models.
	
	\begin{thm}
	\label{joint Brunel-Sucheston}
		Let $l\in\N$ and $X$ be a Banach space. Then every $l$-tuple of Schauder basic sequences in $X$ admits an $l$-joint spreading model.
	\end{thm}
	
	First we describe the following combinatorial lemma which will yield the theorem.

	\begin{lem}\label{basicpropl}
		Let $(x^1_n)_n,\ldots,(x^l_n)_n$ be bounded sequences in a Banach space $X$ and $(\delta_n)_n$ be a decreasing null sequence of positive real numbers. Then for every $M\in[\N]^\infty$, there exists $L\in[M]^\infty$ such that
		\[
		\Bigg| \Big\|\sum_{j=1}^k\sum_{i=1}^l a_{ij}x^i_{s_i(j)}\Big\| - \Big\|\sum_{j=1}^k\sum_{i=1}^l a_{ij}x^i_{t_i(j)}\Big\| \Bigg| < \delta_k
		\]
		for every $k\in\N$, $(a_{ij})_{i=1, j=1}^{l, k}\subset[-1,1]$ and $(s_i)_{i=1}^l,(t_i)_{i=1}^l\in S\-Plm_l([L]^k)$ with $s_1(1),t_1(1)\ge L(k)$.
	\end{lem}
	\begin{proof}
		Let $C>0$ be such that $\|x^i_n\|<C$ for all $i=1,\ldots,l$ and $n\in\N$, and set $L_0=M$. We will construct, by induction, a decreasing sequence $(L_k)_{k\ge0}$ such that for every $k\in\N$, $(a_{ij})_{i=1,j=1}^{l,k}\subset [-1,1]$ and $(s_i)_{i=1}^l,(t_i)_{i=1}^l\in S\-Plm_l([L_k]^k)$,
		\[
		\Bigg| \Big\|\sum_{j=1}^k\sum_{i=1}^l a_{ij}x^i_{s_i(j)}\Big\| - \Big\|\sum_{j=1}^k\sum_{i=1}^l a_{ij}x^i_{t_i(j)}\Big\| \Bigg| < \delta_k.
		\]
		Suppose that $L_0,\ldots,L_{k-1}$ have been chosen, for some $k\in\N$. Let $A$ be a finite $\frac{\delta_k}{4klC}$-net of $[-1,1]$ and $B$ be a partition of $[0,lC]$ consisting of disjoint intervals with length less than $\frac{\delta_k}{4}$. We set $\mathcal{F}=\{f:A^{kl}\to B\}$ and for $f\in\mathcal{F}$
		\[\begin{split}
		P_f=\Big\{(s_i)_{i=1}^l\in S\-Plm_l([L_{k-1}]^{k}):\Big\|\sum_{j=1}^{k}\sum_{i=1}^l a_{ij}x_{s_i(j)}^i\Big\|\in f(a),\\ \text{for all } a=((a_{ij})_{j=1}^k)_{i=1}^l\in A^{kl} \Big\}.
		\end{split}\]
		Then $S\-Plm_l([L_{k-1}]^k)=\cup_{f\in\mathcal{F}}P_f$ and by Theorem \ref{plmramsey} there exist $L_k\in[\N]^\infty$ and $f\in\mathcal{F}$ such that $S\-Plm_l([L_k]^k)\subset P_f$. Hence for every $(a_{ij})_{i=1,j=1}^{l,k}\subset A$ and $(s_i)_{i=1}^l,(t_i)_{i=1}^l\in S\-Plm_l([L_k]^k)$, we have that
		\[
		\Bigg| \Big\|\sum_{j=1}^n\sum_{i=1}^l a_{ij}x^i_{s_i(j)}\Big\| - \Big\|\sum_{j=1}^n\sum_{i=1}^l a_{ij}x^i_{t_i(j)}\Big\| \Bigg| < \frac{\delta_k}{4}.
		\]
		Since $A$ is a net of $[-1,1]$, it is easy to see that $L_{k}$ is as desired. Finally, choosing $L$ to be a diagonalization of $(L_k)_k$ completes the proof.
	\end{proof}
	
    \begin{proof}[Proof of Theorem \ref{joint Brunel-Sucheston}]
        Let $(x^1_n)_n,\ldots,(x^l_n)_n$ be Schauder basic sequences. Observe that the previous lemma yields an infinite subset $L$ of $\N$ such that for every $k\in\N$ and $(a_{ij})_{i=1, j=1}^{l, k}\subset[-1,1]$ and every sequence $((s^n_i)_{i=1}^l)_n$ of strict plegma families in $[L]^k$ with $\lim s_1^n(1)=\infty$, the sequence $(\|\sum_{j=1,i=1}^{k,l}a_{ij}x^i_{s^n_i(j)}\|)_n$ is Cauchy, with the limit independent from the choice of the sequence $((s^n_i)_{i=1}^l)_n$.

        Denote by $(e_n)_n$ the usual basis of $c_{00}(\N)$ and for every $i=1,\ldots,l$ and $n\in\N$, set $e^i_n=e_{k
        (i,n)}$, where $k(i,n)=(n-1)l+in$. Using the above observation, we define a seminorm $\|\cdot\|_*$ on $c_{00}(\N)$ as follows:
        \[\Big\|\sum_{j=1}^k\sum_{i=1}^l a_{ij}e^{i}_j\Big\|_*=\lim_n \Big\|\sum_{j=1}^k\sum_{i=1}^l a_{ij}x^{i}_{s^n_i(j)}\Big\|\]
        where $((s^n_i)_{i=1}^l)_n\subset Plm_l([L]^k)$ with $s^n_1(1)\to\infty$ and $(a_{ij})_{i=1,j=1}^{l,k}\in[-1,1]$.  Since each $(x^i_n)_n$ is a Schauder basic sequence and hence does not contain any norm convergent subsequences, a modification of \cite[Proposition 1.B.2]{BL} yields that $\|\cdot\|_*$ is a norm. Denote by $E$ the completion of $c_{00}(\N)$ with respect to this norm and notice that the $l$-tuple $((x^i_n)_{n\in L}
        )_{i=1}^l$ generates the sequence $(e^i_n)_{i=1,n\in \N}^l$ in $E$ as an $l$-joint spreading model.
    \end{proof}

	The following proposition is an immediate consequence of the definition of $l$-joint spreading models and Theorem \ref{thmbsc}.

\begin{prop}
	Let $(x^1_n)_n,\ldots,(x^l_n)_n$ be Schauder basic sequences in a Banach space $X$ such that each one is either weakly null, equivalent to the basis of $\ell_1$, or non-trivial weak-Cauchy and the $l$-tuple $(x^i_n)_{i=1,n\in\N}^l$ generates the sequence $(e^i_n)_{i=1,n\in\N}^l$ as an $l$-joint spreading model. Let $I\subset\{1,\ldots,l\}$ be such that $(x^i_n)_n$ is a non-trivial weak-Cauchy sequence with $w^*\-\lim x^i_n=x^{**}_i$ for every $i\in I$ and set $F=\spn\{ x^{**}_i\}_{i\in I}$. If $X\cap F=\{0\}$, then $(e^i_n)_{i=1,n\in\N}^l$ is a Schauder basic sequence.
\end{prop}

The next example demonstrates that the opposite statement of the above is not always true, that is, $(e^i_n)_{i=1,n\in\N}^l$ may be Schauder basic whereas $X\cap F\neq\{0\}$.

\begin{exmp}
	Define the norm on $c_{00}(\N)$ given by $\|x\|=\sup \sum_{i=1}^n |\sum_{k\in I_i}x(k)|$, where the supremum is taken over all finite collections $I_1,\ldots,I_n$ of successive intervals of natural numbers with $n\leq \min I_1$. Denote by $X$ the completion of  $c_{00}(\N)$ with respect to this norm. Then the usual basis $(e_n)_n$ is a non-trivial weak-Cauchy sequence that generates a spreading model equivalent to the unit vector basis of $\ell_1$.
	Consider the sequences $e^1_n=e_{2n-1}-e_1$ and $e^2_n=e_{2n}+e_1$. As follows from Proposition \ref{propntwc}, none of their subsequences form a common Schauder basic sequence whereas the $l$-joint spreading model admitted by $(e^i_n)_{i=1,n\in\N}^2$ is equivalent to the unit vector basis of $\ell_1$.
\end{exmp}
	
	Recall that spreading models generated by weakly null sequences are unconditional. This is extended to joint spreading models by an easy modification of the classical case \cite[Proposition 5.1]{BL}.
	
	\begin{prop}
		Let $(x^1_n)_n,\ldots,(x^l_n)_n$ be weakly null Schauder basic sequences in a Banach space $X$ that admit $(e_n^i)_{i=1,n\in\N}^l$ as an $l$-joint spreading model. Then $(e_n^i)_{i=1,n\in\N}^l$ is 1-suppression unconditional and hence for every $\varepsilon>0$ and $k\in\N$ there exists $n\in\N$ such that for every $(s_i)_{i=1}^l\in S\-Plm_l([\N]^k)$ with $n\le s_1(1)$, the sequence $(x_{s_i(j)}^i)_{i=1,j=1}^{l,k}$ is $(1+\varepsilon)$-suppression unconditional.
	\end{prop}
	
\begin{rem}
		The notion of $l$-joint spreading models can be naturally extended, by a diagonalization argument, to $\omega$-joint spreading models which are $\omega$-plegma spreading sequences and are generated by countably many Schauder basic sequences.
	\end{rem}

	\section{spaces with unique joint spreading model}
	In this section we study spaces that admit a uniformly unique joint spreading model with respect to certain families of sequences. In the first part we prove the uniform uniqueness of $l$-joint spreading models for the classical $\ell_p$ and $c_0$ spaces. Then we pass to Asymptotic $\ell_p$ spaces \cite{MMT} and in the last part we study this problem for the James Tree space.
		
	\begin{dfn}\label{defuniqe}
		Let $\mathscr{F}$ be a family of normalized sequences in a Banach space $X$. We will say that $X$ admits a \textit{uniformly unique $l$-joint spreading model} with respect to $\mathscr{F}$ if there exists $K>0$ such that, for every $l\in\N$, any two  $l$-joint spreading models generated by sequences from $\mathscr{F}$ are $K$-equivalent.
	\end{dfn}

	\begin{rem}
		Let $\mathscr{F}$ be a family of normalized sequences in a Banach space $X$ such that, for some $l\in\N$, there exists $K_l>0$ such that any two $l$-joint spreading models generated by sequences from $\mathscr{F}$ are $K_l$-equivalent.
		\begin{enumerate}
			\item[(i)]	For every $l'<l$, there exists $K_{l'}\le K_l$ such that any two $l'$-joint spreading models generated by sequences from $\mathscr{F}$ are $K_{l'}$-equivalent.
			\item[(ii)] The space $X$ may fail to admit a uniformly unique $l$-joint spreading model with respect to $\mathscr{F}$. For examples of such spaces see \cite{AM1} and the space in Definition \ref{definition fails uals}.
		\end{enumerate}
	\end{rem}

	\begin{notn}
		For a Banach space $X$,  we will denote by $\mathscr{F}(X)$ the set of all normalized Schauder basic sequences in $X$, by $\mathscr{F}_0(X)$ its subset consisting of the sequences that are weakly null and by $\mathscr{F}_C(X)$ the set of all normalized $C$-Schauder basic sequences in $X$. Finally, if $X$ has a Schauder basis, we shall denote by $\mathscr{F}_{b}(X)$ the set of all normalized block sequences in $X$.
	\end{notn}

	Next we present some examples of spaces admitting a uniformly unique $l$-joint spreading model. We start with the classical sequence spaces $\ell_p$ and $c_0$.

	\begin{prop} Each  of the following spaces admits a uniformly unique $l$-joint spreading model which is in fact equivalent to its unit vector basis.
	\begin{enumerate}
			\item[(i)] The spaces $\ell_p$, for $1<p<\infty$, with respect to $\mathscr{F}(\ell_p)$.
			\item[(ii)] The space $\ell_1$ with respect to $\mathscr{F}_b(\ell_1)$, but not with respect to $\mathscr{F}(\ell_1)$.
			\item[(iii)] The space $c_0$ with respect to $\mathscr{F}_0(c_0)$, but not with respect to $\mathscr{F}(c_0)$.	 				
		\end{enumerate}
	\end{prop}

	It is immediate to see that the above remain valid for the spaces $\ell_p(\Gamma)$, for $1\le p<\infty$, and $c_0(\Gamma)$, for any infinite set $\Gamma$.

	\begin{rem}
	For $C\geq 1$, the space $\ell_1$ admits a uniformly unique $l$-joint spreading model with respect to $\mathscr{F}_C(\ell_1)$. To see this, let $(x_n)_n$ be an arbitrary normalized $C$-Schauder basic sequence in $\ell_1$. Passing, if it necessary, to a subsequence $(x_n)_n$ has a point-wise limit $x_0$ in $\ell_1$. That is, $\lim_ne_i^*(x_n) = e_i^*(x_0)$ for all $i\in\mathbb{N}$. Also, if we set $z_n = x_n - x_0$, then $\lim_n\|z_n\|= \lambda$ may be assumed to exist. It follows that $\|x_0\| + \lambda = 1$ and that $0<\lambda\leq 1$. We can also assume that $(\lambda^{-1}z_n)_{n\geq n_0}$ is $(1+1/n_0)$-equivalent to the usual basis of $\ell_1$. We conclude that for any $M\in\mathbb{N}$:
	\begin{equation*}
	\begin{split}
	M(1-\lambda) &= M\|x_0\|\leq\lim_k\Big\|\sum_{n=1}^Mx_{n+k}\Big\| \leq C\lim_k \Big\|\sum_{n=1}^Mx_{n+k} - \sum_{n=M+1}^{2M}x_{n+k}\Big\|\\
	&=C\lim_k \Big\|\sum_{n=1}^Mz_{n+k} - \sum_{n=M+1}^{2M}z_{n+k}\Big\| = C2M\lambda.
	\end{split}
	\end{equation*}
	Therefore, $\lambda \geq 1/(2C+1)$. Hence if $(x_n^i)_n$, $1\leq i \leq l$, is an $l$-tuple of $C$-Schauder basic sequences we may pass to subsequences so that $(x_n^i)_n$ converges point-wise to some $x_0^i$ for $1\leq i\leq l$ and if we set $(z_n^i)_n = (x_n^i- x_0^i)$ for $1\leq i\leq l$ then these sequences are point-wise null and they are all bounded bellow from $1/(2C+1)$. We may then conclude that for any $\varepsilon>0$, passing to appropriate subsequences,  $(x_n^i)_n$, $1\leq i\leq m$, is jointly $(2C+1+\varepsilon)$-equivalent to the unit vector basis of $\ell_1$.
\end{rem}

\begin{rem}
	Although for any $C\geq1$ the space $\ell_1$ admits a uniformly unique $l$-joint spreading model with respect to $\mathscr{F}_C(\ell_1)$, this is no longer true for spaces with the Schur property. For example, define for each $n\in\mathbb{N}$ the norm $\|\cdot\|_n$ on $\ell_1$ given by $\|x\|_n = \max\{\|x\|_{\ell_2},n^{-1}\|x\|_{\ell_1}\}$. Set $X = (\sum_n\oplus X_n)_{\ell_1}$, where $X_n=(\ell_1,\|\cdot\|_n)$, which has the Schur property. Although every spreading model of this space is equivalent to the unit vector basis of $\ell_1$, this does not happen for a uniform constant.
\end{rem}

	Another example of spaces admitting a uniformly unique joint spreading model are asymptotic $\ell_p$ spaces. We start with their definition.

	\begin{dfn}[\cite{MT}]
	\label{asellp MT}
	A Banach space $X$ with a normalized Schauder basis is {\em asymptotic} $\ell_p$ (resp. {\em asymptotic} $c_0$) if there exists $C>0$ such that any finite sequence $(x_i)_{i=1}^n$ of normalized vectors in $X$ with $n<\supp(x_1)<\cdots<\supp(x_n)$ is $C$-equivalent to the standard basis of $\ell^n_p$ (resp. $c_0^n$), for $1\le p<\infty$.
\end{dfn}

The classical examples of asymptotic $\ell_p$ spaces are Tsirelson's original space \cite{Tsirelson} and its p-convexifications \cite{FJ}. The next proposition follows easily from the above definition and the fact that an asymptotic $\ell_p$ space, for $1<p<\infty$, is reflexive.

\begin{prop}
	Every asymptotic $\ell_p$ or asymptotic $c_0$ space $X$ admits a uniformly unique $l$-joint spreading model with respect to $\mathscr{F}_b(X)$. Moreover, every asymptotic $\ell_p$ space, for $1<p<\infty$, admits a uniformly unique $l$-joint spreading model with respect to $\mathscr{F}(X)$.
\end{prop}

 The following proposition concerns spaces with uniformly unique joint spreading models with respect to families that have certain stability properties. The joint spreading models of such spaces are unconditional and sometimes even equivalent to some $\ell_p$ or $c_0$. Families with such properties play an important role in the study of the UALS  in the next section.

\begin{prop}
	\label{kind of krivine sort of}
	Let $X$ be a Banach space that admits a $K$-uniformly unique $l$-joint spreading model with respect to a family of normalized Schauder basic sequences $\mathscr{F}$. Assume that $\mathscr{F}$ is such that:

	\begin{itemize}
		\item[(a)] If $(x_j)_j$ in $\mathscr{F}$ then any subsequence of $(x_j)_j$ is in $\mathscr{F}$.
		\item[(b)] If $(x_j)_j$ is in $\mathscr{F}$ then there exists an infinite subset $L = \{l_i:i\in\mathbb{N}\}$ of $\mathbb{N}$ such that if $z_i = \|x_{l_{2i-1}} - x_{l_{2i}}\|^{-1}(x_{l_{2i-1}} - x_{l_{2i}})$, for $i\in\mathbb{N}$, then the sequence $(z_i)_i$ is in $\mathscr{F}$.
		\item[(c)]

		If $(x_j)_j$ is in $\mathscr{F}$ and $(\lambda_i)_{i=1}^N$ is a finite sequence of scalars, not all of which are zero, then there exists an infinite subset $L = \{l_i:i\in\mathbb{N}\}$ of $\mathbb{N}$ such that if for $n\in\N$:
		
		\[z_n = \Big\|\sum_{i=1}^N\lambda_i x_{l_{N(n-1)+i}}\Big\|^{-1}\Bigg(\sum_{i=1}^N\lambda_i x_{l_{N(n-1)+i}}\Bigg),\]
		
		\noindent then the sequence $(z_i)_i$ is in $\mathscr{F}$.

	\end{itemize}
	The following statements hold.
	\begin{itemize}
		
		\item[(i)] If $\mathscr{F}$ satisfies (a) then every $l$-joint spreading model admitted by an $l$-tuple of sequences in $\mathscr{F}$ is spreading when enumerated with the natural plegma order and it is $K$-equivalent to the spreading model generated by any sequence in $\mathscr{F}$.
		
		\item[(ii)] If $\mathscr{F}$ satisfies (a) and (b) then every $l$-joint spreading model admitted by sequences in $\mathscr{F}$ is $K$-suppression unconditional.
		
		\item[(iii)] If $\mathscr{F}$ satisfies (a) and (c) then every $l$-joint spreading model admitted by sequences in $\mathscr{F}$ is  $K$-equivalent to the unit vector basis of $\ell_p$ (for $p=\infty$ we mean the unit vector basis of $c_0$).
	\end{itemize}
\end{prop}

\begin{proof}
	The first item follows by taking an arbitrary sequence $(x_j)_j$ in $\mathscr{F}$, passing to a subsequence that generates some spreading model $(e_i)_i$, and then taking disjointly indexed subsequences $(x_i^1)_i,\ldots,(x_i^l)_i$, which by assumption are all in $\mathscr{F}$. Clearly, they generate an $l$-joint spreading model that is isometrically equivalent to $(e_i)_i$. We conclude that any $l$-joint spreading model generated by an $l$-tuple of sequences in $\mathscr{F}$ is $K$-equivalent to $(e_i)_i$, when endowed with the natural plegma order.
	
	For the second item it is sufficient, by (i), to show that any spreading model admitted by a sequence in $\mathscr{F}$ has the desired property. Pick an arbitrary sequence $(x_j)_j$ in $\mathscr{F}$ which by (a) may be chosen to generate some spreading model $(e_j)_j$. Applying (b) to $(x_j)_j$ we can deduce that there is a sequence in $\mathscr{F}$ that generates as a spreading model the sequence $(\|e_{2j-1}-e_{2j}\|^{-1}(e_{2j-1} - e_{2j}))_j$, which by \cite[Proposition 4.3]{BL} is 1-suppression unconditional. Observe that any sequence that is $K$-equivalent to a 1-suppression unconditional sequence is $K$-suppression unconditional.
	
	Assume now that (a) and (c) hold. Clearly, (a) and (c) together imply (b) so we may pick up where we left off, namely having at hand a sequence $(x_j)_j$ in $\mathscr{F}$ that generates a spreading model that is $1$-suppression unconditional. By \cite[Paragraph 1.6.3]{MMT}, as a direct application of Krivine's Theorem \cite{Krivine} \cite{Lember}, for any $m\in\mathbb{N}$ and $\varepsilon>0$, we may choose scalars $\lambda_1,\ldots,\lambda_N$ such that any $m$ terms of the resulting sequence $(z_n)_n$ are $(1+\varepsilon)$-equivalent to the unit vector basis of $\ell_p^m$, for some $1\leq p\leq \infty$. This means that there exists a constant $K$ such that, for any $m\in\mathbb{N}$, there exists $1\leq p_m\leq\infty$ such that the first $m$ terms of any spreading model generated by a sequence from $\mathscr{F}$ are $K$-equivalent to the unit vector basis of $\ell_{p_m}$. Taking a limit point of $(p_m)_m$ yields the conclusion.
\end{proof}

	\subsection{Coordinate-free asymptotic $\ell_p$ spaces (Asymptotic $\ell_p$ spaces)}
Notice that Definition \ref{asellp MT} of an asymptotic $\ell_p$ space from \cite{MT} depends on the Schauder basis of $X$ and not only on $X$. A coordinate-free version of this definition can be found in \cite[Subsection 1.7]{MMT} and it is based on a game of two players (S) and (V). In each turn of the game player (S) chooses a closed finite codimensional subspace $Y$ of $X$ and player (V) chooses a normalized vector $y\in Y$. A Banach space $X$ is called {\em Asymptotic} $\ell_p$ if there exists a constant $C$ such that, for every $n\in\N$, player (S) has a wining strategy in the game $G(p,n,C)$, that is to force in $n$ steps player (V) to choose a sequence $(y_i)_{i=1}^n$ that is $C$-equivalent to the unit vector basis of $\ell_p^n$ (or $c_0^n$ for $p=\infty$). We point out that the original formulation of this property is different. The equivalence of the original definition with this more convenient version follows from \cite[Subsection 1.5]{MMT}.

Next we show that for a separable Asymptotic $\ell_p$ space $X$, for $1\le p\le \infty$, there exists a certain family of sequences in $X$, described in Proposition \ref{unique joint if these go to zero}, with respect to which $X$ admits a uniformly unique $l$-joint spreading model. This family has certain properties that $\mathscr{F}_0(X)$ fails when $X$ contains $\ell_1$ and this result will be used in the next section to prove that an Asymptotic $\ell_1$ space satisfies the UALS. We start with the following lemma.

\begin{lem}\label{lemma As lp}
	Let $X$ be a separable $C$-Asymptotic $\ell_p$ space, for $1\le p\le \infty$. Then there exists a countable collection $\mathscr{Y}$ of finite codimensional subspaces of $X$ such that, for every $\varepsilon>0$ and $n\in\N$, player (S) has a winning strategy in the game $G(p,n,C+\varepsilon)$ when choosing finite codimensional subspaces from $\mathscr{Y}$.
\end{lem}
\begin{proof}
If $X$ is $C$-asymptotic $\ell_p$ in the sense described in the paragraph above we shall, for fixed $n\in\mathbb{N}$, assume the role of player (V) and let player (S) follow a winning strategy during a multitude of outcomes in a game of $G(p,n,C)$. More accurately, we will describe how to define a collection of vectors of $X$ of the form $\{x^n_F: \emptyset \neq F\in[\mathbb{N}]^{\leq n}\}$ and a collection of closed finite codimensional subspaces of $X$ of the form $\{Y^n_F: F\in[\mathbb{N}]^{\leq n-1}\}$ that satisfy:
\begin{itemize}

\item[(i)] for all $F\in[\mathbb{N}]^{\leq n-1}$, the norm-closure  of $\{x^n_{F\cup\{i\}}:i>\max(F)\}$  is the unit sphere of $Y^n_F$ (here, $\max(\emptyset) = 0$) and
\item[(ii)] for every $\{k_1,\ldots,k_m\}$, in $[\mathbb{N}]^{\leq n}$ we have that
\[\left(Y^n_\emptyset ,x^n_{\{k_1\}}\right),\;\left(Y^n_{\{k_1\}},x^n_{\{k_1,k_2\}}\right),\ldots,\left(Y^n_{\{k_1,\ldots,k_{m-1}\}},x^n_{\{k_1,\ldots,k_m\}}\right)\]
is the outcome of a game of $G(p,n,C)$ after $m$ rounds in which player (S) has followed a winning strategy.

\end{itemize}
Player (S) initiates and he chooses a finite codimensional subspace $Y^n_\emptyset$. As player (V), we choose a dense subset $\{x^n_{\{i\}}: i\in\mathbb{N}\}$ of the unit sphere of $Y^n_\emptyset$. If for some $1\leq m<n$ we have chosen $\{x^n_F: \emptyset\neq F\in[\mathbb{N}]^{\leq m}\}$ and $\{Y^n_F: F\in[\mathbb{N}]^{\leq m-1}\}$, we complete the inductive step as follows: for every $F = \{k_1<\cdots<k_m\}$, by assumption (ii), player (S) may continue following a winning strategy and choose a closed finite codimensional subspace $Y$ such that for every unit vector $y\in Y$ the sequence $(x^n_{k_i})_{i=1}^m\phantom{}^\frown(y)$ is $C$-equivalent to the unit vector basis of $\ell_p^{m+1}$. Set $Y^n_F = Y$ and then, for all $i>\max(F)$, choose a unit vector $x^n_{F\cup\{i\}}$ in $Y_F$ such that the set $\{x^n_{F\cup\{i\}}:i>\max(F)\}$ is dense in the unit sphere of $Y^n_F$.

Set $\mathscr{Y}=\{ Y^n_F:n\in\N, F\in[\mathbb{N}]^{\leq n-1}\}$ and fix $\varepsilon>0$ and $n\in\N$. Let us also take $\tilde\varepsilon>0$ to be determined later. We will describe a winning strategy for player (S) in the game $G(p,n,C+\varepsilon)$, choosing finite codimensional subspaces from $\mathscr{Y}$. Player (S) initiates the game and chooses the subspace $Y_1=Y^n_\emptyset$ and player (V) chooses an arbitrary normalized vector $y_1$ from $Y_1$. Before the next turn, player (S) also chooses $k_1\in\N$ such that $\|y_1-x^n_{k_1}\|<\tilde\varepsilon/n$. Let $(Y_1,y_1),\ldots(Y_m,y_m)$ be the outcome in the first $m$ turns of the game, for $1\le m<n$, while player (S) has also chosen $k_1,\ldots,k_{m}\in\N$ with $\|y_i-x^n_{k_i}\|<\tilde\varepsilon/n$, for $1\le i\le m$. In the next turn of the game, player (S) chooses the subspace $Y_{m+1}=Y^n_{\{k_1,\ldots,k_{m}\}}$ and $k_{m+1}\in\N$ such that $\|y_{m+1}-x^n_{k_{m+1}}\|<\tilde\varepsilon/n$, where $y_{m+1}$ is the vector player (V) choose from $Y_{m+1}$. Hence if $(Y_1,y_1),\ldots(Y_n,y_n)$ is the final outcome of the game, notice that the sequence $(x^n_{k_i})_{i=1}^n$ is $C$-equivalent to the unit vector basis of $\ell_p^n$ and $\|y_i-x^n_{k_i}\|<\tilde\varepsilon/n$ for all $1\le i\le n$. If we take $1\leq A,B$ with $AB \leq C$ such that $(1/A)\leq (\sum_{i=1}^n|a_i|^p)^{1/p}\leq \|\sum_{i=1}^na_ix_{k_i}^n\| \leq B(\sum_{i=1}^n|a_i|^p)^{1/p}$ then we conclude that $(y_i)_{i=1}^n$ is $C(1+\tilde\varepsilon)/(1-\tilde\varepsilon C)$ equivalent to the unit vector basis of $\ell_p^n$. For $\tilde\varepsilon$ sufficiently small we deduce the conclusion.
\end{proof}

\begin{prop}
	\label{unique joint if these go to zero}
	Let $X$ be a separable $C$-Asymptotic $\ell_p$ space, for $1\le p\le \infty$. There exists a countable subset $\mathscr{A}$ of $X^*$ such that, if
	\[\mathscr{F}_{0,\mathscr{A}} = \left\{(x_n)_n:\;(x_n)_n\text{ is normalized and }\lim_nf(x_n) = 0\text{ for all }f\in\mathscr{A}\right\},\]
	then $X$ admits $\ell_p$ as a $C^2$-uniformly unique $l$-joint spreading model with respect to the family $\mathscr{F}_{0,\mathscr{A}}$.
\end{prop}

\begin{proof}
	Let $\mathscr{Y}$ be as in Lemma \ref{lemma As lp} and, for each $Y\in\mathscr{Y}$, choose a finite subset $f_1^{Y},\ldots,f^{Y}_{k_{Y}}$ of $X^*$ such that $Y= \cap_{i=1}^{k_{Y}}\mathrm{ker}(f^{Y}_i)$ and set $\mathscr{A} = \cup_{Y\in\mathscr{Y}}\{f_1^{Y},\ldots,f^{Y}_{k_{Y}}\}$ which is a countable set. We will show that it is the desired one. To that end, let $l\in\mathbb{N}$ and $(x_n^1)_n,\ldots,(x_n^l)_n$ be sequences in $\mathscr{F}_{0,\mathscr{A}}$ generating an $l$-joint spreading model $(e^i_n)_{i=1,n}^l$. Let $k\in\mathbb{N}$, we will show that $(e^i_j)_{i=1,j=1}^{l,k}$ is $C$-equivalent to the unit vector basis of $\ell_p^{lk}$.  Set $m = lk$, fix $\varepsilon >0$ and, using Lemma \ref{bslem}, choose by induction normalized vectors $y_1,\ldots,y_m$ and $(s_i)_{i=1}^l$ in $S\hbox{-}Plm_l([\mathbb{N}]^k)$ such that
	\begin{itemize}
		\item[(i)] $(Y_1,y_1),\ldots,(Y_m,y_m)$ is the outcome of the game $G(m,p,C+\varepsilon)$.
		\item[(ii)] $Y_1,\ldots,Y_m\in\mathscr{Y}$.
		\item[(iii)] For $1\leq i\leq l$ and $1\leq j\leq k$, if we take $n(i,j) = (i-1)k + j$ then  $\|y_{n(i,j)} - x^i_{s_i(j)}\| \leq \varepsilon/m$.
	\end{itemize}  It follows that for any scalars $(a_{ij})_{i=1,j=1}^{l,k}$ with $|a_{ij}|\le 1$, we have that
	\[
	\Bigg| \Big\|\sum_{j=1}^k\sum_{i=1}^l a_{ij}x^i_{s_i(j)}\Big\| - \Big\|\sum_{j=1}^k\sum_{i=1}^l a_{ij}y_{n(i,j)}\Big\| \Bigg| < \varepsilon.
	\]
	As $(y_i)_{i=1}^m$ is $(C+\varepsilon)$-equivalent to the unit vector basis of $\ell_p^m$ the conclusion follows.
\end{proof}

The following is an immediate corollary of the above. In the general case, all Asymptotic $\ell_p$ spaces admit a uniformly unique joint spreading model with respect to the family of normalized weakly null Schauder basic sequences.

\begin{cor}\label{As lp with respect to F_0(X)}
Every Asymptotic $\ell_p$ space $X$, for $1\le p\le \infty$, admits a uniformly unique $l$-joint spreading model with respect to $\mathscr{F}_0(X)$. Moreover, every $l$-joint spreading model generated by a sequence from this family is equivalent to the unit vector basis of $\ell_p$ (or of $c_0$ if $p=\infty$).
\end{cor}

	\subsection{James Tree Space}
	We show that the James Tree space $JT$ admits a uniformly unique joint spreading model with respect to $\mathscr{F}_0(JT)$. This is however not true for joint spreading models with respect to $\mathscr{F}(JT)$ or $\mathscr{F}_b(JT)$.
		\begin{notn}
			We denote by $\dt$ the \textit{dyadic tree}, i.e. $\dt=\{0,1\}^{<\infty}$, ordered by the initial part order. We will use $S$ to denote segments of $\dt$ and $B$ to denote branches. For $m<n$, the \textit{band} $Q_{[m,n]}$ is the set $\{s\in\dt:m\le|s|\le n\}$. We set $c_{00}(\dt)$ to be the linear space of all eventually zero sequences $x:\dt\to\R$. For a segment $S$ of $\dt$ we denote by $S^*$ the linear functional on $c_{00}(\dt)$ defined as $S^*(x)=\sum_{s\in S}x(s)$.
		\end{notn}
		
		\begin{dfn}[\cite{J2}]
			On $c_{00}(\dt)$ we define the following norm
			\[\begin{split}
			\|x\|_{JT}=\sup\Bigg(\sum_{i=1}^n\Big(\sum_{s\in S_i}x(s)\Big)^2\Bigg)^{\frac{1}{2}}
			\end{split}\]
			where the supremum is taken over all finite  collections $S_1,\ldots,S_n$ of pairwise disjoint segments. The \textit{James Tree space}, denoted by $JT$, is the completion of $c_{00}(\dt)$ with respect to the above norm.
		\end{dfn}
		
		\begin{rems}\label{remnormjt} Let us observe, for later  use, that
			\begin{enumerate}
			\item[(i)]	The following set is norming for $JT$:
			\[\begin{split}
			W=\Bigg\{\sum_{i=1}^nb_iS^*_i:n\in\N,\sum_{i=1}^nb_i^2\le1, \{S_i\}_{i=1}^n\;\textnormal{pairwise disjoint segments} \Bigg\}.
			\end{split}\]
			\item[(ii)] Let $\varepsilon>0$, $x\in JT$ with $\|x\|=1$ and $\{S_i\}_{i\in I}$ be pairwise disjoint segments with the property that $|S^*_i(x)|\ge\varepsilon$ for every $i\in I$. Then $\#I\le{1}/{\varepsilon^2}$.			
			\item[(iii)] Let $S_1,\ldots,S_n$ be pairwise disjoint segments and $b_1,\ldots,b_n\in \R$, then
			\[
			\Big\|\sum_{i=1}^nb_iS^*_i\Big\|^2\le \sum_{i=1}^nb_i^2.
			\]
			\end{enumerate}
			
		\end{rems}
		
		We will prove the following theorem.
		
		\begin{thm}\label{jtunique}
			The space $JT$ admits a uniformly unique $l$-joint spreading model with respect to $\mathscr{F}_0(JT)$ and every $l$-joint spreading model generated by sequences from this family is $\sqrt{2}$-equivalent to the unit vector basis of $\ell_2$.
		\end{thm}
			
			The proof of the theorem is a variant of the well known result due to I. Ameniya and T. Ito \cite{AI}, that every normalized weakly null sequence in $JT$ contains a subsequence which is $2$-equivalent to the usual basis of $\ell_2$. From this follows that every spreading model generated by a normalized weakly null sequence is $2$-equivalent to the unit vector basis of $\ell_2$. Our approach yields that every $l$-joint spreading model generated by sequences from $\mathscr{F}_0(JT)$ is equivalent to the unit vector basis of $\ell_2$ with equivalence constant $\sqrt{2}$, which as mentioned in \cite{HB}, \cite{Be} is the best possible.
			
		As a consequence of the fact that James space (see Definition \ref{james space}) is isometric to a subspace of $JT$ it follows that $J$ also admits a uniformly unique $l$-joint spreading model with respect to $\mathscr{F}_0(J)$.
		
		We break up the proof of the theorem into several lemmas and we start with the following Ramsey type result.

		\begin{dfn}\label{w0lbf}
		\label{dfnwnlbf}
		Let $(Q_{[p_n,q_n]})_n$ be successive bands in $\dt$ and $(F_n)_n$ be a sequence of finite subsets of $JT$. We will say that $(F_n)_n$ is a \textit{weakly null level block family} with respect to $(Q_{[p_n,q_n]})_n$ if the following hold.
		\begin{enumerate}
			\item[(i)] $\supp(x)\subset Q_{[p_n,q_n]}$ and $\|x\|=1$ for every $n\in\N$ and $x\in F_n$.
			\item[(ii)] The sequence $(x_n)_n$ is weakly null for any choice of $x_n\in F_n$.
		\end{enumerate}
		\end{dfn}
	
		\begin{lem}\label{jtramsey}
		Let $(F_n)_n$ be a weakly null level block family with $\sup_n\#F_n<\infty$. Then, for every $\varepsilon>0$, there exists an $L\in[\N]^\infty$ such that for every initial segment $S$ there exists at most one $n\in L$ with the property that $|S^*(x)|\ge\varepsilon$ for some $x\in F_n$.
		\end{lem}
\begin{proof}
	If the conclusion is false, using Ramsey's theorem from \cite{Ramsey}, we may assume that there exists $L\in[\N]^\infty$ such that, for every pair $m<n$ in $L$, there exist an initial segment $S_{m,n}$ and $x\in F_m$, $y\in F_n$ such that $|S^*_{m,n}(x)|\ge\varepsilon$ and $|S^*_{m,n}(y)|\ge\varepsilon$.\\\par
		{\em\underline{Claim}} : Set $\mu=\max_n\#F_n/\varepsilon^2$. Then $\#\{ S_{m,n}|_{[0,p_n]}:m\in L,m<n\}\le\mu$ for every $n\in L$, where for a segment $S$ and $p,q\in\N$ we denote $S|_{[p,q]}=S\cap Q_{[p,q]}$.
		\begin{proof}[Proof of Claim]\renewcommand{\qedsymbol}{}
			If $\#\{ S_{m,n}|_{[0,p_n]}:m\in L,m<n\}>\mu$ for some $n\in L$, then using the pigeon hole principle, we may find an $x\in F_n$ and $F\subset\{1,\ldots,n-1\}$ with $\#F>1/\varepsilon^2$ such that $|S_{m,n}(x)|\ge\varepsilon$ and the segments $S_{m,n}|_{[p_n,q_n]}$ are pairwise disjoint for $m\in F$. This contradicts item (ii) of Remark \ref{remnormjt}.
		\end{proof}

  Hence, for every $n\in L$, let $\{S_{m,n}|_{[0,p_n]}:m\in L,m<n\}=\{S^n_1,\ldots,S^n_{\mu(n)} \}$ with $\mu(n)\le\mu$ and set $L_i^n=\{m\in L:m<n\;\text{and }S_{m,n}|_{[0,p_n]}=S^n_i \}$, for $1\le i\le \mu(n)$, and $L^n_i=\emptyset$, for $\mu(n)<i\le \mu$. Notice that $\{m\in L:m<n\}=\cup_{i=1}^{\mu}L_i^n$ for all $n\in L$. Passing to a further subsequence we may assume that, for every $1\le i\le \mu$, $(L_i^n)_{n\in {L}}$ converges point-wise  and we denote that limit by $L_i$. Then it is easy to see that $L=\cup_{i=1}^\mu L_i$ and hence some $L_{i_0}$ is an infinite subset of $L$ such that, for every $n\in L_{i_0}$, there exists an initial segment $S_n$ such that, for all $m<n$ in $L_{i_0}$, we have that $|S^*_n(x_m)|\ge\varepsilon$ for some $x_m\in F_m$. Then there exist $M\in[L_{i_0}]^{\infty}$ and a sequence $(x_n)_{n\in M}$ with $x_n\in F_n$, such that $(S_n)_{n\in M}$ converges point-wise to a branch B and  $|S^*_m(x_n)|\ge\varepsilon$ for all $m>n$ in $M$. Hence $|B(x_n)|\ge\varepsilon$ for all $n\in M$, which contradicts item (ii) of Definition \ref{w0lbf}.
\end{proof}

\begin{lem}
	Let $\varepsilon>0$ and $(F_n)_n$ be a weakly null level block family with respect to $(Q_{[p_n,q_n]})_n$ and assume that $\sup_n\#F_n<\infty$. Then there exist an increasing sequence $(n_k)_k$ in $\N$ and a decreasing sequence $(\varepsilon_k)_k$ of positive reals such that
	\begin{enumerate}
		\item[(i)] For every $k\in\N$ and every initial segment $S$ there exists at most one $k'>k$ such that $|S^*(x)|\ge\varepsilon_{k}$ for some $x\in F_{n_{k'}}$.
		\item[(ii)] $\sum_{k=1}^\infty 2^{q_{n_k}}\sum_{i=k}^\infty (i+1)\varepsilon_{i}<\varepsilon$.
	\end{enumerate}
\end{lem}
\begin{proof}
	Let $(\delta_n)_n$ be a sequence of positive reals such that $\sum_{n=1}^\infty\delta_n<\varepsilon$. We will construct $(n_k)_k$ and $(\varepsilon_k)_k$ by induction on $\N$ as follows. We set $n_1=1$ and $L_1=\N$ and choose $\varepsilon_1$ such that $2^{q_1}2\varepsilon_1<\delta_1$. Suppose that $n_1,\ldots,n_k$ and $\varepsilon_1,\ldots,\varepsilon_k$ have been chosen for some $k$ in $\N$. Then Lemma \ref{jtramsey} yields an $L_k\in[L_{k-1}]^\infty$ such that for every segment $S$ there exist at most one $n\in L_k$ with $|S^*(x)|\ge\varepsilon_k$ for some $x\in F_{n}$. We then choose $n_{k+1}\in L_k$ with $n_{k+1}>n_k$ and $\varepsilon_{k+1}<\varepsilon_k$ such that
	\begin{enumerate}
		\item[(a)] $2^{q_{n_{k+1}}}(k+2)\varepsilon_{k+1}<\delta_{k+1}$
		\item[(b)] $2^{q_{n_m}}\sum_{i=m}^{k+1}(i+1)\varepsilon_{i}<\delta_m$ for every $m\le k$.
	\end{enumerate}
	It is easy to see that $(n_k)_k$ and $(\varepsilon_k)_k$ are as desired.
\end{proof}

\begin{lem}\label{jtlast}
	Let $\varepsilon>0$ and $(\varepsilon_n)_n$ be a decreasing sequence of positive reals. Let also $(F_n)_n$ be a weakly null level block family with respect to $(Q_{[p_n,q_n]})_n$ and assume that $\sup_n\#F_n<\infty$ and that the following hold.
	\begin{enumerate}
		\item[(i)] For every $n\in\N$ and every initial segment $S$ there exists at most one $m>n$ such that $|S^*(x)|\ge\varepsilon_{n}$ for some $x\in F_{m}$.
		\item[(ii)] $\sum_{n=1}^\infty 2^{q_{n}}\sum_{i=n}^\infty (i+1)\varepsilon_{i}<\varepsilon$.
	\end{enumerate}
	Then for every  $n\in\N$ and every choice of $x_1,\ldots,x_n$ with $x_i\in F_i$ and scalars $a_1,\ldots,a_n$ we have that
	\[
	\Big(\sum_{i=1}^na_i^2\Big)^{\frac{1}{2}}\le\Big\|\sum_{i=1}^{n}a_ix_i\Big\| \le (\sqrt{2}+\varepsilon)\Big(\sum_{i=1}^na_i^2\Big)^{\frac{1}{2}}.
	\]
\end{lem}
\begin{proof} Let us first observe that if $(x_n)_n$ is a sequence with each $x_n\in F_n$, then for every $n\in\N$ and every segment $S$ with $|M(x_{n-1})|< |\min S|\le| M(x_{n})|$, where for $x\in c_{00}(\mathscr{D})$ we denote $M(x)=\max\supp(x)$, the following hold due to (i).
	\begin{enumerate}
		\item[(a)] $\#\{i>n:|S^*(x_{i})|\ge\varepsilon_{n} \}\le 1$
		\item[(b)] $\#\{i>n:\varepsilon_{k-1}>|S^*(x_{i})|\ge\varepsilon_{k} \}\le k$ for every $k>n$.
	\end{enumerate}	
	Now for each $1\le i\le n$, there exist pairwise disjoint segments $S^i_1,\ldots,S^i_{m_i}$ such that $S_j^i\subset Q_{[p_i,q_i]}$ and $\sum_{j=1}^{m_i}(S^{i*}_j(x_i))^2=\|x_i\|$ and hence
	\[
	\Big\|\sum_{i=1}^na_ix_i\Big\|\ge \Big(\sum_{i=1}^na_i^2\sum_{j=1}^{m_i}\big( S_j^{i*}(x_i)\big)^2\Big)^{\frac{1}{2}}\ge \Big(\sum_{i=1}^na_i^2\Big)^{\frac{1}{2}}.
	\]
	Pick $S_1,\ldots,S_m$ pairwise disjoint segments and $b_1,\ldots,b_m$ reals with $\sum_{j=1}^mb_j^2\le1$. For given $1\le j\le m$, we will denote by $i_{j,1}$ the unique $1\le i\le n$ such that $|M(x_{i_{j,1}-1})|< |\min S_j|\le|M(x_{i_{j,1}})|$ and also by $i_{j,2}$ the unique, if there exists, $i_{j,1}<i\le n$ such that $|S_j^*(x_{i_{j,2}})|\ge\varepsilon_{i_{j,1}}$. We set $S_{j,k}=S_j\cap Q_{[p_{i_{j,k}},q_{i_{j,k}}]}$, for $k=1,2
	$, and $S_{j,3}=S_{j}\setminus(S_{j,1}\cup S_{j,2})$ and we also set $J_{i}=\{j:i_{j,1}=i\;\textnormal{or}\;i_{j,2}=i\}$, for $1\le i \le n$. Note that, by (a), each $j$ appears in $J_i$ for at most two $i$ and so $\sum_{i=1}^n\sum_{j\in J_i}b_j^2\le2\sum_{j=1}^mb_j^2$. We thus calculate:
	\begin{align*}
	&\Big|\sum_{j=1}^mb_jS^*_{j,1}\Big(\sum_{i=1}^na_ix_i\Big)+\sum_{j=1}^mb_jS^*_{j,2}\Big(\sum_{i=1}^na_ix_i\Big)\Big|=\Big|\sum_{i=1}^na_i\sum_{j\in J_i}b_jS^*_j(x_i)\Big|
	\\&\le\Big(\sum_{i=1}^na_i^2 \Big)^{\frac{1}{2}}\Big(\sum_{i=1}^n\big(\sum_{j\in J_i}b_jS^*_j(x_i)\big)^2 \Big)^{\frac{1}{2}}\le\Big(\sum_{i=1}^na_i^2 \Big)^{\frac{1}{2}}\Big(\sum_{i=1}^n\sum_{j\in J_i}b_j^2 \Big)^{\frac{1}{2}}\le\sqrt{2}\Big(\sum_{i=1}^na_i^2 \Big)^{\frac{1}{2}}.
	\end{align*}	
	Finally, we set $G_i=\{ j:|M(x_{i-1})|< |\min S_j|\le|M(x_i)|\}$ and we have that $\{1,\ldots,m \}=\cup_{i=1}^nG_i$. Notice that $\#G_i\le2^{q_i}$ and $|S^*_{j,3}(\sum_{k=1}^nx_k)|<\sum_{k=i}^\infty (k+1)\varepsilon_{k}$ for any $j\in G_i$. Then due to (b) and (ii), it follows that $\sum_{j=1}^m|S^*_{j,3}(\sum_{i=1}^nx_i)|<\varepsilon$ and hence we conclude that
	\[
	\Big|\sum_{j=1}^mb_jS^*_{j,3}\Big(\sum_{i=1}^na_ix_i\Big)\Big|=\Big|\sum_{i=1}^na_i\sum_{j=1}^mb_jS^*_{j,3}(x_i)\Big|<\varepsilon\Big(\sum_{i=1}^na_i^2 \Big)^{\frac{1}{2}}.
	\]
\end{proof}	

\begin{proof}[Proof of Theorem \ref{jtunique}]
	Let $(x^1_n)_n,\ldots,(x^l_n)_n$ in $\mathscr{F}_0(JT)$ be  such that $((x^i_n)_n)_{i=1}^l$ generates a sequence $(e^i_n)_{i=1,n\in\N}^l$ as an $l$-joint spreading model and by a sliding hump argument we may assume that each sequence $(x^i_n)_n$ is block. Hence we may choose $L\in[\N]^\infty$ such that the family $(F_n)_{n\in L}$, with $F_n=\{x_n^1,\ldots,x_n^l\}$, is a weakly null level block family in $JT$ which satisfies (i) and (ii) in Lemma \ref{jtlast}. Then, since $((x^i_n)_{n\in L})_{i=1}^l$ generates also generates $(e^i_n)_{i=1,n\in\N}^l$ as an $l$-joint spreading model, we conclude that $(e^i_n)_{i=1,n\in\N}^l$ is $\sqrt{2}$-equivalent to the usual basis of $\ell_2$ and therefore any two $l$-joint spreading models, generated by sequences from $\mathscr{F}_0(JT)$, are $2$-equivalent.
\end{proof}
		
		\begin{rem}
            The notion of asymptotic models, which appeared in \cite{HO}, also concerns the asymptotic behavior of countably many basic sequences. Although the asymptotic models are different from joint spreading models, as B. Sari pointed out, a Banach space admits a uniformly unique asymptotic model with respect to a family $\mathscr{F}$ if and only if it admits a uniformly unique joint spreading model with respect to $\mathscr{F}$.
        \end{rem}

	 \section{uniform approximation of bounded operators}
	 We now pass to the study of the UALS property on certain classes of spaces. First we consider spaces with very few operators, namely spaces with the scalar-plus-compact property. The second class includes spaces admitting a uniformly unique joint spreading model with respect to certain families of Schauder basic sequences. Here, the notion of joint spreading models and the UALS property come together in the sense that the first property yields the second one. The third subsection is devoted to the study of the UALS property under duality. A consequence of the main result, Theorem \ref{theorem UALS by duality}, is that the spaces $C(K)$, with $K$ countable compact, satisfy the UALS. In the fourth subsection we show that the spaces $L_p[0,1]$, for $1\leq p\leq \infty$ and $p\neq 2$, and $C(K)$ for uncountable compact metric spaces $K$ fail the UALS property. We close with some final remarks and open problems.

	 \begin{dfn}\label{uals}
	 	We will say that a Banach space $X$ satisfies the {\em Uniform Approximation on Large Subspaces} ({\em UALS} ) property  if there exists $C>0$ such that the following is satisfied. For every convex compact subset $W$ of $\mathcal{L}(X)$, every $A\in\mathcal{L}(X)$ and $\varepsilon>0$ with the property that, for every $x\in B_X$, there is a $B\in W$ such that $\|A(x)-B(x)\|\le\varepsilon$, then there exist a finite codimensional subspace $Y$ of $X$ and a $B\in W$ such that $\|(A-B)|_Y\|_{\mathcal{L}(Y,X)}\le C\varepsilon$.
	 \end{dfn}
	
	 \begin{dfn}\label{ualssat}	
	 	A Banach space $X$ will be called {\em UALS-saturated} if there exists $C>0$ such that for every convex compact subset $W$ of $\mathcal{L}(X)$, every $A\in\mathcal{L}(X)$ and $\varepsilon>0$ with the property that, for every $x\in B_X$, there is a $B\in W$ such that $\|(A-B) x\|\le\varepsilon$, it holds that for every subspace $Y$ of $X$ there exists a further subspace $Z$ such that $\|(A-B)|_Z\|_{\mathcal{L}(Z,X)}\le C\varepsilon$, for some $B\in W$.
	 \end{dfn}
	
	\subsection{The UALS Property for Compact Operators} The first class of spaces satisfying the UALS includes spaces with very few operators. We prove that Banach spaces with the scalar-plus-compact property satisfy the UALS and are in fact UALS-saturated. Hence, the main result from \cite{A} yields that a large class of spaces, that includes all superreflexive spaces, embed into spaces that satisfy the UALS. We start with the following variation of Mazur's theorem \cite[Theorem 1.a.5]{LT}.
		
		\begin{lem}
			Let $X$ be a Banach space, $T\in\mathcal{L}(X)$ and let $\varepsilon>0$ be such that $\|T|_Y\|_{\mathcal{L}(Y,X)}> \varepsilon$ for every subspace $Y$ of $X$ of finite codimension. Then there exists a normalized sequence $(x_n)_n$ in $X$ such that $(Tx_n)_n$ is seminormalized Schauder basic.
		\end{lem}
		\begin{proof}
			Let $\delta>0$. Pick $x_1$ in the unit sphere of $X$ with $\|Tx_1\|\ge\varepsilon$ and assume that $x_1,\ldots,x_n$ have been chosen for some $n\in\N$. Let $G$ be a finite subset of $X^*$ such that, for every $x\in\spn\{Tx_1,\ldots,Tx_n\}$, we have that $\|x\|\le(1+\delta)\max\{g(x):g\in G\}$ and choose $x_{n+1}$ in the unit sphere of $\cap_{g\in G}\ker T^*g$ with $\|Tx_{n+1}\|\ge\varepsilon$. It follows quite easily that  $(Tx_n)_n$ is a Schauder basic sequence.
		\end{proof}
		
		\begin{prop}\label{compact operators inf is zero over codim}
			Let $X$ be a Banach space and $T\in\mathcal{L}(X)$ be a compact operator. Then $\inf \|T|_Y\|_{\mathcal{L}(Y,X)}=0$, where the infimum is taken over all subspaces $Y$ of $X$ of finite codimension.
		\end{prop}
		\begin{proof}
			Suppose that the conclusion is false, then the previous lemma yields that $T[B_X]$ contains a seminormalized Schauder basic sequence and this contradicts the fact that $T$ is compact.
		\end{proof}
		
		\begin{notn}
			Let $X$ be a Banach space, we will denote by $\mathcal{K}(X)$ the ideal of all compact operators in the unital algebra $\mathcal{L}(X)$.
		\end{notn}
		
		\begin{cor}
			Let $X$ be a Banach space, $W$ be a compact subset of $\mathcal{K}(X)$ and $A\in\mathcal{L}(X)$. Assume that there exists $\varepsilon>0$, such that for every $x\in B_X$, there is a $B\in W$ such that $\|A(x)-B(x)\|\le\varepsilon$. Then, for every $\delta>0$, there exists a finite codimensional subspace $Y$ of $X$ such that $\|(A-B)|_{Y}\|_{\mathcal{L}(Y,X)}\le\varepsilon+\delta$ for all $B\in W$.
		\end{cor}
		\begin{proof}
			Let $\delta>0$ and $\{B_i\}_{i=1}^n$ be a ${\delta}$-net of W. Applying Proposition \ref{compact operators inf is zero over codim}, we choose a finite codimensional subspace $Y$ of $X$ such that $\|B_i|_Y\|<\delta$ for every $1\le i\le n$. Note that $\|B|_Y\|\le 2\delta$ for all $B\in W$. Then, for every $x$ in the unit ball of $Y$, there is a $B$ in $W$ such that $\|A(x)-B(x)\|\le\varepsilon$ and hence $\|A(x)\|<\varepsilon+2\delta$. That is, $\|A|_Y\|\le\varepsilon+2\delta$. Therefore $\|(A-B)|_Y\|_{\mathcal{L}(Y,X)}\le\varepsilon+4\delta$ for every $B\in W$.
		\end{proof}

		\begin{thm}\label{apspc}
			Every Banach space with the scalar-plus-compact property satisfies the  UALS.
		\end{thm}
		\begin{proof}
			Let $W,A,\varepsilon$ be as in Definition \ref{uals}, with $A=\lambda_AI+K_A$ and $K_A$ compact. Let $\delta>0$ and $\{B_i\}_{i=1}^n$ be a ${\delta}$-net of $W$ with $B_i=\lambda_iI+K_i$ and $K_i\in\mathcal{K}(X)$, for $i=1,\ldots,n$. Proposition \ref{compact operators inf is zero over codim} then yields a finite codimensional subspace $Y$ of $X$ such that $\|K_A|_Y\|\le\delta$ and $\|K_i|_Y\|\le \delta$ for all $1\le i\le n$. Pick an $x\in Y$ with $\|x\|= 1$ and $B\in W$ with $B=\lambda_BI+K_B$ and $\|A(x)-B(x)\|\le\varepsilon$. Then, for $1\leq i\leq n$ such that $\|B - B_i\| \leq \delta$, we have $\|A(x) - B_i(x)\| \leq \varepsilon +\delta$ and hence $|\lambda_A - \lambda_i| = \|\lambda_Ax - \lambda_ix\| \leq \|Ax - B_ix\| + \|K_Ax + K_ix\| \leq \varepsilon + 2\delta$. Then, for every $y$ in the unit ball of $Y$, we have that $\|A(y)-B_i(y)\|\le\varepsilon+4\delta$ which proves the desired result.
		\end{proof}
		
		\begin{thm}
			Let $X$ be a Banach space such that, for every $A\in \mathcal{L}(X)$, there is a strictly singular operator $S$ and $\lambda\in\R$ such that $A=\lambda I+S$. Then $X$ is  UALS-saturated.
		\end{thm}
		\begin{proof}
			Let $W,A,\varepsilon$ be as in Definition \ref{ualssat}, with $A=\lambda_AI+S_A$ and $S_A$  a strictly singular operator. Let  $\delta>0$ and $\{B_i\}_{i=1}^n$ be a $\delta$-net of $W$ with $B_i=\lambda_iI+S_i$ and $S_i$ strictly singular, $1\le i\le n$. Recall that for every infinite dimensional subspace of $X$ there exists a further subspace $Y$ such that $S_A|_Y$  and $S_i|_Y$, for $1\le i\le n$, are compact operators. Applying the same arguments used in the previous proof we obtain the desired conclusion.
	\end{proof}

	 \subsection{Uniformly Unique Joint Spreading Models and the UALS Property} In this subsection we study spaces that admit uniformly unique $l$-joint spreading models with respect to families with sufficient stability properties, described in the following definition,  to deduce that in certain cases they satisfy the UALS property. Such spaces are, for example, all Asymptotic $\ell_p$ spaces. This should be compared to the examples of the following subsection that fail the UALS and the proof of this fact is based on the existence of diverse plegma spreading sequences in these spaces.

	 The families of sequences that we restrict our study to are very rich, in the sense that any sequence has a subsequence the successive differences of which are in $\mathscr{F}$ and it is also closed under taking subsequences. Moreover, if a space has a uniformly unique $l$-joint spreading model with respect to such a family then, as already shown in Proposition \ref{kind of krivine sort of}, it has to be at least unconditional and in most cases it has to be some $\ell_p$ or $c_0$.

	 \begin{dfn}\label{difference including}
	 	Let $X$ be a Banach space. A collection $\mathscr{F}$ of normalized and Schauder basic sequences in $X$ will be called {\em difference-including} if
	 	\begin{itemize}
	 		
	 		\item[(i)] for every $(x_n)_n$ in $\mathscr{F}$ any  subsequence of $(x_n)_n$ is in $\mathscr{F}$ and
	 		
	 		\item[(ii)]  for every sequence $(x_n)_n$ in $X$ without a norm convergent subsequence there exists an infinite subset $L$ of $\mathbb{N}$ such that, for any further infinite subset $M = \{m_k:k\in\mathbb{N}\}$ of $L$,  if $z_k = \|x_{m_{2k-1}}-x_{m_{2k}}\|^{-1}(x_{m_{2k-1}}-x_{m_{2k}})$, then the sequence $(z_k)_k$ is in $\mathscr{F}$.
	 		
	 	\end{itemize}
	 \end{dfn}

	 \begin{rem}
	 	A difference-including collection clearly satisfies (a) and (b) of Proposition \ref{kind of krivine sort of}. In fact, most naturally defined families of normalized Schauder basic sequences in a Banach space $X$ are difference-including. Such families are:
	 	\begin{itemize}
	 		
	 		\item[(i)] $\mathscr{F}(X)$, the collection of all normalized Schauder basic sequences in $X$.
	 		
	 		\item[(ii)] $\mathscr{F}_{(1+\varepsilon)}(X)$, the collection of all normalized $(1+\varepsilon)$-Schauder basic sequences in $X$ for some fixed $\varepsilon>0$.
	 		
	 		\item[(iii)] $\mathscr{F}_{0,\mathscr{A}}$ for a countable subset $\mathscr{A}$ of the dual, where
	 		\[\mathscr{F}_{0,\mathscr{A}} = \left\{(x_n)_n:\;(x_n)_n\text{ is normalized and }\lim_nf(x_n) = 0\text{ for all }f\in\mathscr{A}\right\}.\]
	 		
	 		\item[(iv)] $\tilde{\mathscr{F}}_{b}(X) = \mathscr{F}_{0,(e_n^*)_n}$ if $X$ has a Schauder basis $(e_n)_n$, where $(e^*_n)_n$ are the biorthogonal functionals associated to the basis. Notice that a Banach space $X$ admits a uniformly unique $l$-joint spreading model with respect to $\mathscr{F}_{b}(X)$ if and only if it does so with respect to $\tilde{\mathscr{F}}_b(X)$.
	 		
	 		\item[(v)] $\mathscr{F}_0(X)$ if $X$ does not contain $\ell_1$.
	 		
	 		\item[(vi)] $\mathscr{F}_\mathrm{su}(X)$, the collection of all normalized Schauder basic sequences that generate a 1-suppression unconditional spreading model.
	 		
	 	\end{itemize}
	 \end{rem}

	 In certain cases, for $X$ non-separable it is convenient to consider different collections $\mathscr{F}_Z$ for different separable subspaces $Z$ of $X$. This is included in the statement of the following Theorem.

	 \begin{thm}
	 	\label{the theorem}
	 	Let $X$ be a Banach space and assume that for every separable subspace $Z$ of $X$ we have a difference-including collection $\mathscr{F}_Z$ of normalized Schauder basic sequences  in $Z$.  If there exists a uniform $K\geq 1$ such that each such $Z$ admits a $K$-uniformly unique $l$-joint spreading model with respect to $\mathscr{F}_Z$, then $X$ satisfies the UALS property.
	 \end{thm}
	
	 We postpone the proof of Theorem \ref{the theorem} to first state and prove its corollaries.  Note that if $X$ is a Banach space and $\mathscr{F}$ is a collection of normalized Schauder basic sequences in $X$ with respect to which it admits a uniformly unique $l$-joint spreading model, then we may consider, for every separable subspace $Z$ of $X$, the family $\mathscr{F}_Z = \{(x_n)_n$ in $ \mathscr{F}: x_i\in Z$ for all $i\in\mathbb{N}\}$. This is in fact  sufficient to prove most cases stated bellow.
	
	 \begin{cor}\label{the theorem 2}
	 	In the following cases, a Banach space $X$ and all of its subspaces satisfy the UALS property.
	 	\begin{itemize}
	 		
	 		\item[(i)] $X$ has a Schauder basis and it admits a uniformly unique $l$-joint spreading model with respect to $\mathscr{F}_b(X)$.
	 		
	 		\item[(ii)] $X$ is an arbitrary Banach space that admits a uniformly unique $l$-joint spreading model with respect to $\mathscr{F}(X)$.
	 		
	 		\item[(iii)] $X$ does not contain $\ell_1$ and it admits a uniformly unique $l$-joint spreading model with respect to $\mathscr{F}_0(X)$.
	 		
	 		\item[(iv)] $X$ is an arbitrary Banach space and, for some $\varepsilon >0$, it admits a uniformly unique $l$-joint spreading model with respect to $\mathscr{F}_{(1+\varepsilon)}(X)$.

	 	\end{itemize}
	 \end{cor}
	
	 \begin{proof}
	 	All cases follow from Theorem \ref{the theorem}. We describe some of the details. The first case follows from the fact that such an $X$ admits a uniformly unique $l$-joint spreading model with respect to $\tilde{\mathscr{F}}_b(X) = \mathscr{F}_{0,(e_i^*)_i}$, which is difference-including. The second case follows directly from the fact that $\mathscr{F}(X)$ is difference-including. In case (iii), from Rosenthal's $\ell_1$ theorem \cite{Rosenthal}, we have that $\mathscr{F}_0(X)$ is difference-including. For case (iv), note that $\mathscr{F}_{(1+\varepsilon)}$ is difference-including as well.
	 \end{proof}
	
	\begin{cor}
		The following Banach spaces and all of their subspaces satisfy the UALS property.
		\begin{itemize}
			\item[(a)] The space $\ell_p(\Gamma)$, for $1\le p< \infty$ and any  infinite set $\Gamma$.
			\item[(b)] The space $c_0(\Gamma)$, for any  infinite set $\Gamma$.
			\item[(c)] The James Tree space.
			\item[(d)] Every Asymptotic $\ell_p$ space, for $1\leq p\leq \infty$.
		\end{itemize}				
	\end{cor}
	
	\begin{proof}
		The case of $\ell_p(\Gamma)$ follows from item (ii) of Corollary \ref{the theorem 2} for $1<p<\infty$ and from (i) for $p=1$, while that of $c_0(\Gamma)$ and the James Tree follows from item (iii). Moreover, for case (d), if $1<p\leq \infty$ and $X$ is an Asymptotic $\ell_p$ space, then it does not contain $\ell_1$ and admits a uniformly unique $l$-joint spreading model with respect to $\mathscr{F}_0(X)$ so the result follows from case (iii) as well. Finally, if $X$ is $C$-Asymptotic $\ell_1$, use Theorem \ref{unique joint if these go to zero} to choose for every separable subspace $Z$ of $X$ a countable subset $\mathscr{A}_Z$ of $Z^*$ such that $Z$ admits a $C^2$-uniformly unique $l$-joint spreading model with respect to $\mathscr{F}_Z = \mathscr{F}_{0,\mathscr{A}_Z}$.
	\end{proof}

	 We break up the proof of Theorem \ref{the theorem} into several steps.
	
	 \begin{lem}
	 	\label{constant coming from unique joint sm}
	 	Let $X$ be a Banach space that admits a $K$-uniformly unique $l$-joint spreading model with respect to a difference-including collection $\mathscr{F}$ of normalized Schauder basic sequences. Then for any  $D > 2K^2$ and any sequences $(z_n^i)_n$, $(y_n^i)_n$, $i=1,\ldots,l$, in $\mathscr{F}$, there exists an infinite subset $L$ of $\mathbb{N}$ such that, for any scalars $a_1,\ldots,a_l$, $\theta_1,\ldots,\theta_l$ and $n_1 <\cdots <n_l$ in $L$ we have that
	 	\begin{equation}
	 	\label{same norms basically}
	 	\min_{1\leq i\leq l}|\theta_i|\frac{1}{D}\Big\|\sum_{i=1}^la_iz_{n_i}^i\Big\| \leq \Big\|\sum_{i=1}^la_i\theta_i y_{n_i}^i\Big\| \leq \max_{1\leq i\leq l}|\theta_i|D\Big\|\sum_{i=1}^la_iz_{n_i}^i\Big\|.
	 	\end{equation}
	 \end{lem}
	
	 \begin{proof}
	 	Choose $C > K$ and pass to an infinite set $L$ so that the $l$-tuples $((z_n^i)_{n\in L})_{i=1}^l$ and $((y_n^i)_{n\in L})_{i=1}^l$ generate some $l$-joint spreading models that are $K$-equivalent to one another. This means that we may perhaps pass to a further subset of $L$ and have that, for any $n_1<\cdots<n_l$ in $L$, the sequences $(z_{n_i}^i)_{i=1}^l$ and $(y_{n_i}^i)_{i=1}^l$ are $C$-equivalent and each of them is ${C}$-suppression unconditional and hence $2{C}$-unconditional. For any scalars $a_1,\ldots,a_l$, $\theta_1,\ldots,\theta_l$, we calculate:
	 	\[\Big\|\sum_{i=1}^la_i\theta_iy_{n_i}^i\Big\|\geq \frac{1}{2{C}}\min_{1\leq i\leq l}|\theta_i|\Big\|\sum_{i=1}^la_iy_{n_i}^i\Big\|\geq \frac{1}{2C^{2}}\min_{1\leq i\leq l}|\theta_i|\Big\|\sum_{i=1}^la_iz_{n_i}^i\Big\|.\]
	 	The other inequality is obtained identically. Therefore, any $D>2K^2$ satisfies the conclusion.
	 \end{proof}

	The following is another variant of Mazur's theorem \cite[Theorem 1.a.5]{LT}.

	 \begin{lem}
	 	\label{mazur with operators}

	 	Let $X$ be a separable Banach space. Let also $T_{ij}\in\mathcal{L}(X)$, for $1\leq i\leq n$, $1\leq j\leq m_i$ and $c>0$ be such that, for every $i=1,\ldots,n$ and $Y$ finite codimensional subspace of $X$, there is an $x_i\in Y$ with $\|x_i\|=1$ and $\|T_{ij}x_i\|>c$ for all $1\leq j\leq m_i$. Then, for every $\varepsilon>0$, there exist normalized sequences $(x^i_k)_k$,  $i = 1,\ldots, n$, in $X$ such that if we set $Z_k = \spn\{\{x^i_k\}_{i=1}^n\cup\{T_{ij}x^i_k\}_{i=1,j=1}^{n,m_i}\}$, for $k\in\N$, and $Z = \cspn\cup_k Z_k$ we have that:
	 	\begin{enumerate}
	 		\item[(i)]  $(Z_k)_k$ forms an FDD for the space $Z$, with projection constant at most $1+\varepsilon$.
	 		\item[(ii)] $\|T_{ij}(x^i_k)\|>c$ for every $i=1,\ldots,n$, $j=1,\ldots,m_i$, and $k\in\N$.
	 	\end{enumerate}
	 \end{lem}
	
	 \begin{proof}
	 	Set $\mathcal{A} = \{I\}\cup\{T_{ij}\}_{i=1,j=1}^{n,m_i}$ and, for every $1\le i\le n$, choose a normalized vector $x_1^i$ in $X$ with $\|T_{ij}x_1^i\| > c$ for all $j=1,\ldots,m_i$. Assume that we have chosen $(x_k^i)_{k=1}^d$ up to some $d\in\N$, for $1\le i\le n$, so that the spaces $(Z_k)_{k=1}^d$ satisfy (i) for the space they generate and (ii)  for $1\leq k\leq d$. Choose a finite subset $G$ of the unit sphere of $X^*$ such that, for all $x$ in the linear span of $\cup_{k=1}^dZ_k$, we have $\|x\| \leq (1+\varepsilon)\max\{g(x): g\in G\}$ and set $F = \cup_{T\in\mathcal{A}}\{T^*g:g\in G\}$. Finally, for $i=1,\ldots,n,$ choose $x_{d+1}^i$ in the unit sphere of $\cap_{f\in F}\ker f$ such that $\|T_{ij}x_{d+1}^i\| > c$ for all $1\leq j\leq m_i$. It follows then quite easily that the sequences are as desired.
	 \end{proof}

	 \begin{lem}
	 	\label{all difference-including domain and target}

	 	Let $X$ be a Banach space and assume that for every separable subspace $Z$ of $X$ we have a difference-including collection $\mathscr{F}_Z$ of normalized Schauder basic sequences in $Z$. Let $T_1,\ldots,T_l$ be bounded linear operators on $X$ and assume that there is $c>0$ such that, for every finite codimensional subspace $Y$ of $X$ and every $i=1,\ldots,l$, we have that $\|T_i|_Y\|_{\mathcal{L}(Y,X)} \geq c$. Assume moreover that, for some $0<\delta<c$ and $i=1,\ldots,l$, we have  $T_{i1},\ldots,T_{im_i}$ in $\mathcal{L}(X)$ with $\|T_{ij} - T_i\| \leq \delta$ for $j=1,\ldots,m_i$. Then, if $\tilde c = c-\delta$, there exist a separable subspace $Z$ of $X$ and normalized  sequences $(z_k^i)_k$, $i=1,\ldots,l$ in $\mathscr{F}_Z$ such that
	 	
	 	\begin{itemize}
	 		
	 		\item[(i)] for any $n_1 < \cdots <n_l$, the sequence $(z_{n_i}^i)_{i=1}^l$ is $(9/8)$-Schauder basic,
	 		
	 		\item[(ii)] $\|T_{ij}z^i_k\| > \tilde c/3$ for   $i=1,\ldots,n$, $j=1,\ldots,m_i$ and $k\in\mathbb{N}$, and
	 		\item[(iii)] if $y_k^{ij} = \|T_{ij}z_k^i\|^{-1}T_{ij}z_k^i$ then $(y_k^{ij})_k$ is in $\mathscr{F}_Z$ for $i=1,\ldots,n$, $j=1,\ldots,m_i$.
	 		
	 	\end{itemize}
	 \end{lem}

	 \begin{proof}
	 	Note that, for any finite codimensional subspace $Y$ of $X$, we may choose $x_i$ in the unit ball of $Y$ such that $\|T_ix_i\| > c - (c-\delta)/4$, which means that, for $j=1,\ldots,m_i$, we have $\|T_{ij}x_i\| > 3(c-\delta)/4 = 3\tilde c/4$. Apply Lemma \ref{mazur with operators} to find normalized sequences $(x_k^i)_k$, $i=1,\ldots,n$, such that $\|T_{ij}x_k^i\| > 3\tilde c/4$ for all $k\in\mathbb{N}$, $i=1,\ldots,n$ and $j=1,\ldots,m_i$, and the sequence $(Z_k)_k$ defined in Lemma \ref{mazur with operators} is an FDD with constant $9/8$. Let $Z = \cspn\cup_kZ_k$. Choose $L$ such that, if we set $z_k^i = \|x^i_{m_{2k-1}}-x^i_{m_{2k}}\|^{-1}(x^i_{m_{2k-1}}-x^i_{m_{2k}})$, then $(z_k^i)_k$ as well as $(\|T_{ij}z_k^i\|^{-1}T_{ij}z_k^i)_k$ are in $\mathscr{F}_Z$ for $i=1,\ldots,n$, $j=1,\ldots,m_i$. By the fact that $(T_{ij}x_k^i)_k$ is $9/8$-Schauder basic we obtain that
	 	\[\big\|T_{ij}z_k^i\big\| = \frac{1}{\|x^i_{m_{2k-1}}-x^i_{m_{2k}}\|}\big\|T_{ij}x^i_{m_{2k-1}}-T_{ij}x^i_{m_{2k}}\big\| \geq \frac{1}{2}\frac{1}{9/8}\big\|T_{ij}x^i_{m_{2k-1}}\big\| > \frac{3\tilde c/4}{9/4}.\]
	 	The fact that statement (i) is true follows from the fact $(Z_k)_k$ is an FDD with constant $9/8$ and $(z_{n_i}^i)_{i=1}^l$ is a block sequence.
	 \end{proof}	 	
	
	S. Kakutani \cite{Kakutani} proved the finite dimensional analog of the following theorem, also known as Kakutani's Fixed Point Theorem. We present the infinite dimensional case by H. F. Bohnenblust and S. Karlin \cite{BK}, which as mentioned already is a key ingredient in the proof of Theorem \ref{the theorem}. Recall that a multivalued mapping $\phi:X\twoheadrightarrow Y$ between topological spaces has closed graph if for every $(x_n)_n\in X$ with $\lim x_n=x$ and $(y_n)_n\in Y$ with $y_n\in\phi(x_n)$ and $\lim y_n=y$, we have that $y\in\phi(x)$.
	
	 \begin{thm}\label{fixedpoint}
	 	Let $X$ be a Banach space, $K$ be a nonempty compact convex subset of $X$ and let the multivalued mapping $\phi:K\twoheadrightarrow K$ have closed graph and nonempty convex values. Then $\phi$ has a fixed point, i.e. there exists $x\in X$ such that $x\in\phi(x)$.
	 \end{thm}	
	
	 \begin{proof}[Proof of Theorem \ref{the theorem}]
	 	Let $D > 2K^2$, i.e. a constant for which the conclusion of Lemma \ref{constant coming from unique joint sm} can be applied to all families $\mathscr{F}_Z$. Set $C = 7D$. Let $W$ be a convex and compact subset of $\mathcal{L}(X)$, $A\in\mathcal{L}(X)$, and $\varepsilon>0$ such that, for all $x$ in the unit ball of $X$, there is $T\in W$ with $\|A(x) - T(x)\|\le \varepsilon$. We claim that there is a finite codimensional subspace $Y$ of $X$ and $T\in W$ such that $\|(A-T)|_Y\|_{\mathcal{L}(Y,X)} < C\varepsilon$. Assume that the conclusion is false. Set $c = C\varepsilon$, $\delta = c/2$, and $\tilde c = c - \delta = C/2$. Choose a maximal $\delta$-separated subset $(T_i)_{i=1}^l$ of $W$, set $\eta = \varepsilon/(27l)$ and for $i=1,\ldots,l$ choose a maximal $\eta$-separated subset $(T_{ij})_{j=1}^{m_i}$ of $B_W(T_i,\delta) = \{T\in W: \|T_i - T\| \leq \delta\}$. Apply Lemma \ref{all difference-including domain and target} to the operators $A-T_i$ and $A - T_{ij}$, for $i=1,\ldots,l$ and $j=1,\ldots,m_i$, to find a separable subspace $Z$, normalized $9/8$-Schauder basic sequences $(z_k^i)_k$ in $\mathscr{F}_Z$ such that for $i=1,\ldots,l$, $j=1,\ldots,m_i$ if $y_k^{ij} = \|(A-T_{ij})z_k^i\|^{-1}(A-T_{ij})z_k^i$, for $k\in\mathbb{N}$, then $\|(A-T_{ij})z_k^i\| \geq \tilde c/3$ and  the sequence $(y_k^{ij})_k$ is in $\mathscr{F}_Z$. Iterate Lemma \ref{constant coming from unique joint sm} to find an infinite subset $L$ of $\mathbb{N}$ such that \eqref{same norms basically} is satisfied for $(z_k^i)_{k\in L}$ and $(y_{k}^{ij_i})_{k\in L}$ for all $i=1,\ldots,l$ and for any choice of $1\leq j_i\leq m_i$.
	 	
	 	Fix $k_1 <\cdots <k_l$ in $L$ and take a partition of unity $f_1,\ldots,f_l$ of $W$ subordinated to $T_1,\ldots,T_l$. That is, $f_i:W\to [0,1]$ is continuous, $\sum_{i=1}^lf_i(T) = 1$ for all $T$ in $W$ and $f_i(T_j) = \delta_{ij}$.  We define a continuous mapping $x:W\to X$ given by
	 	\[x(T) = \frac{\sum_{i=1}^lf_i(T)z^i_{n_i}}{\left\|\sum_{i=1}^lf_i(T)z^i_{n_i}\right\|}.\]
	 	Let $T$ be an arbitrary element of $W$ and if $I_T =  \{i=1,\ldots,l:$ with $\|T - T_i\| \leq \delta\}$, for $i\in I_T$, choose $1\leq j_i\leq m_i$ such that $\|T - T_{ij_i}\|\leq \eta$. Recall that $(z_{n_i}^i)_i$ is $(9/8)$-Schauder basic and therefore
	 	\begin{equation}
	 	\label{hatshu}
	 	\Big\|\sum_{i=1}^lf_i(T)z^i_{n_i}\Big\| \geq \frac{4}{9l}\sum_{i=1}^l\big|f_i(T)\big| = \frac{4}{9l}.
	 	\end{equation}
	 	We observe the following:
	 	\begin{align}
	 	\left\|\left(A-T\right)x(T)\right\| & = \frac{1}{\left\|\sum_{i\in I_T}f_i(T)z^i_{n_i}\right\|}\Big\|\sum_{i\in I_T}f_i(T)\left(A - T\right)z^i_{n_i}\Big\|\nonumber\\
	 	& \geq  \frac{\left\|\sum_{i\in I_T}f_i(T)\left(A - T_{ij_i}\right)z^i_{n_i}\right\|}{\left\|\sum_{i\in I_T}f_i(T)z^i_{n_i}\right\|} - \frac{\sum_{i\in I_T}\left|f_i(T)\right|\left\|T - T_{ij_i}\right\|}{\left\|\sum_{i\in I_T}f_i(T)z^i_{n_i}\right\|} \nonumber\\
	 	&\geq \frac{\left\|\sum_{i\in I_T}f_i(T)\left(A - T_{ij_i}\right)z^i_{n_i}\right\|}{\left\|\sum_{i\in I_T}f_i(T)x^i_{n_i}\right\|}  - \eta\frac{\sum_{i\in I_T}|f_i(T)|}{4/(9l)\sum_{i\in I_T}|f_i(T)|}\text{ (by \eqref{hatshu})}\nonumber\\
	 	& = \frac{\left\|\sum_{i\in I_T}f_i(T)\left\|\left(A-T_{ij_i}\right)z^i_{n_i}\right\|y^{ij}_{n_i}\right\|}{\left\|\sum_{i\in I_T}f_i(T)x^i_{n_i}\right\|} - \frac{9l\eta}{4}\nonumber\\
	 	& \geq \min_{1\leq i\leq l}\left\|\left(A-T_{ij_i}\right)z^i_{n_i}\right\|\frac{1}{D}\frac{\left\|\sum_{i\in I_T}f_i(T)x^i_{n_i}\right\|}{\left\|\sum_{i\in I_T}f_i(T)x^i_{n_i}\right\|} - \frac{9l\eta}{4} \text{ (by \eqref{same norms basically})}\nonumber\\
	 	&\geq \frac{\tilde c}{3D} - \frac{9l\eta}{4} = \frac{7D}{6D}\varepsilon - \frac{1}{12}\varepsilon = \frac{13}{12}\varepsilon.\label{yes success}
	 	\end{align}
	 	We now define a multivalued mapping $\phi:W \twoheadrightarrow W$ with
	 	\[\phi(T) = \left\{S\in W: \left\|\left(A - S\right)x(T)\right\| \leq \varepsilon\right\}.\]
	 	By assumption, the values of $\phi$ are non-empty and they are also closed and convex. It also easily follows that $\phi$ has a closed graph. Hence, from Theorem \ref{fixedpoint}, there exists $T\in W$ with $T\in\phi(T)$, i.e., $\|(A-T)x(T)\| \leq \varepsilon$. This contradicts \eqref{yes success} which completes the proof.
	 \end{proof}

	 The following lemma shows that if $X$ is a Banach space with a shrinking FDD that satisfies the UALS property, the finite codimensional subspaces of $X$ on which the approximations happen, can be assumed to be tail subspaces.

	 \begin{lem}\label{uals shrinking tail subspace}
	 	Let $X$ be a Banach space with a shrinking  FDD $(X_n)_n$ and $Y$ be a finite codimensional subspace of $X$. Then, for every $\varepsilon>0$, there exists a tail subspace $Z$ of $X$ such that $B_Z\subset B_Y+\varepsilon B_X$.
	 \end{lem}
	 \begin{proof}
	 	Let $x_1,\ldots,x_n\in B_X$ with $X=Y\oplus\spn\{x_1,\ldots,x_n\}$ and $x^*_1,\ldots,x^*_n\in {X^*}$ be such that $x^*_i(x_j)=\delta_{ij}$ for every $1\le i,j\le n$. Notice that $Y=\cap_{i=1}^n\ker x^*_i$. Since $(X_n)_n$ is a shrinking FDD, we may choose $n_0\in\N$ such that $\|x^*_i-P^*_{n_0}(x^*_i)\|<\varepsilon/l\|x_i\|$ for all $1\le i\le n$, and set $Z=\spn\cup_{n>n_0}X_n$. Pick a $z$ in the unit ball of $Z$ and set $x=\sum_{i=1}^l{x^*_i(z)x_i}/{\varepsilon}$. Then $|x^*_i(z)|<\varepsilon/l\|x_i\|$ and $x^*_i(x)=x^*_i(z)/\varepsilon$ for all $1\le i\le n$. Hence $\|x\|<1$ and $z-{\varepsilon} x\in\cap_{i=1}^n\ker x_i^*$, from which follows that $z\in B_Y+2\varepsilon B_X$.
	 \end{proof}
	
	 The next example demonstrates a shrinking FDD is necessary above to assume that the uniform approximation happens on tail subspaces. Let us recall that the basis of $\ell_1$ is not shrinking.
	
	 \begin{exmp}\label{counterell1}
	 	Let $(e_n)_n$ denote the unit vector basis of $\ell_1$ and consider the operator $A:\ell_1\to \ell_1$ with
	 	\[
	 	A\big((x_n)_n\big)=\sum_{n=1}^\infty x_{2n-1}e_1+\sum_{n=1}^\infty x_{2n}e_2
	 	\]
	 	and for $z\in\ell_1$, the operators $B_z^+, B_z^-:\ell_1\to\ell_1$ with
	 	\[
	 	B_z^+\big((x_n)_n\big)=\sum_{n=1}^\infty x_nz\quad\textnormal{and}\quad
	 	B_z^-\big((x_n)_n\big)=\Big(\sum_{n=1}^\infty x_{2n-1}-\sum_{n=1}^\infty x_{2n}\Big)z.
	 	\]
	 	Set $W=\textnormal{co}\big\{ B_z^\pm: z\in\spn\{ e_1,e_2 \}\;\textnormal{and}\;\|z\|\le1 \big\}$.	 	
	 	
	 	Let $x\in {\ell_1}$ with $\|x\|\le1$ and $A(x)=a_1e_1+a_2e_2$, where $a_1=\sum_{n=1}^\infty x_{2n-1}$ and $a_2=\sum_{n=1}^\infty x_{2n}$. Suppose that $A(x)\neq0$ and set $a=\max\{|a_1+a_2|,|a_1-a_2|\}$. Notice that $a=|a_1|+|a_2|$. If $a=|a_1+a_2|$, setting $z=\frac{1}{a_1+a_2}A(x)$, we have that $\|z\|=1$ and $B^+_z(x)=A(x)$. If $a=|a_1-a_2|$, then the same hold for $z=\frac{1}{a_1-a_2}A(x)$. Hence, for every $x\in {\ell_1}$ with $\|x\|\le1$, there is a $B\in W$ such that $\|(A-B)x\|=0$.
	 	
	 	Pick any $B\in W$ and $n_0\in\N$. Then there exists a convex combination in $W$ such that $B=\sum_{i=1}^{n}a_iB_{y_i}^++\sum_{i=1}^{m}b_i B_{z_i}^-$ and, for every $k\in\N$, we have that $B(e_{2k-1})= \sum_{i=1}^na_iy_i+\sum_{i=1}^mb_iz_i$ and $B(e_{2k})=\sum_{i=1}^na_iy_i-\sum_{i=1}^mb_iz_i$ and hence
	 	\[\begin{split}
	 	\Big\|\big(A-B\big)\frac{e_{2k-1}+e_{2k}}{2}\Big\|=\Big\|\frac{e_1+e_2}{2}-\sum_{i=1}^na_iy_i\Big\|\ge 1-\sum_{i=1}^na_i.
	 	\end{split}\]
	 	Similarly, $\|(A-B)\frac{e_{2k-1}-e_{2k}}{2}\|\ge \sum_{i=1}^na_i$ and thus, for any $k_0\in\N$ with $n_0\le2k_0-1$,  either $\|(A-B)\frac{e_{2k_0-1}+e_{2k_0}}{2}\|\ge{1}/{2}$ or $\|(A-B)\frac{e_{2k_0-1}-e_{2k_0}}{2}\|\ge{1}/{2}$. Therefore, we conclude that $\|(A-B)|_{\spn\{e_n:n\ge n_0 \}}\|\ge{1}/{2}$ while, for every $x$, in the unit ball of $\ell_1$ there exists a $B\in W$ such that $\|A(x)-B(x)\|=0$.
	 \end{exmp}

	 \subsection{The UALS property and duality}
	 We make a connection between the UALS property of a space and its dual. In particular, for reflexive spaces with an FDD we show that the UALS for $X$ is equivalent to the UALS for $X^*$. We also show that if $X$ has an FDD and $X^*$ has a unique $l$-joint spreading model with respect to a difference-including family, then $X$ must satisfy the UALS as well. This allows us to show indirectly that certain spaces, such as $\mathscr{L}_\infty$ spaces with separable dual, satisfy the UALS.

	 \begin{prop}
	 \label{dualize e-approximation}
	 	
	 Let $X$ be a Banach space, $A\in\mathcal{L}(X)$ and $W$ be a convex and WOT-compact subset of $\mathcal{L}(X)$. If there is an $\varepsilon>0$ such that $W$ $\varepsilon$-pointwise approximates $A$, then the set $W^* = \{T^*: T\in W\}$ $\varepsilon$-pointwise approximates $A^*$.
	 \end{prop}
	
	 \begin{proof}
	  If we assume that the conclusion is false, then there exists $x^*$ in the unit sphere of $X^*$ and $\delta>0$ such that if $W^*x^* = \{T^*x^*: T\in W\}$, then we have that $\mathrm{dist}(A^*x^*,W^*x^*) \geq \varepsilon+\delta$. As $W^*x^*$ is a convex and $w^*$-compact subset of $X^*$, a separation theorem yields that there exists $x$ in the unit sphere of $X$ such that $x(A^*x^*) + (\varepsilon+\delta/2) \leq \inf_{T\in W}x(T^*x^*)$ or \[\|Ax - Tx\| \geq x^*(Tx - Ax)\geq \varepsilon + \delta/2\]
	  for all $T\in W$.
	  \end{proof}
	
	  \begin{rem}
	  The compactness of $W$ is necessary in Proposition \ref{dualize e-approximation}. To see this, consider the case when $X = \ell_1$, $A$ is the identity operator, and $W$ is the closed convex hull of all natural projections onto finite subsets of $\mathbb{N}$ with respect to the unit vector basis.
	  \end{rem}
	
	  We state the main results and prove them afterwards.
	
	  \begin{thm}
	  \label{uals and reflexivity}
	  Let $X$ be a reflexive Banach space with an FDD. Then $X$ satisfies the UALS if and only if $X^*$ does.
	  \end{thm}

	 \begin{thm}
	 \label{theorem UALS by duality}
	 Let $X$ be a Banach space with an FDD. Assume that there exist a uniform constant $C>0$ and, for every separable subspace $Z$ of $X^*$,  a difference-including family $\mathscr{F}_Z$ of normalized sequences in $X^*$ such that $Z$ admits a $C$-uniformly unique $l$-joint spreading model with respect to $\mathscr{F}_Z$. Then $X$ satisfies the UALS property.
	 \end{thm}
	
	 Recall that results from \cite{H}, \cite{HS}, \cite{LS}, and \cite{Ste} yield that if $X$ is an infinite dimensional $\mathscr{L}_\infty$-space with separable dual then $X^*$ is isomorphic to $\ell_1$. Also this is the case if and only if $\ell_1$ is not isomorphic to a subspace of $X$. As it was proven in \cite{FOS}, every Banach space with separable dual embeds in a $\mathscr{L}_\infty$ space with separable dual.
	
	  \begin{cor}
	  \label{some duality stuff}
	 Every $\mathscr{L}_\infty$-space with separable dual satisfies the UALS property. In particular:
	 \begin{itemize}
	 \item[(i)] every hereditarily indecomposable $\mathscr{L}_\infty$-space satisfies the UALS,
	 \item[(ii)] every Banach space with separable dual embeds in a space that satisfies the UALS, and
	 \item[(iii)] for every countable compact metric space $K$, the space $C(K)$ satisfies the UALS.
	 \end{itemize}
	 \end{cor}
	
	 \begin{cor}
	  If $X$ is Banach space such that $X^*$ is an Asymptotic $\ell_p$ space, for some $1\leq p\leq \infty$, then every quotient of $X$ with an FDD satisfies the UALS.
	  \end{cor}

	 \begin{lem}
	  \label{tail in front and in the back}
	  Let $X$ be a Banach space, let $R:X\to X$ be a finite rank operator, and let $Q = I - R$. If $T$ is in $\mathcal{L}(X)$, then there exists a subspace $Y$ of $X$ of finite codimension such that $\|T|_Y\|_{\mathcal{L}(Y,X)} \le\|QT\|$.
	  \end{lem}
	
	  \begin{proof}
	
	  Since $RT$ is a finite rank operator the subspace $Y = \ker RT$ is of finite codimension. Then, $\|T|_Y\| \leq \|RT|_Y\| + \|QT|_Y\|\leq \|QT\|$.
	  \end{proof}

	 \begin{proof}[Proof of Theorem \ref{uals and reflexivity}]
	 It is clearly enough to show one implication. Let us assume that $X^*$ satisfies the UALS with constant $C>0$ and let $A\in\mathcal{L}(X)$ and $W$ be a compact and convex subset of $\mathcal{L}(X)$ that $\varepsilon$-approximates $A$. Then by Proposition \ref{dualize e-approximation} we have that the set $W^* = \{T^*: T\in W\}$ $\varepsilon$-approximates $A^*$ and so there exists a subspace $Z$ of $X^*$ of finite codimension such that $\|(T^*-A^*)|_Z\| \le C\varepsilon$. By Lemma \ref{uals shrinking tail subspace}, and perhaps some additional error, we may assume that $Z$ is a tail  subspace with an associated projection $Q^*_{n}$ and hence \[\|Q_n(T-A)\|=\|(T^*-A^*)Q^*_{n}\| \le C \|Q_{n}\|\varepsilon.\] Applying Lemma \ref{tail in front and in the back}, we may find a subspace $Y$ of $X$ of finite codimension such that $\|(T-A)|_Y\| \le\|Q_n(T-A)\|$ and therefore $\|(T-A)|_Y\|\le C \|Q_{n}\|\varepsilon$.
	 \end{proof}

	  \begin{lem}
	 	\label{DUAL all difference-including domain and target}
Let $X$ be a Banach space with a bimonotone FDD  and let $(Q_n)_n$ denote the basis tail projections (i.e. $Q_n = I-P_n$, for all $n\in\mathbb{N}$). Assume that for every separable subspace $Z$ of $X^*$ we have a difference-including collection $\mathscr{F}_Z$ of normalized Schauder basic sequences in $Z$. Let $T_1,\ldots,T_l$ be bounded linear operators on $X$ and assume that there is $c>0$ such that, for every $n\in\mathbb{N}$ and every $i=1,\ldots,l$, we have that $\|T^*_iQ^*_{n}\| \geq c$. Assume moreover that, for some $0<\delta<c$ and every $i=1,\ldots,l$, we have  $T_{i1},\ldots,T_{im_i}$ in $\mathcal{L}(X)$ with $\|T_{ij} - T_i\| \leq \delta$ for $j=1,\ldots,m_i$. Then, if $\tilde c = c-\delta$, there exist a separable subspace $Z$ of $X^*$ and normalized  sequences $(z_k^i)_k$, $i=1,\ldots,l$ in $\mathscr{F}_Z$ such that
	 	
	 	\begin{itemize}
	 		
	 		\item[(i)] for any $n_1 < \cdots <n_l$, the sequence $(z_{n_i}^i)_{i=1}^l$ is $(9/8)$-Schauder basic,
	 		
	 		\item[(ii)] $\|T^*_{ij}z^i_k\| > \tilde c/3$ for   $i=1,\ldots,n$, $j=1,\ldots,m_i$ and $k\in\mathbb{N}$, and
	 		\item[(iii)] if $y_k^{ij} = \|T^*_{ij}z_k^i\|^{-1}T^*_{ij}z_k^i$, then $(y_k^{ij})_k$ is in $\mathscr{F}_Z$ for $i=1,\ldots,n$, $j=1,\ldots,m_i$.
	 		
	 	\end{itemize}
	 \end{lem}
	
	 \begin{proof}
	 For every $i=1,\ldots,l$, we choose a normalized sequence $(x_n^{i*})_n$ such that $\|T^*_iQ^*_{n}x_n^{i*}\| \geq c - (c-\delta)/4$. Since the FDD is bimonotone we may assume that $\min\supp(x_n^{i*}) > n$ for all $n\in\mathbb{N}$ and $1\leq i\leq l$ and that $\|T_i^*x_n^{i*}\| \geq c - (c-\delta)/4$. This means that all sequences $(x_n^{i*})_n$, $(T^*x_n^{i*})_n$, for $1\leq i\leq l$, are $w^*$-null (by using $w^*$-continuity). We may now apply reasoning identical to that used in Lemmas \ref{mazur with operators} and \ref{all difference-including domain and target} to achieve the desired conclusion.
	 \end{proof}

	 \begin{proof}[Proof of Theorem \ref{theorem UALS by duality}]
	 We renorm the space $X$ so that its FDD is bimonotone. Let $D > 2K^2$, i.e. a constant for which the conclusion of Lemma \ref{constant coming from unique joint sm} can be applied to all families $\mathscr{F}_Z$ for all separable subspaces $Z$ of $X^*$. Set $C = 14D$. We will show that $X$ satisfies the UALS with constant $C$. Let $A\in\mathcal{L}(X)$ and $W$ be a convex compact subset of $\mathcal{L}$ of $X$ that $\varepsilon$-approximates $A$. It is sufficient to find $T\in W$ and $n_0\in\mathbb{N}$ such that $\|(T^*-A^*)Q^*_{n_0}\| < C\varepsilon$. Indeed, then $\|Q_{n_0}(A-T)\|  = \|(T^*-A^*)Q^*_{n_0}\| < C\varepsilon$ and by  Lemma \ref{tail in front and in the back} we will be done. If we assume that the conclusion is false, we may follow the proof of Theorem \ref{the theorem} to the letter, only replacing Lemma \ref{all difference-including domain and target} with Lemma \ref{DUAL all difference-including domain and target}, to reach the desired conclusion.
	 \end{proof}

	 \subsection{Spaces Failing the UALS Property}
In this section we present an archetypal example of a reflexive Banach space $\mathcal{X}$ that fails the UALS and admits a unique spreading model isometric to $\ell_2$. The proof that $\mathcal{X}$ fails the property is based on the the fact that it does not admit a uniformly unique joint spreading model. This reasoning may then be modified and utilized to show that classical spaces such as $L_p[0,1]$, $1\leq p\leq \infty$ and $p\neq 2$, and $C(K)$, for uncountable compact metric spaces $K$, fail the UALS.

	 \begin{dfn}\label{definition fails uals}
	 	For each $n\in\N$, we set $X_n=(\sum_{i=1}^{2n}\oplus\ell_2)_1$ and $Y_n=(\sum_{i=1}^{2n}\oplus\ell_2)_\infty$ and let $\mathcal{X}=(\sum\oplus X_n\oplus Y_n)_2$.
	 \end{dfn}

 		For a vector $x$ in $\mathcal{X}$, we write $x = \sum_{n=1}^\infty x_n+ y_n$ to mean that $x_n\in X_n$ and $y_n \in Y_n$ and  $x_n = \sum_{j=1}^nx_{n(j)}$, $y_n = \sum_{j=1}^{2n}y_{n(j)}$ to denote the coordinates of each $x_n$ and $y_n$ with respect to the natural decomposition of $X_n$ and $Y_n$ respectively. Under this notation we compute the norm of $x$ as follows:
 	\[\|x\|^2 = \sum_{n=1}^\infty\left(\bigg(\sum_{j=1}^{2n}\left\|x_{n(j)}\right\|\bigg)^2 +\bigg(\max_{1\leq j\leq 2n}\left\|y_{n(j)}\right\|\bigg)^2\right). \]
 	By taking an orthonormal basis for each corresponding $\ell_2$-component of the space $X_n$ as well as of the space $Y_n$ and taking a union over all $n\in\mathbb{N}$ and for $j=1,\ldots,2n$, we obtain a 1-unconditional basis for the space $\mathcal{X}$. Henceforth, when we say $(x_k)_k$ is a block sequence in $\mathcal{X}$, it will be understood that this is with respect to a fixed enumeration of the aforementioned basis.
	
	 \begin{prop}\label{fails uals}
	 	The space $\mathcal{X}$ fails the UALS property.
	 \end{prop}
	
	 \begin{proof}	
	Assume that $\mathcal{X}$ satisfies the UALS with constant $C>0$ and pick $n\in \N$ with ${C}/{n}<\frac{1}{2}$. For $G\subset\{1,\ldots,2n\}$, consider the bounded operator $I_G:X_n\to Y_n$ with  $I_G(\sum_{i=1}^{2n}x_i)=\sum_{i\in G}x_i$ and set $A_n=I_{\{1,\ldots,2n\}}$ and $W_n=\co\{I_G:\#G=n\}$. Let $x\in X_n$ with $x=\sum_{i=1}^{2n}x_i$ and $\|x\|=1$, that is $\sum_{i=1}^{2n}\|x_i\|=1$, and $\sigma$ be a permutation  of $\{1,\ldots,2n\}$ such that $\|x_{\sigma(1)}\|\ge\ldots\ge\|x_{\sigma(2n)}\|$. Then notice that $\|x_{\sigma(n+1)}\|\le\frac{1}{n+1}$ and hence $\| A_n(x)-I_G(x)\|\le\frac{1}{n+1}$, for $G=\{\sigma(1),\ldots,\sigma(n)\}$.
	
	The basis $(e_n)_n$ of $\mathcal{X}$ is shrinking, since $\mathcal{X}$ is reflexive, and therefore Lemma \ref{uals shrinking tail subspace} yields a tail subspace $Y={\spn\{e_n:n\ge n_0\}}$ of $X$ such that $\|(A_n-B)|_Y\|<C/n$ for some $B\in W$. Then $B$ is a convex combination $B=\sum_{i=1}^k\lambda_iI_{G_i}$ and we have that $\int\sum_{i=1}^k\lambda_i\scalebox{1.2}{$\chi$}_{G_i}=\frac{1}{2}$, where the integral is with respect to the normalized counting measure on $\{1,\ldots,2n\}$. Hence there exists a $1\le j \le 2n$ such that $\sum_{i=1}^{k}\lambda_i\scalebox{1.2}{$\chi$}_{G_i}(j)\le\frac{1}{2}$. Pick any $x\in X_{n(j)}$ with $\|x\|=1$ and $\supp(x)\ge n_0$ and notice that $\|A_n(x)-B(x) \|\ge  1-\sum_{i=1}^k\lambda_i\scalebox{1.2}{$\chi$}_{G_i}(j)$. Thus $\|(A_n-B)|_Y\|\ge\frac{1}{2}$, which is a contradiction.
\end{proof}
	
	The space $\mathcal{X}$ is a first example of a space failing the UALS property. As we show next, it admits a uniformly unique spreading model while it fails to admit a uniformly unique $l$-joint spreading model. We start with the following lemmas.
	
	 \begin{lem}\label{lem1}
	 	Let $(x^k)_k$ be a block sequence in $\mathcal{X}$ with $x^k=\sum_{n=n_0}^{n_1}x^k_n+y^k_n$ and assume that $\|x^{k_1}_{n(j)}\|=\|x^{k_2}_{n(j)}\|$ and $\|y^{k_1}_{n(j)}\|=\|y^{k_2}_{n(j)}\|$ for every $k_1,k_2\in\N$, $n_0\le n\le n_1$ and $1\le j\le 2n$. Set $\varepsilon=\|x^k\|$, for $k\in\N$. Then, for all $m\in\N$ and $\lambda_1,\ldots,\lambda_m\in\R$, we have that $\|\sum_{k=1}^m\lambda_kx^k\|=\varepsilon(\sum_{k=1}^m\lambda^2_k)^{\frac{1}{2}}$.
	 \end{lem}
	 \begin{proof}
	 	Let $k_0\in\N$. For every $k\in\N$ and $n_0\le n\le n_1$, since $(x^k)_k$ is block, we have that
	 	\[\Big\|\sum_{k=1}^m \lambda_kx^k_{n(j)}\Big\|=\Bigg(\sum_{k=1}^m \lambda_k^2\big\|x^k_{n(j)}\big\|^2\Bigg)^{\frac{1}{2}}=\big\|x^{k_0}_{n(j)}\big\|\Bigg(\sum_{k=1}^m \lambda_k^2\Bigg)^{\frac{1}{2}}.\]
	 	We thus calculate
	 	\begin{equation}\label{l1eq1} \Big\|\sum_{k=1}^m \lambda_kx^k_{n}\Big\|=\sum_{j=1}^{2n}\Big\|\sum_{k=1}^m \lambda_kx^k_{n(j)}\Big\|=\big\|x^{k_0}_{n}\big\|\Bigg(\sum_{k=1}^m \lambda_k^2\Bigg)^{\frac{1}{2}}\end{equation}
	 	and similarly
	 	\begin{equation}\label{l1eq2} \Big\|\sum_{k=1}^m \lambda_ky^k_{n}\Big\|=\max_{1\le j\le 2n}\Big\|\sum_{k=1}^m \lambda_ky^k_{n(j)}\Big\|=\big\|y^{k_0}_{n}\big\|\Bigg(\sum_{k=1}^m \lambda_k^2\Bigg)^{\frac{1}{2}}.\end{equation}
	 	Finally, using \eqref{l1eq1} and \eqref{l1eq2}, we conclude that
	 	\begin{equation}\begin{split}\label{l1eq3}
	 	\Big\|\sum_{k=1}^m \lambda_kx^k\Big\|^2&=\sum_{n=n_0}^{n_1}\bigg(\Big\|\sum_{k=1}^m \lambda_kx^k_{n}\Big\|^2+\Big\|\sum_{k=1}^m \lambda_ky^k_{n}\Big\|^2\bigg)\\
	 	&=\sum_{k=1}^m \lambda_k^2\sum_{n=n_0}^{n_1}\Big(\big\|x^{k_0}_n\big\|^2+\big\|y^{k_0}_n\big\|^2\Big)=\sum_{k=1}^m \lambda_k^2\big\|x^{k_0}\big\|^2.
	 	\end{split}\end{equation}
	 \end{proof}
	
	 \begin{lem}\label{lem2}
	 	Let $(n_k)_{k\ge 0}$ be an increasing sequence of naturals and $(x^k)_k$ be a block sequence in $\mathcal{X}$ such that
	 	\begin{enumerate}
	 		\item[(i)] There exist $c_1,c_2>0$ such that $c_1\le \|x^k\|\le c_2$ for every $k\in\N$.
	 		\item[(ii)] $x^k=\sum_{n=n_0}^{n_1}(x^k_n+y^k_n)+\sum_{n=n_k+1}^{n_{k+1}}(x^k_n+y^k_n)$ for every $k\in\N$.
	 		\item[(iii)] $\|x^{k_1}_{n(j)}\|=\|x^{k_2}_{n(j)}\|$ and $\|y^{k_1}_{n(j)}\|=\|y^{k_2}_{n(j)}\|$ for every $k_1,k_2\in\N$, $n_0\le n\le n_1$ and $1\le j\le 2n$.		
	 	\end{enumerate}
	 	Then, for all $m\in\N$ and $\lambda_1,\ldots,\lambda_m\in\R$, we have that
	 	\[
	 	c_1\Bigg(\sum_{k=1}^m \lambda_k^2\Bigg)^{\frac{1}{2}}\le \Big\|\sum_{k=1}^m \lambda_kx^k\Big\|\le c_2\Bigg(\sum_{k=1}^m \lambda_k^2\Bigg)^{\frac{1}{2}}.
	 	\]
	 \end{lem}
	 \begin{proof}
	 	Using \eqref{l1eq3}, we have that
	 	\begin{equation*}\begin{split}
	 	\bigg\|\sum_{k=1}^m \lambda_kx^k\bigg\|^2&=\sum_{n=n_0}^{n_1}\bigg(\Big\|\sum_{k=1}^m \lambda_kx^k_{n}\Big\|^2+\Big\|\sum_{k=1}^m \lambda_ky^k_{n}\Big\|^2\bigg)+\sum_{k=1}^m\sum_{n=n_k+1}^{n_{k+1}}\bigg(\big\|\lambda_kx^k_n\big\|^2+\big\|\lambda_ky^k_n\big\|^2\bigg)\\
	 	&=\sum_{k=1}^m \lambda_k^2\sum_{n=n_0}^{n_1}\Big(\big\|x^k_n\big\|^2+\|y^k_n\big\|^2\Big)+\sum_{k=1}^m\lambda_k^2\sum_{n=n_k+1}^{n_{k+1}}\Big(\big\|x^k_n\big\|^2+\|y^k_n\big\|^2\Big)\\&=\sum_{k=1}^m \lambda_k^2\bigg(\sum_{n=n_0}^{n_1}\Big(\big\|x^k_n\big\|^2+\|y^k_n\big\|^2\Big)+\sum_{n=n_k+1}^{n_{k+1}}\Big(\big\|x^k_n\big\|^2+\|y^k_n\big\|^2\Big)\bigg)=\sum_{k=1}^m \lambda_k^2\big\|{x}^k\big\|^2
	 	\end{split}\end{equation*}
	 	which, due to (i), yields the desired result.
	 \end{proof}
	
	 \begin{prop}
	 	Let $(x^k)_k$ be a normalized block sequence in $\mathcal{X}$. For every $\varepsilon>0$, $(x^k)_k$ has a subsequence $(x^{k_i})_i$ such that for every $m\in\N$ and $\lambda_1,\ldots,\lambda_m\in\R$
	 	\[
	 	(1-\varepsilon)\Bigg(\sum_{i=1}^m \lambda_i^2\Bigg)^{\frac{1}{2}}\le\Big\|\sum_{i=1}^m \lambda_i x^{k_i}\Big\|\le (1+\varepsilon)\Bigg(\sum_{i=1}^m \lambda_i^2\Bigg)^{\frac{1}{2}}.
	 	\]
	 \end{prop}
	 \begin{proof}
	 	We choose $L\in[\N]^\infty$ such that $\lim_{k\in L}\|x^k_{n(j)}\|=a_{n,j}$ and $\lim_{k\in L}\|y^k_{n(j)}\|=b_{n,j}$ for all $n\in \N$ and $1\le j\le 2n$. Set $\lim_{k\in L} \|x^k_n\|=a_n$ and $\lim_{k\in L} \|y^k_n\|=b_n$. Since $\sum_{n=1}^\infty \|x^k_n\|^2+\|y^k_n\|^2\le 1$ for every $k\in\N$, then  $\sum_{n=1}^\infty a_n^2+b_n^2\le 1$.
	 	
	 	Let $(\varepsilon_i)_i$ and $(\delta_i)_i$ be sequences of positive reals such that $\sum_{i=1}^\infty\varepsilon_i<\varepsilon$ and $\sum_{i=1}^\infty\delta_i<\varepsilon$. We then choose, by induction, $(n_i)_i\subset \N$ and $(k_i)_i\subset L$ increasing sequences such that the following hold for every $i\in \N$.
	 	\begin{enumerate}
	 		\item[(i)] $n_i > \max\{n:x^{k_{i-1}}_n\neq 0\;\text{or }y^{k_{i-1}}_n\neq 0 \}$ when
	 		$i>1$.
	 		\item[(ii)] $\sum_{n>n_i}a_n^2+b_n^2<\varepsilon_i$.
	 		\item[(iii)] $\sum_{n=1}^{n_i}\sum_{j=1}^{2n}\big| \|x^{k_i}_{n(j)}\|-a_{n,j}\big|^2+\big| \|y^{k_i}_{n(j)}\|-b_{n,j}\big|^2<\delta_i$.
	 	\end{enumerate}
	 	For each $i\in\N$, due to (iii), we may assume that $\|x^{k_i}_{n(j)}\|=a_{n,j}$ and $\|y^{k_i}_{n(j)}\|=b_{n,j}$ for all $1\le n\le n_i$ and $1\le j\le 2n$, with an error $\delta_i$. Then Lemma \ref{lem1} yields that
	 	\begin{equation*}\begin{split}
	 	\Big\|\sum_{i=2}^m\sum_{j=i}^m\lambda_j\sum_{n=n_{i-1}+1}^{n_i}\big(x^{k_j}_n+y^{k_j}_n\big)\Big\|\le \Bigg(\sum_{i=2}^m\varepsilon_{i-1}\sum_{j=i}^m \lambda_j^2\Bigg)^{\frac{1}{2}}+\Bigg(\sum_{i=2}^m\delta_{i}\sum_{j=i}^m \lambda_j^2\Bigg)^{\frac{1}{2}}.
	 	\end{split}\end{equation*}
	 	Hence, applying Lemma \ref{lem2},  we calculate
	 	\begin{equation*}\begin{split}
	 	\Big\|\sum_{i=1}^m \lambda_i x^{k_i}\Big\|&\ge\Big\|\sum_{i=1}^m \lambda_i\Big(\sum_{n=1}^{n_1} \big(x^{k_i}_n+y^{k_i}_n\big)+\sum_{n=n_i+1}^{n_{i+1}}\big(x^{k_i}_n+y^{k_i}_n\big)\Big)\Big\|-2\sqrt{\varepsilon}\Bigg(\sum_{i=1}^m \lambda_i^2\Bigg)^{\frac{1}{2}}\\
	 	&\ge\sqrt{1-4\varepsilon}\Bigg(\sum_{i=1}^m a_i^2\Bigg)^{\frac{1}{2}}-2\sqrt{\varepsilon}\Bigg(\sum_{i=1}^m a_i^2\Bigg)^{\frac{1}{2}}
	 	\end{split}\end{equation*}
	 	and
	 	\begin{equation*}\begin{split}
	 	\Big\|\sum_{i=1}^m \lambda_i x^{k_i}\Big\|&\le\Big\|\sum_{i=1}^m \lambda_i\Big(\sum_{n=1}^{n_1} \big(x^{k_i}_n+y^{k_i}_n\big)+\sum_{n=n_i+1}^{n_{i+1}}\big(x^{k_i}_n+y^{k_i}_n\big)\Big)\Big\|+2\sqrt{\varepsilon}\Bigg(\sum_{i=1}^m \lambda_i^2\Bigg)^{\frac{1}{2}}\\
	 	&\le\sqrt{1+4\varepsilon}\Bigg(\sum_{i=1}^m \lambda_i^2\Bigg)^{\frac{1}{2}}+2\sqrt{\varepsilon}\Bigg(\sum_{i=1}^m \lambda_i^2\Bigg)^{\frac{1}{2}}.
	 	\end{split}\end{equation*}
	 \end{proof}
	
	 \begin{cor}
	 	Every spreading model generated by a basic sequence in $\mathcal{X}$ is isometric to $\ell_2$ and hence $\mathcal{X}$ admits a uniformly unique spreading model with respect to $\mathscr{F}(X)$.
	 \end{cor}
	
	 \begin{rem}
	 	Using similar arguments we may show that every $l$-joint spreading model generated by a basic sequence in $\mathcal{X}$ is isomorphic to $\ell_2$, while this does not happen with a uniform constant and as already shown $\mathcal{X}$ fails the UALS property. This shows a strong connection between the UALS and spaces with uniformly unique joint spreading models, which fails when the space only admits a uniformly unique spreading model.
	 	
	 	As mentioned in Section 3, the space from \cite{AM1} is another example of a space that admits a uniformly unique spreading model and fails to have a uniform constant for which all of its $l$-joint spreading models, for every $l\in\N$, are equivalent. This space however satisfies the stronger property that every one of its subspaces does not admit a uniformly unique $l$-joint spreading model, contrary to the space $\mathcal{X}$ which contains $\ell_2$.	 	
	 \end{rem}

 	Motivated by the definition of space $\mathcal{X}$, we modify the above arguments to show that every $L_p[0,1]$, for $1\le p\le\infty$, as well as the space $C(K)$ for an uncountable compact metric space $K$ fail the UALS.
	
	 \begin{prop}
	 	For every $1<p<q<\infty$, the spaces $(\sum\oplus\ell_p)_q$ and $(\sum\oplus\ell_q)_p$ fail the UALS property.
	 \end{prop}
	 \begin{proof}
	 	For each $n\in\N$, we set $X_n=(\sum_{i=1}^{2n}\oplus\ell_p)_p$ and $Y_n=(\sum_{i=1}^{2n}\oplus\ell_p)_q$ and $X=(\sum X_n\oplus Y_n)_q$. Assume that $X$ satisfies the  UALS property with constant $C>0$ and pick $n\in\N$ with $C/n^{r}<\frac{1}{2}$, where $r=(q-p)/pq$.
	 	
	 	For every $G\subset\{1,\ldots,2n\}$, consider the operator $I_G:X_n\to Y_n$ such that $I_G(\sum_{i=1}^{2n}a_ix_i)=\sum_{i\in G}a_ix_i$ and set $A_n=I_{\{1,\ldots,2n\}}$ and $W_n=\co\{I_G:\#G=n \}$. Let $x\in X_n$ with $x=\sum_{i=1}^{2n}x_i$ and $\sum_{i=1}^{2n}\|x_i\|^p=1$ and let $\sigma$ be a permutation of $\{1,\ldots,2n\}$ such that $\|x_{\sigma(1)}\|^p\ge\ldots\ge\|x_{\sigma(2n)}\|^p$. Hence for $G=\{\sigma(1),\ldots,\sigma(n)\}$ we have that $\|A_n(x)-I_G(x)\|<1/{n}^{r}$ and using the same arguments as in the proof of Proposition \ref{fails uals}, we derive a contradiction. The case of $(\sum\oplus\ell_q)_p$ is similar.
	 \end{proof}
	
	 \begin{rem}
	 It is immediate that if some infinite dimensional complemented subspace of Banach space $X$ fails the  UALS property, then the same holds for $X$.
	 \end{rem}
	
	 \begin{prop}
	 	The space $L_p[0,1]$, for $1<p<\infty$ and $p\neq 2$, fails the UALS property.
	 \end{prop}
	 \begin{proof}
	 	Recall that, as follows from Khintchine's inequality, $\ell_2$ embeds isomorphically as a complemented subspace into $L_p[0,1]$, for all $1<p<\infty$. If $p>2$, for each $n\in\N$ set $X_{n}=(\sum_{i=1}^{2n}\oplus\ell_2)_2$ and $Y_{n}=(\sum_{i=1}^{2n}\oplus\ell_2)_p$ and if $p<2$ we then  set $X_{n}=(\sum_{i=1}^{2n}\oplus\ell_2)_p$ and $Y_{n}=(\sum_{i=1}^{2n}\oplus\ell_2)_2$. Then, following the proof of Proposition \ref{fails uals}, we may show that the space $X=(\sum\oplus X_n\oplus Y_n)_p$ fails the UALS and since it is complemented into $L_p[0,1]=(\sum\oplus L_p[0,1])_p$, the latter also fails the property.
	 \end{proof}
	
	 \begin{prop}
	 	The space $L_1[0,1]$ fails the UALS property.
	 \end{prop}
	 \begin{proof}
	 	Assume that $L_1[0,1]$ satisfies the UALS with constant $C>0$ and pick $n\in\N$ with $C/n<\frac{1}{9}$. Set $X_n=(\sum_{i=1}^{2n}\oplus\ell_1)_1$ and $Y_n=(\sum_{i=1}^{2n}\oplus\ell_2)_2$. Since $\ell_1$ is isometric to a complemented subspace of $L_1[0,1]$, then the same holds for $(\sum\oplus X_n)_1$. Moreover, Khintchine's inequality yields that $\ell_2$ embeds isomorphically into $L_1[0,1]$ and hence so does $(\sum \oplus Y_n)_2$.
	 	
	 	For every $G\subset \{1,\ldots,2n\}$, consider the operator $I_G:X_n\to Y_n$ such that $I_G(\sum_{i=1}^{2n}a_ix_i)=\sum_{i\in G}a_ix_i$ and set $A_n=I_{\{1,\ldots,2n\}}$ and $W_n=\co\{I_G:\#G=n\}$. As above, for all $x\in X_n$ with $\|x\|\le1$, we may find $G\subset\{1,\ldots,2n\}$ such that $\|A_n(x)-I_G(x)\|<{1}/{n}$. Let $B=\sum_{i=1}^k\lambda_iI_{G_i}$ in $W_n$ and $Y$ be a finite codimensional subspace of $L_1[0,1]$ such that $\|(A_n-B)|_Y\|<C/n$ and choose, as in the proof of Proposition \ref{fails uals}, $1\le j\le 2n$ with $\sum_{i=1}^{k}\lambda_i\scalebox{1.2}{$\chi$}_{G_i}(j)\le\frac{1}{2}$. Let $x^*_1,\ldots,x^*_l\in L_{\infty}[0,1]$ with $Y=\cap_{i=1}^l\ker x^*_i$. Denote by $(e_m)_m$ the basis of $X_{n(j)}$ and choose $M\in[\N]^\infty$ such that $(x^*_i(e_m))_{m\in M}$ converges for all $1\le i\le l$. Using Lemma \ref{bslem} we choose $m_1,m_2\in M$ such that $d(x,Y)<\frac{1}{8}$, for $x=(e_{m_1}-e_{m_2})/2$. Then $\|A_n(x)-B(x)\|\ge\frac{1}{4}$ and hence $\|(A_n-B)|_Y\|\ge\frac{1}{9}$, which is a contradiction.
	 \end{proof}
	
	 \begin{prop}\label{capital L infty uals}
	 	The space $L_\infty[0,1]$ fails the UALS property.
	 \end{prop}
	 \begin{proof} Fix $n\in\N$. The $\sigma$-algebra $\mathcal{B}[0,1]$ of all Borel sets of $[0,1]$ is homeomorphic to that of $[0,1]^{2n}$ and hence $L_\infty[0,1]$ is isometric to $L_\infty[0,1]^{2n}$. For $1\le i \le 2n$, denote by $\mathcal{B}_i$ the $\sigma$-algebra generated by  $\{B\in \prod_{i=1}^{2n}\mathcal{B}[0,1]:B_j=[0,1]\;\text{for }j>i \}$ and for $f\in L_\infty[0,1]^{2n}$ set $E_i(f)=E[f|\mathcal{B}_i]$ and consider the operator $\Delta_i:L_\infty[0,1]^{2n}\to  L_2([0,1]^i,\otimes_{j\le i}\lambda)$ with $\Delta_i(f)=E_i(f)-E_{i-1}(f)$, where $E_0(f)=0$ and $\lambda$ denotes the Lebesgue measure on $[0,1]$.

 For every $G\subset\{1,\ldots,2n\}$, let $\Delta_G:L_\infty[0,1]^{2n}\to (\sum_{i=1}^{2n}\oplus L_2([0,1]^i,\otimes_{j\le i}\lambda) )_\infty$ with $\Delta_G=\sum_{i\in G}\Delta_i$ and set $A_n=\Delta_{\{1,\ldots,2n \}}$ and  $W_n=\co\{\Delta_G:\#G=n \}$. Observe that $(\sum_{i=1}^{2n}\oplus L_2([0,1]^i,\otimes_{j\le i}\lambda) )_\infty$ embeds isometrically into $L_\infty[0,1]^{2n}$ and hence we have that $\Delta_G:L_\infty[0,1]\to L_\infty[0,1]$. Let $f\in L_\infty[0,1]^{2n}$ and notice that $(E_i(f))_{i=1}^{2n}$ is a martingale, since $\mathcal{B}_i$ is a subalgebra of $\mathcal{B}_j$ for every $1\le i<j\le 2n$. Then for the martingale differences $(\Delta_{i}(f))_{i=1}^{2n}$ the Burkholder inequality \cite{B} yields a $c_2>0$ such that
 \begin{equation}\label{burkholder}\Bigg(\int_{0}^1 \sum_{i=1}^{2n}\Delta_i(f)^2\Bigg)^{\frac{1}{2}}\le c_2\|f\|_2.\end{equation}

 		{ \underline{\em {Claim} 1} }  : For every $\varepsilon>0$, there exists $n_0\in\N$ such that, for every $n\ge n_0$ and $f\in L_\infty[0,1]^{2n}$, there is a $B\in W_n$ such that $\|(A_n-B)f\|\le\varepsilon\|f\|$.
		\begin{proof}[Proof of Claim 1]\renewcommand{\qedsymbol}{}
            Pick $n_0\in\N$ such that $c_2/\sqrt{n_0}<\varepsilon$. Let $n\ge n_0$ and $f$ in $L_\infty[0,1]^{2n}$ with $\|f\|=1$. Then, as a direct consequence of \eqref{burkholder}, we have that $\#\{i:\|\Delta_i(f)\|_2>c_2/\sqrt{n+1} \}\le n$. Let $\sigma$ be permutation of $\{1,\ldots,2n\}$ such that $\|\Delta_{\sigma(1)}(f)\|_2\ge\ldots\ge\|\Delta_{\sigma(2n)}(f)\|_2$. Hence for $G=\{\sigma(1),\ldots,\sigma(n)\}$, we conclude that $\|(A_n-\Delta_{G})f\|< c_2/\sqrt{n}$ and this yields the desired result.
		\end{proof}

 		{ \underline{\em {Claim} 2} }  : For every $n\in\N$, every finite codimensional subspace $Y$ of $L_\infty[0,1]^{2n}$ and $B\in W_n$, we have that $\|(A-B)|_Y\|\ge 1/9$.
		\begin{proof}[Proof of Claim 2]\renewcommand{\qedsymbol}{}
            There exist $x^{*}_1,\ldots,x^{*}_l\in (L_\infty[0,1]^{2n})^*$ with $Y=\cap_{i=1}^l\ker x^{*}_i$ and also $B$ is a convex combination $B=\sum_{i=1}^k\lambda_i\Delta_{G_i}$ in $W_n$. Then, as in Proposition \ref{fails uals}, choose $1\le j \le 2n$ with $\sum_{i=1}^{k}\lambda_i\scalebox{1.2}{$\chi$}_{G_i}(j)\le\frac{1}{2}$. Denote by $(R_m)_m$ the Rademacher system and consider a natural extension $(\tilde{R}_m)_m$ into $L_\infty[0,1]^{2n}$ such that $\tilde{R}_m(t_1,\ldots,t_{2n})=R_m(t_{j})$. We choose $M\in[\N]^\infty$ such that $(x^*_i({\tilde{R}_m}))_{m\in M}$ converges for all $1\le i\le l$, and applying Lemma \ref{bslem} we may find $m_1,m_2\in M$ such that $d(f,Y)<1/8$, for $f=(\tilde{R}_{m_1}-\tilde{R}_{m_2})/2$. We recall that $(\tilde{R}_m)_m$ is isomorphic to the unit vector basis of $\ell_1$ in the $L_\infty$-norm and hence $\|f\|_\infty=1$. Notice that, for every $m\in M$, we have that $\Delta_i(\tilde{R}_m)=\delta_{ij}\tilde{R}_m$ and $\|\tilde{R}_m\|_2=1$ and since $({R}_m)_m$ are orthogonal, $\|\tilde{R}_{m_1}-\tilde{R}_{m_2}\|_2=(\|\tilde{R}_{m_1}\|^2+\|\tilde{R}_{m_2}\|^2)^\frac{1}{2}$. Hence $\|(A_n-B)f\|\ge1/4$ and so we conclude that $\|(A_n-B)|_Y\|\ge 1/9$, since $d(f,Y)<1/8$.
		\end{proof}

    Assume that $L_\infty[0,1]$ satisfies the UALS with constant $C>0$ and pick $\varepsilon>0$ such that $C\varepsilon<1/9$. The first claim yields an $n\in\N$ such that, for every $f\in L_\infty[0,1]^{2n}$ with $\|f\|\le1$, there exists $B\in W_n$ with $\|(A_n-B)f\|<\varepsilon$. Hence there exist a subspace $Y$ of $L_\infty[0,1]^{2n}$ of finite codimension and a $B\in W_n$ such that $\|(A-B)|_Y\|<C\varepsilon$ and this contradicts our second claim, since $C\varepsilon<1/9$.
	 \end{proof}
	
	 \begin{prop}\label{C(K)}
	 	Let $K$ be an uncountable compact metrizable space. Then the space $C(K)$ fails the UALS property.
	 \end{prop}
	 \begin{proof}
        We set $\Omega=\{-1,1\}^\N$ and Milutin's Theorem \cite{M} yields that the space $C(K)$ is isomorphic to $C(\Omega)$ for every $K$ uncountable compact metrizable. We now fix $n\in\N$, consider a partition of $\N$ into disjoint infinite sets $N_1,\ldots,N_{2n}$ and set $\Omega_i=\{-1,1\}^{N_i}$, for $1\le i\le 2n$. Clearly $C(\Omega)$ is isometric to $C(\prod_{i=1}^{2n}\Omega_i)$.

        In a similar manner as in the previous proposition, for every $1\le i\le 2n$, we define $E_i,\Delta_i:C(\prod_{i=1}^{2n}\Omega_i)\to L_2(\prod_{j\le i}\Omega_j,\otimes_{j\le i}\mu_j)$, where by $\mu_j$ we denote the Haar probability measure on $\Omega_j$. Moreover, for every $G\subset\{1,\ldots,2n\}$, we define the operator $\Delta_G:C(\prod_{i=1}^{2n}\Omega_i)\to (\sum_{i=1}^{2n}L_2(\prod_{j\le i}\Omega_j,\otimes_{j\le i}\mu_j))_\infty$ with $\Delta_G=\sum_{i\in G}\Delta_i$. Observe that $(\sum_{i=1}^{2n}L_2(\prod_{j\le i}\Omega_j,\otimes_{j\le i}\mu_j))_\infty$ is isometric to a subspace of $C(\Omega)$ and hence $\Delta_G:C(\Omega)\to C(\Omega)$. Also set $A_n=\Delta_{\{1,\ldots,2n\}}$ and $W_n=\co\{\Delta_G:\#G=n\}$.

        The family $(\pi_n)_n$ of the projections of $\Omega$ onto its coordinates corresponds to the Rademacher system in $L_\infty[0,1]$. Therefore, assuming that $C(\Omega)$ satisfies the UALS property, we arrive at a contradiction applying the corresponding arguments of Proposition \ref{capital L infty uals}.
	 \end{proof}

 	\subsection{Final Remarks} This last subsection contains some final remarks and open problems concerning the UALS property. We start with the following example suggested by W. B. Johnson which shows that in the definition of the UALS, we cannot expect the uniform approximation to happen on the whole space.
 	
 		 \begin{exmp}\label{ex1}
 		Let $\|\cdot\|$ be a norm on $\R^2$ and for $x,x^*\in\R^2$ define the operator $x^*\otimes x:\R^2 \to \R^2$ with $x^*\otimes x(y)=x^*(y) x$ and set
 		\[
 		W={\textnormal{co}}\{x^*\otimes x:x,x^*\in\R^2\;\textnormal{and}\;\|x\|,\|x^*\|\le1\}.
 		\]
 		Let $y\in {\R^2}$ with $\|y\|\le1$ and $x^*\in\R^2$ with $\|x^*\|=1$ such that $x^*(y)=\|y\|$. Then for $x={y}/{\|y\|}$, we have that $x^*\otimes x\in W$ and $\|x^*\otimes x(y)-I(y)\|=0$, where $I$ denotes the identity operator.
 		
 		For any $B\in W$, there exists a convex combination $\sum_{i=1}^5a_iB_i$ in $W$ such that $B=\sum_{i=1}^5a_iB_i$. Then $a_{i_0}\ge{1}/{5}$ for some $1\le i_0\le 5$, and for $x\in\ker B_{i_0}$ with $\|x\|=1$ we have  that $\|x-\sum_{i=1}^4a_iB_i(x)\|	\ge 1- \sum_{i=1}^4a_i$. Hence $\|I-B\|\ge{1}/{5}$ for all $B\in W$.
 	\end{exmp}
 	
 	This example is extended to every Banach space with dimension greater than two, by the following easy modification.
 	
 	\begin{prop}
 		Let $X$ be a Banach space with $\dim X\ge 2$. There exist $C> 0$ and a convex compact subset $W$ of $\mathcal{L}(X)$ with the property that, for every $x\in B_X$, there exists a $B\in W$ such that $\|x-B(x)\|=0$ whereas $\|I-B\|\ge C$ for all $B\in W$, where $I:X\to X$ denotes the identity operator.
 	\end{prop}
 	\begin{proof}
 		Let $e_1,e_2$ be linearly independent vectors in $X$, denote by $Y$ their linear span and let $Z$ be a subspace of $X$ such that $X=Y\oplus Z$. Set
 		\[
 		W={\textnormal{co}}\{x^*\otimes x|_Y+I|_Z:x,x^*\in Y\;\textnormal{and}\;\|x\|,\|x^*\|\le1\}
 		\]
 		and notice that, using similar arguments to those in the previous example, we obtain the desired result.
 	\end{proof}

	\begin{rem}
        I. Gasparis pointed out that in the case of $c_0$, the UALS can be proved without the use of Kakutani's theorem. This is a consequence of the following fact. Let $T\in\mathcal{L}(c_0)$ and $(x^i_n)_n$, $1\le i\le l$, be normalized block sequences such that for some $\varepsilon>0$, we have that $\|T(x^1_n)\|\ge \varepsilon$ for all $n\in\N$. Then, for every $\delta>0$, there exists a choice $n_1<\ldots<n_l$ such that $\|T(\sum_{i=1}^lx^i_{n_i})\|>\varepsilon-\delta$. Assume now that $T_1,\ldots,T_l\in\mathcal{L}(c_0)$ and $\varepsilon>0$, such that for every $x$ in the unit ball of $c_0$, there exists $1\le i\le l$ such that $\|T_i(x)\|\le\varepsilon$. Then, for every $\varepsilon'>\varepsilon$, there exist $1\le i\le l$ and $n_0\in\N$ such that $\|T_i|_{\spn\{e_n:n\ge n_0\}}\|\le\varepsilon'$. If not, we may choose for each $i=1,\ldots,l$ a normalized block sequence $(x^i_n)_n$ such that $\|T_i(x^i_n)\|\ge\varepsilon'$ for all $n\in\N$. Then applying simultaneously the above observation, for the operators $T_1,\ldots,T_l$, we may select $n_1<\ldots<n_l$ such that $\|T_i(\sum_{i=1}^lx^i_{n_i})\|>\varepsilon$ for all $i=1,\ldots,l$, and this yields a contradiction.
    \end{rem}

	\begin{rem}
There exist Banach spaces which satisfy the UALS while this is not true for all of their subspaces. As already shown, every $L_p[0,1]$ for $1<p<\infty$ and $p\neq2$, fails the UALS whereas item (ii) of Corollary \ref{some duality stuff} yields that it embeds in a space satisfying the property.
	\end{rem}

	Another open problem in a similar context as the above remark is the following. Notice that all spaces in the previous subsection failing the UALS contain a subspace which satisfies the property.

	\begin{pr}
		Does there exist a Banach space that none of its subspaces satisfy the UALS property?		
	\end{pr}

\end{document}